\documentclass[francais,12pt]{smfart}

\usepackage{amsmath,amssymb,amsfonts,footnote}
\usepackage[all]{xy}

\newtheorem{prop0}{Proposition}[section]
\newtheorem{lem0}[prop0]{Lemme}
\newtheorem{rem0}[prop0]{Remarque}
\newtheorem{thm0}[prop0]{Th\'eor\`eme}

\newtheorem{prop}{Proposition}[subsection]
\newtheorem{definit}[prop]{D\'efinition}
\newtheorem{lem}[prop]{Lemme}
\newtheorem{thm}[prop]{Th\'eor\`eme}
\newtheorem{rem}[prop]{Remarque}
\newtheorem{cor}[prop]{Corollaire}

\def\={\buildrel {\rm d\acute ef}\over =}

\DeclareMathOperator{\ind}{ind}
\DeclareMathOperator{\res}{res}
\DeclareMathOperator{\Fr}{Fr}

\DeclareMathOperator{\End}{End}
\DeclareMathOperator{\Ind}{Ind}
\DeclareMathOperator{\Hom}{Hom}
\DeclareMathOperator{\Gal}{Gal}
\DeclareMathOperator{\Sym}{Sym}
\DeclareMathOperator{\val}{val}
\DeclareMathOperator{\GL}{GL}
\DeclareMathOperator{\PGL}{PGL}
\DeclareMathOperator{\WD}{WD}
\DeclareMathOperator{\Aut}{Aut}
\DeclareMathOperator{\Disc}{disc}
\DeclareMathOperator{\ad}{ad}
\DeclareMathOperator{\tr}{tr}

\newcommand{\R}{{\mathbb R}}
\newcommand{\C}{{\mathbb C}}
\newcommand{\oDv}{{\mathcal O}_{D_v}}
\newcommand{\oDw}{{\mathcal O}_{D_w}}
\newcommand{\oM}{{\mathcal O}_M}
\newcommand{\et}{{\mathrm{\acute e t}}}
\newcommand{\CNL}{{\mathcal{C}}_{\mathcal{O}}}
\newcommand{\CF}{D^\triangle}
\newcommand{\CD}{D}
\newcommand{\field}[1]{\mathbb{#1}}
\newcommand{\norm}{\vert\cdot\vert}
\newcommand{\M}{\mathcal{M}}
\newcommand{\Mbar}{\overline{\M}}

\newcommand{\Sbar}{\overline{S}}
\newcommand{\Fv}{F_{v}}
\newcommand{\Fw}{F_{w}}
\newcommand{\pE}{\varpi_E}
\newcommand{\oD}{{\mathcal O}_D}
\newcommand{\oE}{{\mathcal O}_E}
\newcommand{\oL}{{\mathcal O}_L}
\newcommand{\oFv}{{\mathcal O}_{\Fv}}
\newcommand{\oFw}{{\mathcal O}_{\Fw}}

\newcommand{\A}{\field{A}}
\newcommand{\F}{\field{F}}
\newcommand{\T}{\field{T}}
\newcommand{\Z}{{\mathbb Z}}
\newcommand{\Q}{{\mathbb Q}}
\newcommand{\Qp}{\Q_{p}}
\newcommand{\Zp}{\Z_{p}}
\newcommand{\Fp}{{\mathbb F}_{p}}
\newcommand{\Fq}{{\mathbb F}_{q}}
\newcommand{\Qbar}{\overline\Q}
\newcommand{\Qpbar}{\overline\Qp}
\newcommand{\Fvbar}{\overline\Fv}
\newcommand{\Fwbar}{\overline\Fw}
\newcommand{\Zpbar}{\overline\Zp}
\newcommand{\Fpbar}{\overline\Fp}
\newcommand{\rhobar}{\overline\rho}
\newcommand{\G}{{\mathrm{GL}}_2(L)}

\newcommand{\GFv}{{\mathrm{GL}}_2(\Fv)}
\newcommand{\B}{{\mathrm{B}}(L)}
\newcommand{\I}{{\mathrm{I}}(\oL)}
\newcommand{\II}{{\mathrm{I}_1}(\oL)}

\newcommand{\K}{{\mathrm{GL}}_2(\oL)}
\newcommand{\Gbar}{{\mathrm{GL}}_2(\Fpbar)}
\newcommand{\g}{{\Gal}(\Qpbar/L)}
\newcommand{\gF}{{\Gal}(\Qbar/F)}
\newcommand{\gFv}{{{\Gal}(\Fvbar/\Fv)}}
\newcommand{\gFvbar}{{{\Gal}(\overline{k_v}/k_v)}}
\newcommand{\gFw}{{{\Gal}(\Fwbar/\Fw)}}
\newcommand{\gp}{\mathfrak{p}}

\newcommand{\gm}{\mathfrak{m}}
\newcommand{\nr}{\mathrm{nr}}
\newcommand{\univ}{\mathrm{univ}}
\newcommand{\loc}{\mathrm{loc}}
\newcommand{\fl}{\mathrm{fl}}
\newcommand{\ab}{\mathrm{ab}}

\newcommand{\smat}[1]{\left( \begin{smallmatrix} #1 \end{smallmatrix} \right)}

\author[C. Breuil]{Christophe Breuil}
\address{B\^atiment 425\\
C.N.R.S. et Universit\'e Paris-Sud\\
91405 Orsay Cedex\\
France}
\email{christophe.breuil@math.u-psud.fr}

\author[F. Diamond]{Fred Diamond}
\address{Department of Math.\\
King's College London\\
Strand, London WC2R 2LS\\
U.K.}
\email{fred.diamond@kcl.ac.uk}

\thanks{C.\ B.\ remercie pour leur soutien le CNRS, l'universit\'e Paris-Sud, le projet Th\'eHopaD ANR-2011-BS01-005 et le Fields Institute. F.\ D.\ remercie l'I.H.\'E.S. pour son soutien durant la phase de recherche, le C.R.M. de Barcelone, l'I.A.S. et le Fields Institute pour leur hospitalit\'e durant la phase de r\'edaction.\\ \\
${}^\dag$La version finale de cet article sera publi\'ee dans les Annales Scientifiques de l'\'Ecole Normale Sup\'erieure sous le titre {\it Formes modulaires de Hilbert
modulo $p$ et valeurs d'extensions entre caract\`eres galoisiens.}}

\title[Formes modulaires de Hilbert et extensions galoisiennes]{Formes modulaires de Hilbert modulo $p$ et valeurs d'extensions galoisiennes${}^\dag$}

\begin{document} 

\begin{abstract}
Soit $F$ un corps totalement r\'eel, $v$ une place de $F$ non ramifi\'ee divisant $p$ et $\rhobar:\gF\rightarrow \Gbar$ une repr\'esentation continue irr\'eductible dont la restriction $\rhobar\vert_{\gFv}$ est r\'eductible et suffisamment g\'en\'eri\-que. Si $\rhobar$ est modulaire (et satisfait quelques conditions techniques faibles), nous montrons comment retrouver l'extension correspondante entre les deux caract\`eres de $\gFv$ en terme de l'action de $\GFv$ sur la cohomologie modulo $p$. 
\end{abstract}

\begin{altabstract}
Let $F$ be a totally real field, $v$ an unramified place of $F$ dividing $p$ and $\rhobar:\gF\rightarrow \Gbar$ a continuous irreducible representation such that $\rhobar\vert_{\gFv}$ is reducible and sufficiently generic. If $\rhobar$ is modular (and satisfies some weak technical assumptions), we show how to recover the corresponding extension between the two characters of $\gFv$ in terms of the action of $\GFv$ on the cohomology mod $p$.
\end{altabstract}


\maketitle

\setcounter{tocdepth}{2}

\tableofcontents

\section{Introduction}\label{intro}

Soit $F$ un corps totalement r\'eel, $v$ une place de $F$ divisant $p$ et $F_v$ le compl\'et\'e de $F$ en $v$, le $H^1$ \'etale modulo $p$ de tours de courbes de Shimura sur $F$ de niveau en $v$ arbitrairement grand fournit des repr\'esentations lisses de $\GFv$ sur $\Fpbar$ que l'on aimerait comprendre. Si $f$ est une forme de Hilbert propre pour les op\'erateurs de Hecke de repr\'esentation galoisienne modulo $p$ associ\'ee $\rhobar_f$ irr\'eductible, on aimerait par exemple d\'ej\`a savoir d\'ecrire la partie $\rhobar_f$-isotypique de ces repr\'esentations de $\GFv$. Ceci est chose faite lorsque $F_v=\Qp$ (au moins lorsque $F=\Q$ \cite{Em}, mais cela devrait s'\'etendre \`a tout $F$) mais demeure largement myst\'erieux lorsque $F_v\ne \Qp$.

Une des premi\`eres tentatives a \'et\'e de comprendre les repr\'esentations de $\GL_2({\mathcal O}_{F_v})$ apparaissant (\`a multiplicit\'e pr\`es) dans le $\GL_2({\mathcal O}_{F_v})$-socle de cette partie $\rhobar_f$-isotypique : dans \cite{BDJ}, les auteurs donnent une liste conjecturale explicite de ces ``poids de Serre'' lorsque $F_v$ est une extension non ramifi\'ee de $\Qp$, conjecture qui vient d'\^etre compl\`etement d\'emontr\'ee par Gee et Kisin (\cite{GeK}, voir aussi le travail \`a venir de Newton). Cette liste ne d\'epend que de la repr\'esentation locale $\rhobar_f\vert_{\gFv}$ (et m\^eme seulement de sa restriction \`a l'inertie). \`A la suite de \cite{BDJ}, des repr\'esentations lisses admissibles de $\GFv$ sur $\Fpbar$ avec les $\GL_2({\mathcal O}_{F_v})$-socles de \cite{BDJ} ont \'et\'e construites dans \cite{BP} de mani\`ere purement locale en supposant $\rhobar_f\vert_{\gFv}$ suffisamment g\'en\'erique. Des r\'esultats r\'ecents d'Emerton, Gee et Savitt (\cite{EGS}) (faisant suite \`a des r\'esultats partiels dans le cas de vari\'et\'es de Shimura compactes \`a l'infini (cf. \cite{Br}) et des calculs informatiques de Demb\'el\'e (dans le m\^eme cadre, cf. \cite{De})) montrent que la partie $\rhobar_f$-isotypique ci-dessus contient l'une des repr\'esentations de \cite{BP} (lorsque $\rhobar_f\vert_{\gFv}$ est g\'en\'erique). Mais l'une des nouveaut\'e de \cite{BP} est que, d\`es que $F_v\ne \Qp$ (avec $F_v$ non ramifi\'ee), alors il y a {\it \'enorm\'ement} de repr\'esentations lisses admissibles de $\GFv$ sur $\Fpbar$ de socle fix\'e (celui correspondant \`a $\rhobar_f\vert_{\gFv}$). Plus pr\'ecis\'ement, dans les ``diagrammes'' que contiennent les repr\'esentations de \cite{BP} apparaissent de multiples ``param\`etres'' dont le nombre grossit exponentiellement avec le degr\'e $[F_v:\Qp]$ et dont les valeurs sur la partie $\rhobar_f$-isotypique ci-dessus (suppos\'ee contenir l'un de ces diagrammes) sont pour la plupart \`a ce jour myst\'erieuses (par exemple, on ignore si toutes sont locales, i.e. ne d\'ependent que de $\rhobar_f\vert_{\gFv}$).

Du c\^ot\'e repr\'esentations de Galois, il n'y a pas de param\`etres nouveaux qui apparaissent lorsque l'on passe de ${\Gal}(\Qpbar/\Qp)$ \`a $\gFv$, {\it sauf si la repr\'esentation de $\gFv$ est une extension non scind\'ee entre deux ca\-ract\`eres} : on sait en effet que l'espace de ces extensions est g\'en\'eriquement de dimension $[F_v:\Qp]$ sur $\Fpbar$. Une question naturelle se pose alors : est-ce que parmi les nombreux param\`etres qui apparaissent c\^ot\'e $\GFv$ il s'en trouve au moins quelques uns dont les valeurs d\'eterminent compl\`etement l'extension entre les deux caract\`eres de la repr\'esentation de $\gFv$ lorsque celle-ci est (g\'en\'erique) r\'eductible non scind\'ee ? Le but de cet article est de montrer que oui : lorsque $\rhobar_f\vert_{\gFv}$ est g\'en\'erique r\'eductible, nous montrons d'une part que certains des param\`etres de \cite{BP} c\^ot\'e $\GFv$ sont bien d\'efinis (sans aucune conjecture) sur la partie $\rhobar_f$-isotypique du $H^1$ \'etale ci-dessus, et d'autre part que leurs valeurs permettent de retrouver effectivement l'extension pr\'ecise entre les caract\`eres de $\rhobar_f\vert_{\gFv}$. Notons que ce genre de r\'esultat n'a d'\'equivalent ni modulo $\ell\ne p$ (puisque g\'en\'eriquement il n'y a pas d'extension non scind\'ee entre deux caract\`eres de $\gFv$ modulo $\ell\ne p$) ni modulo $p$ pour ${\Gal}(\Qpbar/\Qp)$ (puisque g\'en\'eriquement il y a alors une seule extension non scind\'ee).

\'Enon\c cons maintenant plus pr\'ecis\'ement les r\'esultats principaux de l'article. On suppose donc $F_v$ non ramifi\'ee et on note $k_v$ son corps r\'esiduel. On fixe une repr\'esentation continue, irr\'eductible, totalement impaire $\overline\rho:\gF\rightarrow \Gbar$ et on suppose que $\overline\rho\vert_{\gFv}$ est r\'eductible g\'en\'erique, c'est-\`a-dire de la forme (quitte \`a tordre par un caract\`ere) :
$$\overline\rho\vert_{\gFv}\cong\begin{pmatrix}\nr_v\displaystyle{\prod_{\sigma:k_v\hookrightarrow \Fpbar}}\omega_{\sigma}^{r_{v,\sigma}+1} &*\\0&\nr'_v\end{pmatrix}$$
o\`u $r_{v,\sigma}\in \{0,\cdots,p-3\}$ (non tous \'egaux \`a $0$ ou \`a $p-3$), $\omega_{\sigma}$ est le caract\`ere fondamental de Serre associ\'e au plongement $\sigma$ et $\nr_v, \nr'_v$ des caract\`eres non ramifi\'es. On peut alors d\'ecrire explicitement $\rhobar\vert_{\gFv}$ par le truchement de son module de Fontaine-Laffaille contravariant :
$$\prod_{\sigma:k_v\hookrightarrow \Fpbar}\big(M^{\sigma}=\Fpbar e^{\sigma}\oplus \Fpbar f^{\sigma},{\rm Fil}^{r_{v,\sigma}+1}M^{\sigma}=\Fpbar f^{\sigma}\big)$$
avec $\ \left\{\begin{array}{lll}
\varphi(e^{\sigma})&=&\alpha_{v,\sigma}e^{\sigma\circ\varphi^{-1}}\\
\varphi_{r_{v,\sigma}+1}(f^{\sigma})&=&\beta_{v,\sigma}(f^{\sigma\circ\varphi^{-1}}+x_{v,\sigma}e^{\sigma\circ\varphi^{-1}})
\end{array}\right.$ o\`u $\alpha_{v,\sigma},\beta_{v,\sigma}\in \Fpbar^{\times}$ et $x_{v,\sigma}\in \Fpbar$. Maintenant, on suppose $\rhobar$ modulaire, c'est-\`a-dire :
$$\pi_D(\rhobar)\=\Hom_{\Fpbar[\gF]} \Big(\rhobar,\varinjlim_U H^1_{\et}(X_{U,\Qbar},\Fpbar)(1)\Big)\ne 0$$
o\`u $(X_U)_U$ est une tour de courbes de Shimura sur $F$ associ\'ee \`a une alg\`ebre de quaternions $D$ sur $F$ d\'eploy\'ee en une seule des places infinies de $F$ ainsi qu'aux places divisant $p$ ($U$ parcourant les sous-groupes ouverts compacts de $(D\otimes_\Z\widehat \Z)^\times$). Sous quelques hypoth\`eses techniques sur $\rhobar$ (que nous n'avons pas cherch\'e \`a optimiser, cf. d\'ebut du \S\ \ref{facteur}), on peut utiliser l'action de $(D\otimes_\Z\widehat \Z)^\times$ sur $\pi_D(\rhobar)$ aux places diff\'erentes de $v$ pour d\'efinir un ``facteur local'' $\pi_{D,v}(\rhobar)$ en $v$ qui est une repr\'esentation lisse admissible de $(D\otimes_FF_v)^\times\cong \GFv$ sur $\Fpbar$ (mais dont on ignore si elle ne d\'epend que de $\rhobar|_{\gFv}$). Notons que l'on ne dispose pas ici {\it a priori} d'une factorisation de $\pi_D(\rhobar)$ ``\`a la Flath'' (bien que cela soit conjectur\'e, cf. \cite[Conj.4.7]{BDJ} et le \S\ \ref{globalprelim}), d'o\`u la n\'ecessit\'e de d\'efinir soigneusement ce facteur local $\pi_{D,v}(\rhobar)$.

Si $J$ est un sous-ensemble des plongements de $k_v$ dans $\Fpbar$, on d\'efinit la ``fronti\`ere de $J$'' $F(J)$ comme l'ensemble des plongements $\sigma:k_v\hookrightarrow \Fpbar$ tels que ou bien $\sigma\in J$ et $\sigma\circ\varphi^{-1}\notin J$, ou bien $\sigma\notin J$ et $\sigma\circ\varphi^{-1}\in J$ o\`u $\varphi$ est le Frobenius usuel $x\mapsto x^p$ sur $k_v$. Notons ${\rm I}({\mathcal O}_{F_v})$ le sous-groupe de $\GL_2({\mathcal O}_{F_v})$ des matrices triangulaires sup\'erieures modulo $p$ et ${\rm I}_1({\mathcal O}_{F_v})\subset {\rm I}({\mathcal O}_{F_v})$ celui des matrices unipotentes sup\'erieures modulo $p$. Le premier th\'eor\`eme associe \`a $\pi_{D,v}(\rhobar)$ certains invariants $x(J)$ de $\Fpbar^\times$ qui apparaissent naturellement dans les diagrammes de \cite[\S\ 13]{BP} (bien qu'ils n'y soient pas explicit\'es).

\begin{thm0}\label{invariant}
Soit $J$ tel que $F(J)\cap \{\sigma,x_{v,\sigma}=0\}=\emptyset$. Il existe \`a scalaire pr\`es un unique vecteur $v\in \pi_{D,v}(\rhobar)^{{\rm I}_1({\mathcal O}_{F_v})}$ non nul sur lequel ${\rm I}({\mathcal O}_{F_v})$ agisse par le caract\`ere ${\prod_{\sigma\in J}}\sigma^{p-1}{\prod_{\sigma\notin J}}\sigma^{r_{v,\sigma}}\otimes {\prod_{\sigma\in J}}\sigma^{r_{v,\sigma}}{\prod_{\sigma\notin J}}\sigma^{p-1}$ de ${\rm I}({\mathcal O}_{F_v})/{\rm I}_1({\mathcal O}_{F_v})\cong k_v^{\times}\times k_v^{\times}$ et un unique \'el\'ement $x(J)\in \Fpbar^{\times}$ tel que l'on ait l'\'egalit\'e dans $\pi_{D,v}(\rhobar)$ :
$$\sum_{s\in k_v}\!\bigg(\prod_{\sigma\in J}\sigma(s)^{p-1-r_{v,\sigma}}\!\bigg)\!\begin{pmatrix}[s] & 1\\ 1 & 0\end{pmatrix}\begin{pmatrix}0 & 1\\ p & 0\end{pmatrix}\!v = x(J)\!\sum_{s\in k_v}\!\bigg(\prod_{\sigma\notin J}\sigma(s)^{p-1-r_{v,\sigma}}\!\bigg)\!\begin{pmatrix}[s] & 1\\ 1 & 0\end{pmatrix}\!v$$
o\`u $[s]$ est le repr\'esentant multiplicatif de $s$ dans ${\mathcal O}_{F_v}$.
\end{thm0}

Lorsque $\{\sigma,x_{v,\sigma}=0\}= \emptyset$, les invariants $x(J)$ ci-dessus sont les {\it seuls} invariants de \cite{BP}, mais il en appara\^\i t bien d'autres lorsque $\{\sigma,x_{v,\sigma}=0\}\ne \emptyset$. Les r\'esultats de \cite{BP} montrent par ailleurs que l'on peut construire des repr\'esentations lisses admissibles de $\GFv$ sur $\Fpbar$ avec le $\GL_2({\mathcal O}_{F_v})$-socle correspondant \`a $\overline\rho\vert_{\gFv}$ et des valeurs {\it presque quelconques} de ces invariants $x(J)$ (voir \S\ \ref{spec}), de sorte que les valeurs prises par les scalaires $x(J)$ du th\'eor\`eme \ref{invariant} ne peuvent pas du tout \^etre pr\'edites {\it a priori}. Le deuxi\`eme th\'eor\`eme donne ces valeurs pr\'ecises.

\begin{thm0}\label{valeur}
Soit $J$ tel que $F(J)\cap \{\sigma,x_{v,\sigma}=0\}=\emptyset$, on a :
$$x(J)=-\Big(\prod_{\sigma\in J}\alpha_{v,\sigma}\prod_{\sigma\notin J}\beta_{v,\sigma}\Big)^{-1}\frac{\displaystyle{\prod_{\stackrel{\sigma\in J}{ \sigma\circ\varphi^{-1}\notin J}}}\!\!x_{v,\sigma}(r_{v,\sigma}+1)}{\displaystyle{\prod_{\stackrel{\sigma\notin J}{ \sigma\circ\varphi^{-1}\in J}}}\!\!x_{v,\sigma}(r_{v,\sigma}+1)}\in \Fpbar^{\times}.$$
\end{thm0}

En particulier, on voit que ces valeurs sont {\it locales}, i.e. ne d\'ependent que de $\rhobar\vert_{\gFv}$, ce qui n'\'etait pas \'evident {\it a priori}. Notons que les scalaires $\alpha_{v,\sigma},\beta_{v,\sigma}$ et $x_{v,\sigma}$ ne sont pas d\'efinis de mani\`ere unique (comme le lecteur peut imm\'ediatement le voir en faisant un changement de base sur $M^{\sigma}$ qui respecte les structures), mais on peut v\'erifier directement que les scalaires $x(J)$ du th\'eor\`eme \ref{valeur} ne d\'ependent pas des choix faits. En parti\-culier, on peut supposer tous les $\alpha_{v,\sigma}$ (resp. $\beta_{v,\sigma}$) \'egaux \`a $1$ sauf un, donn\'e alors par $\nr'_v(p^{-1})$ (resp. $\nr_v(p^{-1})$) et, quitte \`a faire un changement de base, on peut \'egalement supposer que l'un des $x_{v,\sigma}$ vaut $1$ (du moins s'il existe un $x_{v,\sigma}$ non nul, mais dans le cas contraire $\overline\rho\vert_{\gFv}$ est scind\'ee et les invariants $x(J)$ donn\'es par les th\'eor\`emes ci-dessus se limitent \`a $\nr_v(p)$ et $\nr'_v(p)$). Le lecteur pourra alors v\'erifier, en prenant par exemple des $J$ de la forme $\{\sigma,\sigma\circ \varphi,\cdots,\sigma\circ\varphi^j\}$ pour $\sigma$ et $j$ convenables, que l'on retrouve facilement les valeurs de tous les autres $x_{v,\sigma}$ non nuls \`a partir des valeurs des $x(J)$ du th\'eor\`eme \ref{valeur} et des $r_{v,\sigma}$ (cf. la remarque \ref{comment}(iii)). 

Le travail r\'ecent d'Emerton, Gee et Savitt (\cite{EGS}, voir aussi \cite{Br}) montre que, sous nos hypoth\`eses, la $\GL_2(F_v)$-repr\'esentation $\pi_{D,v}(\rhobar)$ contient un des diagrammes construits dans \cite{BP}, notons le $D_v(\rhobar)$. Une conjecture naturelle support\'ee par le Th\'eor\`eme \ref{valeur} ci-dessus est que $D_v(\rhobar)$ est enti\`erement {\em local}, c'est-\`a-dire ne d\'epend que de la restriction de $\rhobar$ \`a $\gFv$ (c'est par exemple le cas si $\{\sigma,x_{v,\sigma}=0\}= \emptyset$). Si l'on est optimiste on peut m\^eme penser que toute la $\GL_2(F_v)$-repr\'esentation $\pi_{D,v}(\rhobar)$ pourrait elle-m\^eme \^etre locale. Bien que nous esp\'erons que ces \'enonc\'es sont vrais, nous avons n\'eanmoins choisi de ne pas les pr\'esenter sous forme de conjectures. Une raison est que, en dehors des r\'esultats de cet article et de l'article de Hu \cite{Hu}, nous ne savons pour l'instant rien de plus sur les diagrammes $D_v(\rhobar)$, qui demeurent donc en g\'en\'eral largement myst\'erieux (sans parler des $\GL_2(F_v)$-repr\'esentations $\pi_{D,v}(\rhobar)$). 

Disons quelques mots sur les preuves des th\'eor\`emes \ref{invariant} et \ref{valeur}. Le coeur du th\'eor\`eme \ref{invariant} est de montrer que le poids de Serre $\otimes_{\sigma:k_v\hookrightarrow \Fpbar}(\Sym^{r_{v,\sigma}}\Fpbar^2)^{\sigma}$ (voir \S\ \ref{prel} pour les notations) appara\^\i t {\it avec multiplicit\'e un} dans le $\GL_2({\mathcal O}_{F_v})$-socle de $\pi_{D,v}(\rhobar)$ (le fait qu'il apparaisse \'etait essentiellement d\'ej\`a connu). Cela se d\'emontre en utilisant les techniques de multiplicit\'e un issues de la m\'ethode de Taylor-Wiles comme inaugur\'e par Fujiwara (\cite{Fu}) et l'un d'entre nous (\cite{Di1}). Un deuxi\`eme ingr\'edient essentiel est que la repr\'esentation de $\GL_2({\mathcal O}_{F_v})$ sur $\Fpbar$ :
$$\ind_{{\rm I}({\mathcal O}_{F_v})}^{\GL_2({\mathcal O}_{F_v})}\Big({\prod_{\sigma\in J}}\sigma^{p-1}{\prod_{\sigma\notin J}}\sigma^{r_{v,\sigma}}\otimes {\prod_{\sigma\in J}}\sigma^{r_{v,\sigma}}{\prod_{\sigma\notin J}}\sigma^{p-1}\Big)$$
n'a qu'un seul de ses constituants qui appara\^\i t dans ce $\GL_2({\mathcal O}_{F_v})$-socle : \`a savoir le poids de Serre $\otimes_{\sigma}(\Sym^{r_{v,\sigma}}\Fpbar^2)^{\sigma}$ ci-dessus. Cela se d\'eduit par exemple directement de \cite{GeK} et d'un calcul facile (mais peut aussi se d\'emontrer de mani\`ere plus \'el\'ementaire sans utiliser \cite{GeK}). Une fois ces deux ingr\'edients disponibles, l'existence de $x(J)$ se ram\`ene essentiellement \`a de la th\'eorie des repr\'esentations (cf. proposition \ref{xj}). 

Le coeur du th\'eor\`eme \ref{valeur} est un calcul local c\^ot\'e $\gFv$ et un autre c\^ot\'e $\GFv$. C\^ot\'e $\gFv$, on calcule la r\'eduction modulo $p$ des valeurs propres du Frobenius (multipli\'ees par les bonnes puissances de $p$) sur les modules de Dieudonn\'e des repr\'esentations potentiellement Barsotti-Tate de $\gFv$ de donn\'ee de descente $[{\prod_{\sigma\in J}}\omega_{\sigma}^{p-1}{\prod_{\sigma\notin J}}\omega_{\sigma}^{r_{v,\sigma}}]\oplus [{\prod_{\sigma\in J}}\omega_{\sigma}^{r_{v,\sigma}}{\prod_{\sigma\notin J}}\omega_{\sigma}^{p-1}]$ relevant $\rhobar\vert_{\gFv}$ ($[\cdot ]$ est le repr\'esentant de Teichm\"uller). C\^ot\'e $\GFv$, on calcule la r\'eduction mo\-dulo $p$ d'un scalaire $\widehat x(J)\in \Zpbar$ d\'efini essentiellement comme le scalaire $x(J)$ du th\'eor\`eme \ref{invariant} mais \`a l'int\'erieur de la s\'erie principale lisse usuelle en caract\'eristique $0$ associ\'ee (par la correspondance de Langlands locale classique) \`a la repr\'esentation de Weil-Deligne d'une repr\'esentation potentiellement Barsotti-Tate comme ci-dessus, au lieu de la repr\'esentation $\pi_{D,v}(\rhobar)$. Comme ce calcul fait intervenir les valeurs propres du Frobenius via la compa\-tibilit\'e local-global classique (\cite{Sai}), en mettant bout \`a bout les deux calculs, on obtient la formule du th\'eor\`eme \ref{valeur}. Signalons que ces calculs locaux ont \'et\'e \'etendus par Hu (\cite{Hu}) pour d\'eterminer les valeurs de quelques autres param\`etres de \cite{BP} analogues aux param\`etres $x(J)$.

Au passage, on donne aussi dans le cours du texte un r\'esultat annexe qui a un int\'er\^et ind\'ependamment des deux th\'eor\`emes ci-dessus. M\^eme lorsque $D$ n'est pas d\'eploy\'ee en $v$, on peut d\'efinir un facteur local $\pi_{D,v}(\rhobar)$. On montre que cette repr\'esentation de $(D\otimes_FF_v)^\times$ est alors toujours {\it de longueur infinie} (m\^eme si $F_v=\Qp$, cf. corollaire \ref{extra+}). La preuve utilise l'existence de repr\'esentations irr\'eductibles en caract\'eristique $0$ de dimension (finie) arbitrairement grande de $(D\otimes_FF_v)^\times$, le calcul de la r\'eduction modulo $p$ des types de Bushnell-Kutzko pr\'esent\'e en appendice et des arguments de congruence. \'Evidemment, elle ne marche plus si $D$ est d\'eploy\'ee en $v$.

Passons maintenant bri\`evement en revue les diff\'erentes sections de l'article. La premi\`ere partie rassemble tous les calculs locaux et la seconde tous les r\'esultats globaux (dont les deux th\'eor\`emes \ref{invariant} et \ref{valeur}). Apr\`es des pr\'eliminaires (\S\ \ref{prel}), on calcule explicitement aux \S\S\ \ref{fl1} et \ref{fl2} le module de Fontaine-Laffaille contravariant de la r\'eduction modulo $p$ de certaines repr\'esentations potentiellement Barsotti-Tate de donn\'ee de descente $[{\prod_{\sigma\in J}}\omega_{\sigma}^{p-1}{\prod_{\sigma\notin J}}\omega_{\sigma}^{r_{\sigma}}]\oplus [{\prod_{\sigma\in J}}\omega_{\sigma}^{r_{\sigma}}{\prod_{\sigma\notin J}}\omega_{\sigma}^{p-1}]$. Au \S\ \ref{bt}, on en d\'eduit le calcul de la r\'eduction modulo $p$ des valeurs propres de Frobenius mentionn\'e ci-avant (th\'eor\`eme \ref{thm_reduction}). Au \S\ \ref{jacobi}, on calcule la r\'eduction modulo $p$ des invariants $\widehat x(J)$ dans les s\'eries principales mod\'er\'ement ramifi\'ees provenant des repr\'esentations potentiellement Barsotti-Tate des \S\S\ \ref{fl1} et \ref{fl2} (th\'eor\`eme \ref{main}). Au \S\ \ref{spec}, on donne des conditions suffisantes pour pouvoir d\'efinir (abstraitement) des invariants $x(J)$ comme dans le th\'eor\`eme \ref{invariant}, on rappelle des r\'esultats de \cite{BP} et on donne le r\'esultat local sous sa forme finale (corollaire \ref{mainlocal}). Au \S\ \ref{globalprelim} on introduit le cadre global et la repr\'esentation $\pi_D(\rhobar)$ ci-dessus. Au \S\ \ref{lifts0}, on rappelle (et g\'en\'eralise tr\`es l\'eg\`erement) des r\'esultats de Barnet-Lamb, Gee et Geraghty (\cite{Ge}, \cite{BLGG}, \cite{BLGG2}) sur l'existence de repr\'esentations galoisiennes globales modulaires avec certaines conditions locales fix\'ees que l'on utilise pour d\'eterminer quand $\pi_D(\rhobar)$ est non nul (corollaire \ref{nonzero}). Au \S\ \ref{facteur}, on d\'efinit le facteur local $\pi_{D,v}(\rhobar)$ pr\'ec\'edent. Au \S\ \ref{deformation}, on introduit les anneaux de d\'eformations de repr\'esenta\-tions galoisiennes locales et globales qui serviront \`a la preuve du th\'eor\`eme de multiplicit\'e un. Aux \S\S\ \ref{unI} et \ref{unII}, on introduit les syst\`emes de Taylor-Wiles dont on a besoin puis on les utilise pour montrer qu'un certain module est libre de rang $2$ sur l'alg\`ebre de Hecke (th\'eor\`eme \ref{thm:free}). Au \S~\ref{resprinc}, on en d\'eduit le th\'eor\`eme de multiplicit\'e un (th\'eor\`eme \ref{thm:multone}), puis on l'utilise (ainsi que tout ce qui pr\'ec\`ede) pour d\'emontrer les th\'eor\`emes \ref{invariant} et \ref{valeur}. Enfin, en appendice, on calcule la semi-simplification modulo $p$ des types de Bushnell-Kutzko (ou $K$-types) pour $\GL_2$ ou pour les unit\'es d'une alg\`ebre de quaternions.

On ach\`eve cette introduction avec quelques notations g\'en\'erales. Dans tout le texte, $E$ est une extension finie de $\Qp$ qui d\'esigne le corps des coefficients, $\oE$ est son anneau d'entiers et $k_E$ son corps r\'esiduel. On suppose toujours $E$ ``suffisamment grand'' (cela sera explicit\'e dans le corps du texte). Tout ce qui est r\'eduit modulo une uniformisante $\pE$ de $\oE$ est surlign\'e : par exemple, si $x\in \oE$, $\overline x$ est sa r\'eduction dans $k_E$, si $M$ est un $\oE$-module, $\overline M$ d\'esigne $M/\pE$, etc.

On note $\varepsilon$ le caract\`ere cyclotomique $p$-adique usuel et $\omega$ (plut\^ot que $\overline\varepsilon$) sa r\'eduction modulo $p$. On note $[x]$ le repr\'esentant multiplicatif d'un \'el\'ement $x$ d'une extension finie de $\Fp$. On normalise l'application de r\'eciprocit\'e de la th\'eorie du corps de classes local en envoyant les Frobenius g\'eom\'etriques sur les uniformisantes. 

Si $L$ est une extension finie de $\Qp$ d'anneau d'entiers $\oL$, on note $\B\subset \G$ le sous-groupe des matrices triangulaires sup\'erieures et $\I\subset \K$ (resp. $\II\subset \I$) le sous-groupe des matrices triangulaires (resp. unipotentes) sup\'erieures modulo une uniformisante de $\oL$. 

Les autres notations seront introduites au fur et \`a mesure des besoins.

Les auteurs remercient Toby Gee pour leur avoir signal\'e les r\'esultats r\'ecents de \cite{BLGG2} et \cite{GeK} qui leur ont permis d'all\'eger les hypoth\`eses techniques dans les \'enonc\'es globaux, ainsi que Colin Bushnell et Guy Henniart pour leur avoir signal\'e des r\'ef\'erences utiles pour l'appendice. Ils remercient \'egalement deux rapporteurs anonymes pour leur lecture minutieuse et leurs remarques pertinentes.

\section{R\'esultats locaux}\label{local}

\subsection{Quelques pr\'eliminaires}\label{prel}

Cette partie contient divers rappels, notations, d\'efini\-tions et r\'esultats \'el\'ementaires qui seront utilis\'es dans la suite.

On d\'esigne par $L$ une extension finie de $\Qp$ {\it non ramifi\'ee} de degr\'e $f\geq 1$ et d'anneau d'entiers $\oL$ et on suppose que le corps des coefficients $E$ est tel que $\vert\Hom(L,E)\vert =f$. On pose $q\=p^f$ et on note $\varphi$ le Frobenius $x\mapsto x^p$ sur $\Fq$. On note ${\mathcal S}$ l'ensemble des plongements de $\Fq$ dans $k_E$ (qui s'identifie \`a l'ensem\-ble des plongements de $L$ dans $E$ puisque $L$ est non ramifi\'ee). Si $\sigma\in {\mathcal S}$, on note $\omega_{\sigma}$ le caract\`ere (fondamental) induit sur $\Gal(\Qpbar/L)$ par $\sigma$ compos\'e avec :
$$\g\twoheadrightarrow \Gal(L[\sqrt[p^f-1]{-p}]/L) \buildrel\sim\over\rightarrow \Fq^{\times},\ g \mapsto \overline{\frac{g(\sqrt[p^f-1]{-p})}{\sqrt[p^f-1]{-p}}}.$$ 
Via le corps de classes local on a $\omega_{\sigma}(p)=1$ pour tout $\sigma\in {\mathcal S}$. On note $[\sigma]:\Fq\hookrightarrow \oE$ le caract\`ere multiplicatif tel que $[\sigma](x)\=[\sigma(x)]$ pour $x\in \Fq$ et $\nr(y)$ le caract\`ere non ramifi\'e de $L^\times$ ou $\g$ envoyant $p$ sur $y$. 

On fixe une repr\'esentation lin\'eaire continue $\rhobar:\g\rightarrow \GL_2(k_E)$ r\'eductible et g\'en\'erique au sens de \cite[Def.11.7]{BP}, c'est-\`a-dire de la forme :
$$\rhobar\cong\begin{pmatrix}\big(\nr(\lambda)\displaystyle{\prod_{\sigma\in {\mathcal S}}}\omega_{\sigma}^{r_{\sigma}}\big)\omega &*\\0&\nr(\mu)\end{pmatrix}\otimes \theta$$
o\`u $\lambda,\mu \in k_E^{\times}$, $r_{\sigma}\in \{0,\cdots,p-3\}$ avec $(r_{\sigma})_{\sigma\in {\mathcal S}}\ne (0,\cdots,0),(p-3,\cdots,p-3)$ (ce qui implique donc $p>3$) et $\theta:\g\rightarrow k_E^\times$. Quitte \`a modifier $\lambda$ et $\mu$, on ne restreint pas la g\'en\'eralit\'e en supposant $\theta(p)=1$. La repr\'esentation $\rhobar\otimes \theta^{-1}$ est alors toujours dans la cat\'egorie de Fontaine-Laffaille, c'est-\`a-dire s'\'ecrit : 
$$\overline\rho\otimes \theta^{-1}={\rm Hom}_{{\rm Fil}^{\cdot},\varphi_{\cdot}}(M,A_{\rm cris}\otimes_{\Zp}\Fp)$$
o\`u $M$ est un $\varphi$-module filtr\'e de Fontaine-Laffaille (\cite{FL}) de la forme $M=\prod_{\sigma\in {\mathcal S}}M^{\sigma}$, ${\rm Fil}^{i}M=\prod_{\sigma\in {\mathcal S}}{\rm Fil}^{i}M^{\sigma}$ avec pour $\sigma\in {\mathcal S}$ :
\begin{equation*}
\left\{\begin{array}{llcl}
M^{\sigma}&=&k_E e^{\sigma}\oplus k_Ef^{\sigma}\\
{\rm Fil}^iM^{\sigma}&=&M^{\sigma}&i\leq 0\\
{\rm Fil}^{i}M^{\sigma}&=& k_E f^{\sigma}&1\leq i\leq r_{\sigma}+1\\
{\rm Fil}^{i}M^{\sigma}&=&0&r_{\sigma}+2\leq i
\end{array}\right.
\end{equation*}
\begin{equation}\label{norm}
\!\!\!\!\!\!\!\!\!\!\left\{\begin{array}{lll}
\varphi(e^{\sigma})&=&\alpha_{\sigma}e^{\sigma\circ\varphi^{-1}}\\
\varphi_{r_{\sigma}+1}(f^{\sigma})&=&\beta_{\sigma}(f^{\sigma\circ\varphi^{-1}}+x_{\sigma}e^{\sigma\circ\varphi^{-1}})
\end{array}\right.
\end{equation}
o\`u $\alpha_{\sigma},\beta_{\sigma}\in k_E^{\times}$ et $x_{\sigma}\in k_E$. 

On pose :
\begin{equation}\label{zrho}
Z(\rhobar)\=\{\sigma\in {\mathcal S},\ x_{\sigma}=0\}
\end{equation}
et on note que $Z(\rhobar)=\mathcal S$ si et seulement si $\rhobar$ est scind\'ee. 

Les scalaires $\alpha_{\sigma},\beta_{\sigma},x_{\sigma}$ ne sont pas bien d\'efinis mais on voit que l'on doit avoir par exemple $(\prod_{\sigma}\beta_{\sigma})^{-1}=\lambda$ et $(\prod_{\sigma}\alpha_{\sigma})^{-1}=\mu$, de sorte que les deux scalaires $(\prod_{\sigma}\beta_{\sigma})^{-1}$ et $(\prod_{\sigma}\alpha_{\sigma})^{-1}$ ne d\'ependent que de $\rhobar$ (i.e. ne d\'ependent ni de $\theta$ ni de l'\'ecriture (\ref{norm})). On voit aussi que l'ensemble de plongements $Z(\rhobar)$ ne d\'epend que de $\rhobar$. Plus g\'en\'eralement, on a le lemme \'el\'ementaire suivant, dont la preuve est laiss\'ee au lecteur comme (plaisant) exercice.

\begin{lem}\label{joli}
Pour tout $J\subseteq {\mathcal S}$ tel que $Z(\rhobar)\cap\{\sigma\in {\mathcal S},\sigma\notin J, \sigma\circ\varphi^{-1}\in J\}=\emptyset$, les scalaires :
$$\Big(\prod_{\sigma\in J}\alpha_{\sigma}\prod_{\sigma\notin J}\beta_{\sigma}\Big)^{-1}\frac{\displaystyle{\prod_{\stackrel{\sigma\in J}{ \sigma\circ\varphi^{-1}\notin J}}}x_{\sigma}}{\displaystyle{\prod_{\stackrel{\sigma\notin J}{\sigma\circ\varphi^{-1}\in J}}}x_{\sigma}}\ \ \in \ \ k_E$$
ne d\'ependent que de $\rhobar$.
\end{lem}

\begin{rem}\label{comment}
{\rm (i) Noter les deux cas extr\^emes $J=\emptyset$ et $J=\mathcal S$ qui cor\-respondent aux deux scalaires ci-dessus (lorsque $Z(\rhobar)=\mathcal S$, i.e. lorsque $\rhobar$ est scind\'ee, ce sont d'ailleurs les deux seuls cas du lemme).\\
(ii) Lorsque $\rhobar$ est scind\'ee (ce qui n'est pas le cas important de cet article), il y a deux possibilit\'es pour $(\lambda,\mu,(r_{\sigma})_{\sigma\in \mathcal S}, \theta)$, l'autre choix \'etant $(\mu,\lambda,(p-3-r_{\sigma})_{\sigma\in \mathcal S}, \theta\prod_{\sigma\in {\mathcal S}}\omega_{\sigma}^{r_{\sigma}+1})$. Nous choisissons dans ce cas une des deux possibilit\'es.\\
(iii) Pour $J$ fix\'e, le scalaire correspondant du lemme \ref{joli} est non nul si et seulement si l'on a de plus $Z(\rhobar)\cap\{\sigma\in {\mathcal S},\sigma\in J, \sigma\circ\varphi^{-1}\notin J\}=\emptyset$. Il est facile de voir que l'on peut retrouver $\rhobar$ (\`a torsion pr\`es par un caract\`ere de la forme $\prod_{\sigma\in {\mathcal S}}\omega_{\sigma}^{s_{\sigma}}$) \`a partir de la connaissance des $r_{\sigma}$, de $Z(\rhobar)$ et des valeurs de tous les scalaires du lemme \ref{joli} (voir la discussion de l'introduction apr\`es le th\'eor\`eme \ref{valeur}).}
\end{rem}

On \ rappelle \ qu'un poids de Serre est une repr\'esentation irr\'eductible de $\K$, ou de mani\`ere \'equivalente de $\GL_2(\Fq)$, sur $k_E$. \`A torsion pr\`es par un caract\`ere, il est de la forme $\otimes_{\sigma\in {\mathcal S}}(\Sym^{s_{\sigma}}k_E^2)^{\sigma}$ o\`u $s_{\sigma}\in \{0,\cdots,p-1\}$ et o\`u $\GL_2(\Fq)$ agit sur $(\Sym^{s_{\sigma}}k_E^2)^{\sigma}$ via $\sigma:\Fq\hookrightarrow k_E$ et l'action sur la base canonique de $k_E^2$. \`A toute repr\'esentation $\rhobar$ de dimension $2$ est associ\'e dans \cite{BDJ} un ensemble de poids de Serre que l'on note ${\mathcal D}(\rhobar)$. Dans le cas ci-dessus, par \cite[Prop.A.3]{Br} (voir aussi \cite{CD} pour un r\'esultat {\it a posteriori} \'equivalent), c'est l'ensemble des poids de Serre :
$$\big(\otimes_{\sigma\in {\mathcal S}}(\Sym^{s_{\sigma}}k_E^2)^{\sigma}\big)\otimes \Big(\prod_{\sigma\circ\varphi\in I}\sigma\circ{\det}^{-(s_{\sigma}+1)}\Big)\otimes \theta\circ \det$$
pour lesquels il existe $I\subseteq Z(\rhobar)$ tel que :
$$\rhobar\cong\begin{pmatrix}\nr(\lambda)\displaystyle{\prod_{\sigma\circ \varphi\notin I}}\omega_{\sigma}^{s_{\sigma}+1}
&*\\0&\nr(\mu)\displaystyle{\prod_{\sigma\circ \varphi\in I}}\omega_{\sigma}^{s_{\sigma}+1}\end{pmatrix}\otimes \theta\displaystyle{\prod_{\sigma\circ\varphi\in I}}\omega_{\sigma}^{-(s_{\sigma}+1)}.$$
Avec nos hypoth\`eses sur les $r_{\sigma}$, on a $\vert {\mathcal D}(\rhobar)\vert = 2^{\vert Z(\rhobar)\vert}$. De plus on a toujours $\big(\otimes_{\sigma\in {\mathcal S}}(\Sym^{r_{\sigma}}k_E^2\big)^{\sigma})\otimes \theta\circ\det\in {\mathcal D}(\rhobar)$.

Pour $J\subseteq {\mathcal S}$, on d\'efinit les caract\`eres multiplicatifs de $\Fq^{\times}$ \`a valeurs dans $\oE^{\times}$ :
\begin{equation}\label{type}
\left\{\begin{array}{lll}
\eta(J)&\=&[\theta\vert_{[\Fq^{\times}]}]\displaystyle{\prod_{\sigma\in J}}[\sigma]^{r_{\sigma}}\displaystyle{\prod_{\sigma\notin J}}[\sigma]^{p-1}\\
\eta'(J)&\=&[\theta\vert_{[\Fq^{\times}]}]\displaystyle{\prod_{\sigma\in J}}[\sigma]^{p-1}\displaystyle{\prod_{\sigma\notin J}}[\sigma]^{r_{\sigma}}.
\end{array}\right.
\end{equation}
Les hypoth\`eses sur $r_{\sigma}$ entra\^\i nent $\eta(J)\ne \eta'(J)$. Notons que $\eta({\mathcal S}\backslash J)=\eta'(J)$ pour tout $J$ et que $\overline\eta'(\emptyset)=\big((\prod_{\sigma\in {\mathcal S}}\omega_{\sigma}^{r_{\sigma}})\theta\big)\vert_{[\Fq^{\times}]}=\big((\nr(\lambda)\prod_{\sigma\in {\mathcal S}}\omega_{\sigma}^{r_{\sigma}})\theta\big)\vert_{[\Fq^{\times}]}$ et $\overline\eta(\emptyset)=\theta\vert_{[\Fq^{\times}]}=(\nr(\mu)\theta)\vert_{[\Fq^{\times}]}$.

Nous aurons besoin du lemme qui suit (un calcul \'el\'ementaire laiss\'e au lecteur).

\begin{lem}\label{csigma}
On a $\eta(J)=\eta'(J)\prod_{\sigma\in {\mathcal S}}[\sigma]^{c_{\sigma}}$ o\`u :
$$\begin{array}{ccccccc}
c_{\sigma}&=&p-2-r_{\sigma}&\ {si}&\ \sigma\notin J&{et}&\sigma\circ\varphi^{-1}\in J\\
c_{\sigma}&=&p-1-r_{\sigma}&\ {si}&\ \sigma\notin J&{et}&\sigma\circ\varphi^{-1}\notin J\\
c_{\sigma}&=&r_{\sigma}+1&\ {si}&\ \sigma\in J&{et}&\sigma\circ\varphi^{-1}\notin J\\
c_{\sigma}&=&r_{\sigma}&\ {si}&\ \sigma\in J&{et}&\sigma\circ\varphi^{-1}\in J.
\end{array}$$
\end{lem}

On note $\eta'(J)\otimes\eta(J): \I\rightarrow \oE^{\times}$ le caract\`ere :
\begin{eqnarray}\label{carJ}
\begin{pmatrix} a & b\\ pc & d\end{pmatrix}\mapsto \eta'(J)(\overline a)\eta(J)(\overline d)
\end{eqnarray}
et $\ind_{\I}^{\K}\eta'(J)\otimes\eta(J)$ le $E$-espace vectoriel des fonctions $f:\K\rightarrow E$ telles que :
$$f(kk')=(\eta'(J)\otimes\eta(J))(k)f(k')$$
($k\in \I$, $k'\in \K$) que l'on munit de l'action \`a gauche de $\K$ par translation \`a droite sur les fonctions. Notons que cette action se factorise par ${\mathrm{GL}}_2(\Fq)$. On note de 
m\^eme $\ind_{\I}^{\K}\overline\eta'(J)\otimes\overline\eta(J)$ le $k_E$-espace vectoriel des fonctions $f:\K\rightarrow k_E$ telles que $f(kk')=(\overline\eta'(J)\otimes\overline\eta(J))(k)f(k')$ muni de la m\^eme action de $\K$. Rappelons que, si la $E$-repr\'esentation $\ind_{\I}^{\K}\eta'(J)\otimes\eta(J)$ \ est \ irr\'eductible, \ il \ n'en \ est \ pas \ de \ m\^eme \ de \ la \ $k_E$-repr\'esentation $\ind_{\I}^{\K}\overline\eta'(J)\otimes\overline\eta(J)$. En effet, ses constituants sont naturellement index\'es par un certain sous-ensemble de cardinal $\geq 2$ de l'ensemble des parties de ${\mathcal S}$ (qui co\"\i ncide g\'en\'erique\-ment avec l'ensemble des parties de ${\mathcal S}$), cf. \cite[\S\ 2]{BP} ou \cite[\S\ 2]{Br}. Par exemple son socle (irr\'eductible) correspond \`a $\emptyset$ et son co-socle (ibid.) \`a ${\mathcal S}$.

Pour $J\subseteq {\mathcal S}$, on pose :

\begin{equation}\label{fj}
F(J)\=\{\sigma\in {\mathcal S},\sigma\in J, \sigma\circ\varphi^{-1}\notin J\}\amalg \{\sigma\in {\mathcal S},\sigma\notin J, \sigma\circ\varphi^{-1}\in J\}
\end{equation}
($F(J)$ pour ``Fronti\`ere de $J$''). On note que $\vert F(J)\vert$ est toujours pair (\'eventuelle\-ment nul) et que l'on a $F({\mathcal S}\backslash J)=F(J)$. 

\begin{lem}\label{unik}
Soit $J\subseteq {\mathcal S}$ tel que $Z(\rhobar)\cap F(J)=\emptyset$. Un seul des constituants de $\ind_{\I}^{\K}\overline\eta'(J)\otimes\overline\eta(J)$ est un poids de Serre associ\'e \`a $\rhobar$, et c'est $\big(\otimes_{\sigma\in {\mathcal S}}(\Sym^{r_{\sigma}}k_E^2\big)^{\sigma})\otimes \theta\circ\det$. Dans l'indexation ci-dessus, il correspond \`a ${\mathcal S}\backslash J$.
\end{lem}
\begin{proof}
Fixons $\sigma_0\in {\mathcal S}$. Pour $j\in \{0,\cdots,f-1\}$, l'indice $j$ dans cette preuve signifie $\sigma_0\circ \varphi^j$. Le socle de $\ind_{\I}^{\K}\overline\eta'(J)\otimes\overline\eta(J)$ est irr\'educ\-tible, donc est un poids de Serre qui, par le lemme \ref{csigma} et avec les notations de \cite[\S\ 4]{Br}, correspond au $f$-uplet $\lambda=(\lambda_j(x_j))_{j\in \{0,\cdots,f-1\}}$ avec :
\begin{equation}\label{lambda}
\left\{\begin{array}{llcllll}
\lambda_j(x_j)&=&p-2-x_j&{\rm si}&\sigma_0\circ\varphi^j\notin J&{\rm et}&\sigma_0\circ\varphi^{j-1}\in J\\
\lambda_j(x_j)&=&p-1-x_j&{\rm si}&\sigma_0\circ\varphi^j\notin J&{\rm et}&\sigma_0\circ\varphi^{j-1}\notin J\\
\lambda_j(x_j)&=&x_j+1&{\rm si}&\sigma_0\circ\varphi^j\in J&{\rm et}&\sigma_0\circ\varphi^{j-1}\notin J\\
\lambda_j(x_j)&=&x_j&{\rm si}&\sigma_0\circ\varphi^j\in J&{\rm et}&\sigma_0\circ\varphi^{j-1}\in J
\end{array}\right.
\end{equation}
(attention qu'ici $x_j$ est une variable formelle comme dans \cite[\S\ 4]{Br} et ne d\'esigne pas l'\'el\'ement $x_{\sigma_0\circ\varphi^j}\in k_E$ de (\ref{norm})), c'est-\`a-dire qu'il s'agit du poids de Serre :
$$\big(\otimes_{j=0}^{f-1}(\Sym^{\lambda_j(r_j)}k_E^2)^{\sigma_0\circ \varphi^j}\big)\otimes \sigma_0(\det)^{e(\lambda)(r_0,\cdots,r_{f-1})}\theta\circ\det$$
o\`u ${\mathrm{GL}}_2(\Fq)$ agit sur le $j$-i\`eme facteur du produit tensoriel via $\sigma_0\circ \varphi^j$ et o\`u $e(\lambda)$ est d\'efini comme dans \cite[\S\ 4]{Br}. Par \cite[\S\ 2]{BP} (en particulier \cite[Lem.2.2]{BP}) et (\ref{lambda}), on en d\'eduit que $\big(\otimes_{\sigma\in {\mathcal S}}(\Sym^{r_{\sigma}}k_E^2\big)^{\sigma})\otimes \theta\circ\det$ est un constituant de $\ind_{\I}^{\K}\overline\eta'(J)\otimes\overline\eta(J)$ pour tout $J\subseteq {\mathcal S}$. Par ailleurs on a par l'\'egalit\'e (18) de \cite[\S\ 4]{Br} (attention que l'\'el\'ement $\mu_{j+1}=\mu_{\sigma_0\circ\varphi^{-j-1}}\in k_E$ de \cite[(16)]{Br} est \'egal \`a $x_{\sigma_0\circ\varphi^{-j}}=x_{\sigma_0\circ\varphi^{f-j}}\in k_E$ en (\ref{norm})) :
\begin{equation}\label{zj}
Z(\rhobar)=\{\sigma_0\circ\varphi^j,j\in J_{\rhobar}\}.
\end{equation}
Si $\lambda_j(x_j)\in \{p-2-x_j,x_j+1\}$, alors $\sigma_0\circ\varphi^j\in F(J)$ par (\ref{lambda}), donc $\sigma_0\circ\varphi^j\notin Z(\rhobar)$ (puisque $Z(\rhobar)\cap F(J)=\emptyset$) et donc $j\notin J_{\rhobar}$ par (\ref{zj}). En particulier les deux ensembles $J^{\rm min}$ et $J^{\rm max}$ de \cite[Prop.4.3]{Br} (cf. les \'egalit\'es (19) dans la preuve de {\it loc.cit.}) sont tous les deux \'egaux \`a :
\begin{equation}\label{jmin}
\left\{j\in \{0,\cdots,f-1\},\lambda_{j+1}(x_{j+1})\in \{p-1-x_{j+1},x_{j+1}+1\}\right\}.
\end{equation}
Par \cite[Prop.4.3]{Br}, l'ensemble des constituants de $\ind_{\I}^{\K}\overline\eta'(J)\otimes\overline\eta(J)$ qui sont dans ${\mathcal D}(\rhobar)$ est un singleton, n\'ecessairement constitu\'e du poids de Serre $(\otimes_{\sigma\in {\mathcal S}}(\Sym^{r_{\sigma}}k_E^2)^{\sigma})\otimes \theta\circ\det$. Ce poids de Serre est index\'e par $\{\sigma_0\circ\varphi^j,\ j\in J^{\rm min}\}$, c'est-\`a-dire en utilisant (\ref{jmin}) et (\ref{lambda}) par l'ensemble des $\sigma_0\circ\varphi^j$ pour $j$ tel que $\sigma_0\circ\varphi^{(j+1)-1}\notin J$, c'est-\`a-dire par ${\mathcal S}\backslash J$.
\end{proof}

\subsection{De Barsotti-Tate \`a Fontaine-Laffaille I}\label{fl1}

Cette partie et la suivante, qui seront utilis\'ees au \S\ \ref{bt}, ont pour but de calculer le module de Fontaine-Laffaille de repr\'esentations r\'eductibles de $\g$ sur $k_E$ provenant de certains modules fortement divisibles (ou groupes $p$-divisibles) avec donn\'ee de descente mod\'er\'ement ramifi\'ee.

On conserve les notations du \S\ \ref{prel} et on fixe un plongement $\sigma_0:\Fq\hookrightarrow k_E$. 

Comme dans \cite[\S\ 5]{Br}, on pose $e\=p^f-1$ et on note $S$ le compl\'et\'e $p$-adique de l'enveloppe aux puissances divis\'ees de $\oE[u]$ par rapport \`a l'id\'eal $(u^e+p)\oE[u]$ compatibles avec les puissances divis\'ees sur l'id\'eal $p\oE[u]$, et ${\rm Fil}^pS\subseteq S$ le compl\'et\'e $p$-adique de l'id\'eal engendr\'e par $\frac{(u^e+p)^i}{i!}$ pour $i\geq p$. On renvoie \`a \cite[\S\ 5]{Br} (et aux r\'ef\'erences donn\'ees dans {\it loc.cit.}) pour le rappel de ce qu'est un $\oE$-module fortement divisible et, si $\eta,\eta':\Fq^{\times}\rightarrow \oE^{\times}$ sont deux caract\`eres multiplicatifs distincts, pour la d\'efinition d'un $\oE$-module fortement divisible de type $\eta\otimes\eta'$.

On fixe $\eta,\eta':\Fq^{\times}\rightarrow \oE^{\times}$ distincts et on note $c\in \{1,\cdots,q-2\}$ l'unique entier tel que $\eta=\omega_{\sigma_0}^c\eta'$, que l'on \'ecrit $c=\sum_{i=0}^{f-1}c_ip^i$ avec $c_i\in \{0,\cdots,p-1\}$ ($c_i$ doit \^etre vu comme $c_{\sigma_0\circ\varphi^i}$). Pour $j\in \{0,\cdots,f-1\}$, on note $c^{(j)}\=\sum_{i=0}^{j-1}c_{f-(j-i)}p^i+\sum_{i=j}^{f-1}c_{i-j}p^i\in \{1,\cdots,q-2\}$. 

On consid\`ere dans la suite des $\oE$-modules fortement divisibles $\M$ de type $\eta\otimes\eta'$ qui ont la forme suivante :

(i) $\M=\M^{\sigma_0}\times\M^{\sigma_0\circ\varphi^{-1}}\times\cdots\times\M^{\sigma_0\circ\varphi^{-(f-1)}}$ avec $\M^{\sigma_0\circ\varphi^{-j}}=S e^{\sigma_0\circ\varphi^{-j}}_{\eta}\oplus Se^{\sigma_0\circ\varphi^{-j}}_{\eta'}$

(ii) $\Gal(L[\sqrt[e]{-p}]/L)$ agit sur $e^{\sigma_0\circ\varphi^{-j}}_{\eta}$ (resp. $e^{\sigma_0\circ\varphi^{-j}}_{\eta'}$) par $\eta$ (resp. $\eta'$)

(iii) pour tout $j\in \{0,\cdots,f-1\}$, on a l'une des deux possibilit\'es ci-dessous pour l'application $\varphi_1:{\rm Fil}^1\M^{\sigma_0\circ\varphi^{-j}}\rightarrow \M^{\sigma_0\circ\varphi^{-(j+1)}}$ (o\`u l'on remplace $\sigma_0\circ\varphi^{-j}$ par $j$ pour all\'eger l'\'ecriture) :
\begin{equation}\label{iii_1}
\left\{\begin{array}{lll}
{\rm Fil}^1\M^j&=&\Big(S(e^j_{\eta}+a_ju^{c^{(j)}}e^j_{\eta'})\oplus S(u^e+p) e^j_{\eta'}\Big) + {\rm Fil}^pS \M^j\\
\varphi_1(e^j_{\eta}+a_ju^{c^{(j)}}e^j_{\eta'})&=&e^{j+1}_{\eta}\\
\varphi_1((u^e+p)e^j_{\eta'})&=&e^{j+1}_{\eta'}
\end{array}\right.
\end{equation}
\begin{equation}\label{iii_2}
\ \ \ \ \ \ \left\{\begin{array}{lll}
{\rm Fil}^1\M^j&\!\!\!=&\!\!\!\Big(S(u^e+p)e^j_{\eta}\oplus S(e^j_{\eta'}+a_ju^{e-c^{(j)}}e^j_{\eta})\Big)+ {\rm Fil}^pS \M^j\\
\varphi_1((u^e+p)e^j_{\eta})&\!\!\!=&\!\!\!e^{j+1}_{\eta}\\
\varphi_1(e^j_{\eta'}+a_ju^{e-c^{(j)}}e^j_{\eta})&\!\!\!=&\!\!\!e^{j+1}_{\eta'}
\end{array}\right.
\end{equation}
pour des $a_j=a_{\sigma_0\circ\varphi^{-j}}\in \oE$, et avec $\alpha e^{0}_{\eta}$ et $\alpha' e^{0}_{\eta'}$ au lieu de $e^{j+1}_{\eta}$ et $e^{j+1}_{\eta'}$ dans l'image de $\varphi_1$ si $j=f-1$ (pour des $\alpha,\alpha'\in \oE^{\times}$). On note $I_{\eta}$ (resp. $I_{\eta'}$) le sous-ensemble de ${\mathcal S}$ form\'e des $\sigma_0\circ\varphi^{-j}$ pour $\M^j=\M^{\sigma_0\circ\varphi^{-j}}$ comme en (\ref{iii_1}) (resp. comme en (\ref{iii_2})) et $I_{\eta}^{\times}$ (resp. $I_{\eta'}^{\times}$) le sous-ensemble de $I_{\eta}$ (resp. $I_{\eta'}$) des $\sigma_0\circ\varphi^{-j}$ tels que $a_{\sigma_0\circ\varphi^{-j}}\in \oE^{\times}$.

Jusqu'\`a la fin de cette section, on fixe $(r_{\sigma})_{\sigma\in {\mathcal S}}$ et $\theta$ comme au \S\ \ref{prel}, c'est-\`a-dire $r_{\sigma}\in \{0,\cdots,p-3\}$ pour tout $\sigma$ avec $(r_{\sigma})_{\sigma\in {\mathcal S}}\ne (0,\cdots,0),(p-3,\cdots,p-3)$, et $\theta(p)=1$.

\begin{definit}\label{J}
Soit $J\subseteq {\mathcal S}$ et $\eta(J)$, $\eta'(J)$ comme en (\ref{type}). On dit qu'un $\oE$-module fortement divisible $\M$ est de type $J$ si $\M$ est comme ci-dessus avec $\eta=\eta(J)$, $\eta'=\eta'(J)$ et si l'on a :
\begin{eqnarray*}
{I}_{\eta(J)}^{\times}&\subseteq &\{\sigma\in {\mathcal S},\sigma\circ\varphi^{-1}\notin J\}\\
{I}_{\eta'(J)}^{\times}&\subseteq &\{\sigma\in {\mathcal S},\sigma\circ\varphi^{-1}\in J\}\\
{I}_{\eta(J)}\backslash {I}_{\eta(J)}^{\times}&\subseteq &\{\sigma\in {\mathcal S},\sigma\notin J, \sigma\circ\varphi^{-1}\notin J\}\\
{I}_{\eta'(J)}\backslash {I}_{\eta'(J)}^{\times}&\subseteq &\{\sigma\in {\mathcal S},\sigma\in J, \sigma\circ\varphi^{-1}\in J\}.
\end{eqnarray*}
\end{definit}

\begin{rem}\label{plus}
{\rm (i) Les modules fortement divisibles de type $J$ sont donc des modules fortement divisibles de type $\eta(J)\otimes\eta'(J)$ (au sens de \cite[D\'ef.5.1]{Br}) particuliers.\\
(ii) On a ${I}_{\eta(J)}\subseteq \{\sigma\in {\mathcal S},\sigma\circ\varphi^{-1}\notin J\}$ et ${I}_{\eta'(J)}\subseteq \{\sigma\in {\mathcal S},\sigma\circ\varphi^{-1}\in J\}$. Or ${I}_{\eta(J)}\amalg {I}_{\eta'(J)}={\mathcal S}$ ce qui force en fait ${I}_{\eta(J)}= \{\sigma\in {\mathcal S},\sigma\circ\varphi^{-1}\notin J\}$ et ${I}_{\eta'(J)}= \{\sigma\in {\mathcal S},\sigma\circ\varphi^{-1}\in J\}$. En particulier on a toujours $\vert {I}_{\eta(J)}\vert=\vert {\mathcal S}\backslash J\vert$ et $\vert {I}_{\eta'(J)}\vert=\vert J\vert$.\\
(iii) \^Etre de type $J$ implique en particulier $a_{\sigma}\in \oE^{\times}$ si $\sigma\in F(J)$ (voir (\ref{fj})).\\
(iv) Un module fortement divisible de type $J$ est de type ${\mathcal S}\backslash J$ si l'on \'echange $\alpha$ et $\alpha'$, $\eta$ et $\eta'$ et si l'on remplace $c$ par $e-c$.}
\end{rem}

Nous montrons dans cette section et la suivante que les $\oE$-modules fortement divisibles de type $J$ sont exactement les $\oE$-modules fortement divisibles de type $\eta(J)\otimes\eta'(J)$ (au sens de \cite[D\'ef.5.1]{Br}) tels que la repr\'esentation $\rhobar={\rm Hom}_{{\rm Fil}^1,\varphi_1}(\Mbar,\widehat A_{\rm cris}\otimes_{\Zp}\Fp)^{\vee}(1)$ (voir \cite[\S\ 6]{Br} pour les notations) est r\'eductible g\'en\'erique avec $Z(\rhobar)\cap F(J)=\emptyset$.

\begin{prop}\label{J'}
Soit $\rhobar$ r\'eductible g\'en\'erique et $J\subseteq {\mathcal S}$ tel que $Z(\rhobar)\cap F(J)=\emptyset$. Soit $\M$ un $\oE$-module fortement divisible de type $\eta(J)\otimes\eta'(J)$ tel que~:
$$\rhobar\simeq {\rm Hom}_{{\rm Fil}^1,\varphi_1}(\Mbar,\widehat A_{\rm cris}\otimes_{\Zp}\Fp)^{\vee}(1).$$
Alors $\M$ est de type $J$.
\end{prop}
\begin{proof}
Par le lemme \ref{unik} et \cite[Thm.8.1(ii)]{Br}, on doit avoir ${\mathcal S}={I}_{\eta(J)}\amalg {I}_{\eta'(J)}$ et : 
$${I}_{\eta(J)}=\{\sigma\circ\varphi, \sigma\in {\mathcal S}\backslash J\}=\{\sigma\in {\mathcal S},\sigma\circ\varphi^{-1} \notin J\}.$$
On a donc aussi ${I}_{\eta'(J)}=\{\sigma\in {\mathcal S},\sigma\circ\varphi^{-1}\in J\}$. Par \cite[Thm.8.1(i)]{Br} si $\rhobar$ est scind\'ee, (\ref{zj}) et \cite[Prop.7.3]{Br} si $\rhobar$ est non scind\'ee, on a $Z(\rhobar)=\{\sigma\in {\mathcal S}, \overline a_{\sigma}=0\}$. Avec $Z(\rhobar)\cap F(J)=\emptyset$, on en d\'eduit les inclusions :
\begin{eqnarray*}
{I}_{\eta(J)}\backslash {I}_{\eta(J)}^{\times}&\subseteq &\{\sigma\in {\mathcal S},\sigma\notin J, \sigma\circ\varphi^{-1}\notin J\}\\
{I}_{\eta'(J)}\backslash {I}_{\eta'(J)}^{\times}&\subseteq &\{\sigma\in {\mathcal S},\sigma\in J, \sigma\circ\varphi^{-1}\in J\}.
\end{eqnarray*}
\end{proof}

Soit $\M$ un $\oE$-module fortement divisible de type $J$, on pose :
\begin{equation}\label{ab}
\left\{\begin{array}{llcc}
A\=\overline\alpha \prod_{\sigma\in F(J)}\overline a_{\sigma}&{\rm si}&\sigma_0\in J\\
A\=\overline\alpha'\prod_{\sigma\in F(J)}\overline a_{\sigma}&{\rm si}&\sigma_0\notin J
\end{array}\right.
\end{equation}
et on remarque que $A\in k_E^{\times}$ par la remarque \ref{plus}(iii).
 
\begin{prop}\label{prop1}
Soit $J\subseteq {\mathcal S}$, $\M$ un $\oE$-module fortement divisible de type $J$
et $\rhobar\={\rm Hom}_{{\rm Fil}^1,\varphi_1}(\Mbar,\widehat A_{\rm cris}\otimes_{\Zp}\Fp)^{\vee}(1)$. On a:
$$\rhobar\cong\begin{pmatrix}\big(\nr(A)\displaystyle{\prod_{\sigma\in {\mathcal S}}}\omega_{\sigma}^{r_{\sigma}}\big)\omega
&*\\0&\nr(A^{-1}\overline\alpha\overline\alpha')\end{pmatrix}\otimes \theta.$$
\end{prop}
\begin{proof}
On montre d'abord que $\Mbar$ contient un sous-module libre de rang $1$ facteur direct stable par toutes les op\'erations ($\varphi_1$ sur le ${\rm Fil}^1$ induit et action de 
$\Gal(L[\sqrt[e]{-p}]/L)$). On pose pour $j\in \{0,\cdots,f-1\}$ :
\begin{equation*}
\begin{array}{llcllll}
e^j&\=&\overline e^j_{\eta(J)}&{\rm si}&\sigma_0\circ\varphi^{-(j-1)}\in J&{\rm et}&\sigma_0\circ\varphi^{-j}\in J\\
e^j&\=&\overline e^j_{\eta'(J)}&{\rm si}&\sigma_0\circ\varphi^{-(j-1)}\notin J&{\rm et}&\sigma_0\circ\varphi^{-j}\notin J\\
e^j&\=&\overline e^j_{\eta(J)}-\overline a_{j-1}^{-1}u^{pc^{(j-1)}}\overline e^j_{\eta'(J)}&{\rm si}&\sigma_0\circ\varphi^{-(j-1)}\notin J&{\rm et}&\sigma_0\circ\varphi^{-j}\in J\\
e^j&\=&\overline e^j_{\eta'(J)}-\overline a_{j-1}^{-1}u^{p(e-c^{(j-1)})}\overline e^j_{\eta(J)}&{\rm si}&\sigma_0\circ\varphi^{-(j-1)}\in J&{\rm et}&\sigma_0\circ\varphi^{-j}\notin J
\end{array}
\end{equation*}
en rempla\c cant $\overline a_{j-1}^{-1}$ par $\overline a_{f-1}^{-1}(\frac{\overline\alpha}{\overline\alpha'})^{-1}$ (resp. $\overline a_{f-1}^{-1}\frac{\overline\alpha}{\overline\alpha'}$) si $j=0$ et $\sigma_0\in J$ (resp. $\sigma_0\notin J$). Montrons que le sous-module $\Sbar e^0\times\cdots\times \Sbar e^{f-1}$ de $\Mbar$ est stable par toutes les op\'erations. La stabilit\'e par $\Gal(L[\sqrt[e]{-p}]/L)$ est \'el\'ementaire et laiss\'ee au lecteur. Consid\'erons d'abord les cas $\sigma_0\circ\varphi^{-j}\in F(J)$. En utilisant :
\begin{equation*}
\begin{array}{lllll}
e-c^{(j)}+pc^{(j-1)}&=&e(c_{f-j}+1)&=&e+ec_{\sigma_0\circ\varphi^{-j}}\\
c^{(j)}+p(e-c^{(j-1)})&=&e(p-c_{f-j})&=&e+e(p-1-c_{\sigma_0\circ\varphi^{-j}})
\end{array}
\end{equation*}
on a par un petit calcul :
$$\begin{array}{llcllll}
u^{e-c^{(j)}}e^j&\in&{\rm Fil}^1\Mbar&{\rm si}&\sigma_0\circ\varphi^{-j}\in J&{\rm et}&\sigma_0\circ\varphi^{-(j+1)}\notin J\\
u^{c^{(j)}}e^j&\in&{\rm Fil}^1\Mbar&{\rm si}&\sigma_0\circ\varphi^{-j}\notin J&{\rm et}&\sigma_0\circ\varphi^{-(j+1)}\in J
\end{array}$$
(noter que $\sigma_0\circ\varphi^{-j}\in {I}_{\eta(J)}^{\times}$ dans le premier cas, $\sigma_0\circ\varphi^{-j}\in {I}_{\eta'(J)}^{\times}$ dans le deuxi\`eme). Comme $c_{\sigma_0\circ\varphi^{-j}}=r_{\sigma_0\circ\varphi^{-j}}+1\geq 1$ dans le premier cas, $p-1-c_{\sigma_0\circ\varphi^{-j}}=r_{\sigma_0\circ\varphi^{-j}}+1\geq 1$ dans le deuxi\`eme (cf. lemme \ref{csigma}), on v\'erifie que (rappelons que ${\rm Fil}^1\Mbar\cap \Sbar e^j=u^{e-c^{(j)}}\Sbar e^j$, resp. ${\rm Fil}^1\Mbar\cap \Sbar e^j=u^{c^{(j)}}\Sbar e^j$) :
$$\begin{array}{llllll}
\varphi_1(u^{e-c^{(j)}}e^j)&=&\varphi_1(u^{e-c^{(j)}}\overline e_{\eta(J)}^j)&=&-\overline a_je^{j+1}&{\rm dans\ le\ premier\ cas}\\
\varphi_1(u^{c^{(j)}}e^j)&=&\varphi_1(u^{c^{(j)}}\overline e_{\eta'(J)}^j)&=&-\overline a_je^{j+1}
&{\rm dans\ le\ deuxi\grave eme}
\end{array}$$
(en rempla\c cant $\overline a_{j}$ par $\overline a_{f-1}\overline\alpha$ si $j=f-1$ et $\sigma_0\in J$, par $\overline a_{f-1}\overline\alpha'$ si $j=f-1$ et $\sigma_0\notin J$). Dans les autres cas, c'est-\`a-dire $\sigma_0\circ\varphi^{-j}\notin F(J)$, on v\'erifie qu'on a toujours $\varphi_1(u^ee^j)=e^{j+1}$ si $j<f-1$, $\varphi_1(u^ee^{f-1})=\overline\alpha e^{0}$ si $\sigma_0\in J$, $\varphi_1(u^ee^{f-1})=\overline\alpha' e^{0}$ si $\sigma_0\notin J$. Pour r\'esumer, on a donc :
\begin{equation}\label{resum}
\left\{\begin{array}{llcllll}
\varphi_1(u^{e-c^{(j)}}e^j)&=&-\overline a_je^{j+1}&{\rm si}&\sigma_0\circ\varphi^{-j}\in J&\!\!\!{\rm et}&\sigma_0\circ\varphi^{-(j+1)}\notin J\\
\varphi_1(u^{c^{(j)}}e^j)&=&-\overline a_je^{j+1}&{\rm si}&\sigma_0\circ\varphi^{-j}\notin J&\!\!\!{\rm et}&\sigma_0\circ\varphi^{-(j+1)}\in J\\
\varphi_1(u^{e}e^j)&=&e^{j+1}&{\rm si}&\sigma_0\circ\varphi^{-j}\notin F(J)&&
\end{array}\right.
\end{equation}
en multipliant \`a droite par $\overline\alpha$ si $j=f-1$ et $\sigma_0\in J$, par $\overline\alpha'$ si $j=f-1$ et $\sigma_0\notin J$. Donc $\Sbar e^0\times\cdots\times \Sbar e^{f-1}$ est stable par toutes les op\'erations et ${\rm Hom}_{{\rm Fil}^1,\varphi_1}\big(\prod_{j=0}^{f-1}\Sbar e^j,\widehat A_{\rm cris}\otimes_{\Zp}\Fp\big)^{\!\vee}\!(1)$ est une sous-repr\'esentation de $\rhobar$ de dimension $1$ (sur $k_E$). Notons $0\leq j_1< j_2< \cdots < j_{2t-1}< j_{2t}\leq f-1$ les \'el\'ements de $\{0,\cdots,f-1\}$ tels que $\sigma_0\circ\varphi^{-j}\in F(J)$ (il y en a un nombre pair) et supposons d'abord $\sigma_0\in J$. On a alors par le lemme \ref{csigma} :
\begin{equation}\label{rsigma}
\left\{\begin{array}{llclllc}
c_{\sigma_0\circ\varphi^{-j}}&=&p-2-r_{\sigma_0\circ\varphi^{-j}}&{\rm si}&j = j_{2s}\\
c_{\sigma_0\circ\varphi^{-j}}&=&p-1-r_{\sigma_0\circ\varphi^{-j}}&{\rm si}&j_{2s-1}<j<j_{2s}\\
c_{\sigma_0\circ\varphi^{-j}}&=&r_{\sigma_0\circ\varphi^{-j}}+1&{\rm si}&j = j_{2s-1}\\
c_{\sigma_0\circ\varphi^{-j}}&=&r_{\sigma_0\circ\varphi^{-j}}&{\rm si}&j_{2s}<j<j_{2s+1}
\end{array}\right.
\end{equation}
pour $s\in \{1,\cdots,t\}$ et en identifiant $\{j_{2t}+1,\cdots,j_{2t+1}-1\}$ \`a $\{j_{2t}+1,\cdots,f-1\}\amalg \{0,\cdots,j_1-1\}$. Par \cite[Ex.3.7]{Sa} (et un petit travail de traduction), \cite[(33)]{Br}, (\ref{resum}) et (\ref{ab}), on a :
$${\rm Hom}_{{\rm Fil}^1,\varphi_1}\bigg(\prod_{j=0}^{f-1}\Sbar e^j,\widehat A_{\rm cris}\otimes_{\Zp}\Fp\bigg)^{\!\vee}\!(1)=\nr(A)\overline\eta(J)\omega_{\sigma_0}^{h}$$
o\`u $h=\displaystyle{\sum_{j\notin \{j_{s},1\leq s\leq 2t\}}\!\!\!\!\!p^{f-j}+\!\!\!\sum_{j\in \{j_{2s-1},1\leq s\leq t\}}\!\!\!\!\!p^{f-j}-\sum_{s=1}^t\sum_{j_{2s-1}<j\leq j_{2s}}\!\!\!p^{f-j}c_{\sigma_0\circ\varphi^{-j}}}$. On explicite $h$ avec (\ref{rsigma}) :
\begin{eqnarray*}
h&=&\!\!\!\displaystyle{\sum_{j\notin \{j_{2s}\}}\!\!p^{f-j}-\!\!\sum_{j\in \{j_{2s}\}}\!\!p^{f-j}(p-2-r_{\sigma_0\circ\varphi^{-j}})-\sum_{s=1}^t\sum_{j_{2s-1}<j<j_{2s}}\!\!\!\!p^{f-j}(p-1-r_{\sigma_0\circ\varphi^{-j}})}\\
&=&\displaystyle{\sum_{j=0}^{f-1}p^{f-j}-\sum_{s=1}^t\sum_{j_{2s-1}<j\leq j_{2s}}\!\!\!\!p^{f-j}(p-1-r_{\sigma_0\circ\varphi^{-j}})}\\
&=&\displaystyle{\sum_{j=0}^{f-1}p^{f-j}-\!\!\!\!\sum_{\sigma_0\circ\varphi^{-j}\notin J}\!\!\!\!p^{f-j}(p-1-r_{\sigma_0\circ\varphi^{-j}})}\\
&=&\!\!\!\!\displaystyle{\sum_{\sigma_0\circ\varphi^{-j}\notin J}\!\!\!\!p^{f-j}r_{\sigma_0\circ\varphi^{-j}}-\!\!\!\!\sum_{\sigma_0\circ\varphi^{-j}\notin J}\!\!\!\!p^{f-j}(p-1)+\sum_{j=0}^{f-1}p^{f-j}}
\end{eqnarray*}
donc $\omega_{\sigma_0}^h=\prod_{\sigma\notin J}\omega_{\sigma}^{r_{\sigma}}\prod_{\sigma\notin J}\omega_{\sigma}^{-(p-1)}\prod_{\sigma\in {\mathcal S}}\omega_{\sigma}$. Comme $\overline\eta(J)=\theta\!\prod_{\sigma\in J}\!\omega_{\sigma}^{r_{\sigma}}\!\prod_{\sigma\notin J}\!\omega_{\sigma}^{p-1}$ (cf. (\ref{type}) vu c\^ot\'e Galois), on a bien :
\begin{eqnarray}\label{enplus}
\overline\eta(J)\omega_{\sigma_0}^{h}=\theta\omega\prod_{\sigma\in {\mathcal S}}\omega_{\sigma}^{r_{\sigma}}.
\end{eqnarray} 
La forme du quotient de $\rhobar$ se d\'eduit de son d\'eterminant, qui, vu les hypoth\`eses sur $\M$, est $\omega\nr(\overline\alpha\overline\alpha')\overline\eta(J)\overline\eta'(J)=\theta^2\omega\nr(\overline\alpha\overline\alpha')\prod_{\sigma\in {\mathcal S}}\omega_{\sigma}^{r_{\sigma}}$ (cf. (\ref{type})). Le cas $\sigma_0\notin J$ est strictement analogue au cas $\sigma_0\in J$ et laiss\'e au lecteur.
\end{proof}

Une des cons\'equences de la proposition \ref{prop1} (et des hypoth\`eses sur les $r_{\sigma}$) est que $\rhobar$ est en particulier g\'en\'erique au sens de \cite[Def.11.7]{BP}, et donc est comme au \S\ \ref{prel}.

\subsection{De Barsotti-Tate \`a Fontaine-Laffaille II}\label{fl2}

Dans cette partie, on d\'etermine compl\`etement le module de Fontaine-Laffaille contravariant de la repr\'esentation $\rhobar\otimes \theta^{-1}$ provenant d'un $\oE$-module fortement divisible de type $J$.

On conserve les notations du \S\ \ref{fl1}. La preuve de la proposition qui suit est contenue dans celle de \cite[Prop.5.1.2]{BM} (en particulier, notons que par rapport \`a \cite[\S7]{Br} nous tenons compte ici dans les modules fortement divisibles des torsions par des caract\`eres de Galois non ramifi\'es). Pour la commodit\'e du lecteur, nous redonnons les d\'etails dans le cas particulier qui nous concerne.

\begin{prop}\label{prop2}
Soit $J$, $\M$ et $\rhobar$ comme dans la proposition \ref{prop1}. On a $\overline\rho\otimes\theta^{-1}={\rm Hom}_{{\rm Fil}^{\cdot},\varphi_{\cdot}}(M,A_{\rm cris}\otimes_{\Zp}\Fp)$ o\`u $M$ est le module de Fontaine-Laffaille :
$$M=M^{\sigma_0}\times M^{\sigma_0\circ\varphi^{-1}}\times\cdots\times M^{\sigma_0\circ\varphi^{-(f-1)}}$$
avec (cf. (\ref{ab}) pour $A$) :
\begin{equation*}
\!\!\!\!\!\!\!\!\!\!\!\!\!\!\!\!\!\!\left\{\begin{array}{llcl}
M^{\sigma_0\circ\varphi^{-j}}&=&k_E v^{\sigma_0\circ\varphi^{-j}}\oplus k_Ew^{\sigma_0\circ\varphi^{-j}}\\
{\rm Fil}^iM^{\sigma_0\circ\varphi^{-j}}&=&M^{\sigma_0\circ\varphi^{-j}}&i\leq 0\\
{\rm Fil}^{i}M^{\sigma_0\circ\varphi^{-j}}&=& k_E w^{\sigma_0\circ\varphi^{-j}}&1\leq i\leq r_{\sigma_0\circ\varphi^{-j}}+1\\
{\rm Fil}^{i}M^{\sigma_0\circ\varphi^{-j}}&=&0&r_{\sigma_0\circ\varphi^{-j}}+2\leq i
\end{array}\right.
\end{equation*}
\begin{equation*}
\left\{\begin{array}{lll}
\varphi(v^{\sigma_0\circ\varphi^{-j}})&=&v^{\sigma_0\circ\varphi^{-(j+1)}}\\
\varphi_{r_{\sigma_0\circ\varphi^{-j}}+1}(w^{\sigma_0\circ\varphi^{-j}})&=&w^{\sigma_0\circ\varphi^{-(j+1)}}-A_{j}v^{\sigma_0\circ\varphi^{-(j+1)}}
\end{array}\right.\ \ \ 0\leq j\leq f-2
\end{equation*}
\begin{equation*}\!\!\!\!\!\!\!\!\!\!\!\!\!\!\!\!\!\!\!\!\!\!\!\!\!\!\!\!\!\!\!\!\!\!\!\!\!\!\!\!\!\!\!\!\left\{\begin{array}{lll}
\varphi(v^{\sigma_0\circ\varphi^{-(f-1)}})&=&A\overline\alpha^{-1}\overline\alpha'^{-1}v^{\sigma_0}\\ 
\varphi_{r_{\sigma_0\circ\varphi^{-(f-1)}}+1}(w^{\sigma_0\circ\varphi^{-(f-1)}})&=&A^{-1}w^{\sigma_0}-A_{f-1}v^{\sigma_0}
\end{array}\right.
\end{equation*}
o\`u :
\begin{equation}\label{aj}
\left\{\begin{array}{llll}
A_{j}\!&=&\!\overline a_{\sigma_0\circ\varphi^{-j}}\!\!{\displaystyle \prod_{\stackrel{0\leq i\leq j-1}{ \sigma_0\circ\varphi^{-i}\in F(J)}}}\!\!(-\overline a_{\sigma_0\circ\varphi^{-i}}^{-2})&{\rm si}\ \ {\sigma_0\circ\varphi^{-j}}\notin F(J)\\
A_{j}\!&=&\!\overline a_{\sigma_0\circ\varphi^{-j}}^{-1}\!\!{\displaystyle\prod_{\stackrel{0\leq i\leq j-1}{ {\sigma_0\circ\varphi^{-i}}\in F(J)}}}\!\!(-\overline a_{\sigma_0\circ\varphi^{-i}}^{-2})&{\rm si}\ \ \sigma_0\circ\varphi^{-j}\in F(J)
\end{array}\right.
\end{equation}
en multipliant le membre de droite de (\ref{aj}) par $A\overline\alpha^{-1}\overline\alpha'^{-1}$ si $j=f-1$.
\end{prop}
\begin{proof}
On pose $\widetilde\eta(J)\=[\theta^{-1}]\eta(J)$, $\widetilde\eta'(J)\=[\theta^{-1}]\eta'(J)$ pour $J\subseteq \mathcal S$. On d\'efinit d'abord une suite $(\eta_{\sigma})_{\sigma\in {\mathcal S}}$ avec $\eta_{\sigma}\in \{\widetilde\eta(J),\widetilde\eta'(J)\}$ comme suit : si $J=\emptyset$, $\eta_{\sigma}\=\widetilde\eta'(J)$ pour tout $\sigma\in {\mathcal S}$, si $J={\mathcal S}$, $\eta_{\sigma}\=\widetilde\eta(J)$ pour tout $\sigma\in {\mathcal S}$ et si $J\notin \{\emptyset,{\mathcal S}\}$ :
\begin{equation}\label{eta}
\left\{\begin{array}{cclcc}
\eta_{\sigma}\=\widetilde\eta(J),\ \eta_{\sigma\circ\varphi^{-1}}\=\widetilde\eta'(J)&{\rm si}&\sigma\in J&{\rm et}&\sigma\circ\varphi^{-1}\notin J\\
\eta_{\sigma}\=\widetilde\eta'(J),\ \eta_{\sigma\circ\varphi^{-1}}\=\widetilde\eta(J)&{\rm si}&\sigma\notin J&{\rm et}&\sigma\circ\varphi^{-1}\in J\\
\eta_{\sigma}=\eta_{\sigma\circ\varphi^{-1}}&{\rm si}&\sigma\notin F(J).&&
\end{array}\right.
\end{equation}
Un examen facile de (\ref{eta}) montre que l'on obtient $\eta_{\sigma}=\widetilde\eta(J)$ si $\sigma\in J$ et $\eta_{\sigma}=\widetilde\eta'(J)$ si $\sigma\notin J$. Par ailleurs, par la d\'efinition \ref{J} et \cite[(32)]{Br}, on voit que (\ref{eta}) implique $\eta_{\sigma_0\circ\varphi^{-j}}=\eta_j$ pour tout $j\in \{0,\cdots,f-1\}$ o\`u $\eta_j$ est comme dans \cite[\S\ 6]{Br}. Supposons d'abord $\sigma_0\in J$ dans ce qui suit, ce qui revient donc \`a $\eta_0=\widetilde\eta(J)$. Par la proposition \ref{prop1} et (\ref{enplus}), on a :
$$(\rhobar\otimes \theta^{-1})^{\vee}(1)\otimes \overline{\widetilde\eta}(J)\omega_{\sigma_0}^{h}\omega^{-1}\cong(\rhobar\otimes \theta^{-1})\otimes \nr(\overline\alpha^{-1}\overline\alpha'^{-1}).$$
Le module de Fontaine-Laffaille contravariant de $\rhobar\otimes \theta^{-1}$ est donc donn\'e par les formules \`a la fin de la preuve de \cite[Prop.7.3]{Br} tordues (c\^ot\'e Galois) par $\nr(\overline\alpha\overline\alpha')$. Nous explicitons compl\`etement ce module dans ce qui suit. La d\'efinition \ref{J} et ce qui pr\'ec\`ede donnent les \'equivalences :
$$\begin{array}{ccccccc}
\sigma_0\circ\varphi^{-j}\in I_{\eta(J)}^{\times}&{\rm et}&\eta_j=\widetilde\eta(J)&\Leftrightarrow & \sigma_0\circ\varphi^{-j}\in J&{\rm et}&\sigma_0\circ\varphi^{-(j+1)}\notin J\\
\sigma_0\circ\varphi^{-j}\in I_{\eta'(J)}&{\rm et}&\eta_j=\widetilde\eta(J)&\Leftrightarrow & \sigma_0\circ\varphi^{-j}\in J&{\rm et}&\sigma_0\circ\varphi^{-(j+1)}\in J\\
\sigma_0\circ\varphi^{-j}\in I_{\eta'(J)}^{\times}&{\rm et}&\eta_j=\widetilde\eta'(J)&\Leftrightarrow & \sigma_0\circ\varphi^{-j}\notin J&{\rm et}&\sigma_0\circ\varphi^{-(j+1)}\in J\\
\sigma_0\circ\varphi^{-j}\in I_{\eta(J)}&{\rm et}&\eta_j=\widetilde\eta'(J)&\Leftrightarrow & \sigma_0\circ\varphi^{-j}\notin J&{\rm et}&\sigma_0\circ\varphi^{-(j+1)}\notin J.
\end{array}$$
Avec le lemme \ref{csigma}, on voit donc que le module de Fontaine-Laffaille contrava\-riant $M=M^{\sigma_0}\times\cdots\times M^{\sigma_0\circ\varphi^{-(f-1)}}$ de la preuve de \cite[Prop.7.3]{Br} (tordu par $\nr(\overline\alpha\overline\alpha')$ c\^ot\'e Galois) est donn\'e par $M^{\sigma_0\circ\varphi^{-j}}=k_E e^{\sigma_0\circ\varphi^{-j}}\oplus k_Ef^{\sigma_0\circ\varphi^{-j}}=k_E e^{j}\oplus k_Ef^{j}$ avec (la d\'efinition de la filtration \'etant implicite) :
$$\begin{array}{llcll}
\left\{\begin{array}{llc}
\varphi(e^{j})&=&e^{j+1}-\overline a_jf^{j+1}\\
\varphi_{r_{j}+1}(f^{j})&=&f^{j+1}
\end{array}\right.
\ \ \ \ \ &{\rm si}&\ \sigma_0\circ\varphi^{-j}\in J&{\rm et}&\sigma_0\circ\varphi^{-(j+1)}\notin J\\
\left\{\begin{array}{llc}
\varphi(e^{j})&=&e^{j+1}\\
\varphi_{r_{j}+1}(f^{j})&=&f^{j+1}-\overline a_je^{j+1}
\end{array}\right.
\ \ \ \ \ &{\rm si}&\ \sigma_0\circ\varphi^{-j}\in J&{\rm et}&\sigma_0\circ\varphi^{-(j+1)}\in J\\
\left\{\begin{array}{llc}
\varphi_{r_{j}+1}(e^{j})&=&e^{j+1}\\
\varphi(f^{j})&=&f^{j+1}-\overline a_je^{j+1}
\end{array}\right.
\ \ \ \ \ &{\rm si}&\ \sigma_0\circ\varphi^{-j}\notin J&{\rm et}&\sigma_0\circ\varphi^{-(j+1)}\in J\\
\left\{\begin{array}{llc}
\varphi_{r_{j}+1}(e^{j})&=&e^{j+1}-\overline a_jf^{j+1}\\
\varphi(f^{j})&=&f^{j+1}
\end{array}\right.
\ \ \ \ \ &{\rm si}&\ \sigma_0\circ\varphi^{-j}\notin J&{\rm et}&\sigma_0\circ\varphi^{-(j+1)}\notin J
\end{array}$$
o\`u $r_{j}=r_{\sigma_0\circ\varphi^{-j}}$, $\overline a_{j}=\overline a_{\sigma_0\circ\varphi^{-j}}$ et avec $\overline\alpha'^{-1} e^0$ et $\overline\alpha^{-1} f^0$ au lieu de $e^{j+1}$ et $f^{j+1}$ dans l'image de $\varphi$ et $\varphi_{r_{j}+1}$ si $j=f-1$ (les formules pr\'ecises avec $\overline\alpha'$ et $\overline\alpha$ pour $j=f-1$ ne sont pas explicit\'ees dans la preuve de \cite[Prop.7.3]{Br} o\`u l'on se contentait de d\'ecrire l'action de l'inertie seulement, mais un petit calcul les donne en notant que l'on a ici tordu c\^ot\'e Galois par $\nr(\overline\alpha\overline\alpha')$, ce qui revient c\^ot\'e Fontaine-Laffaille \`a multiplier $\varphi$ et $\varphi_{r_{j}+1}$ par $(\overline\alpha\overline\alpha')^{-1}$ au cran $j=f-1$). Pour $j\in \{0,\cdots,f-1\}$ on pose $f^{'j}\=f^j$ si $\sigma_0\circ\varphi^{-j}\in J$ et $f^{'j}\=e^j$ si $\sigma_0\circ\varphi^{-j}\notin J$. Pour $j\in \{1,\cdots,f-1\}$ on pose :
\begin{equation*}
\begin{array}{llcllll}
e^{'j}&\=&e^j&{\rm si}&\sigma_0\circ\varphi^{-(j-1)}\in J&{\rm et}&\sigma_0\circ\varphi^{-j}\in J\\
e^{'j}&\=&f^j&{\rm si}&\sigma_0\circ\varphi^{-(j-1)}\notin J&{\rm et}&\sigma_0\circ\varphi^{-j}\notin J\\
e^{'j}&\=&e^j-\overline a_{j-1}f^j&{\rm si}&\sigma_0\circ\varphi^{-(j-1)}\in J&{\rm et}&\sigma_0\circ\varphi^{-j}\notin J\\
e^{'j}&\=&f^j-\overline a_{j-1}e^j&{\rm si}&\sigma_0\circ\varphi^{-(j-1)}\notin J&{\rm et}&\sigma_0\circ\varphi^{-j}\in J
\end{array}
\end{equation*}
et pour $j=0$ :
\begin{equation*}
\begin{array}{llcll}
e^{'0}&\=&e^0&{\rm si}&\sigma_0\circ\varphi^{-(f-1)}\in J\\
e^{'0}&\=&e^0-\overline a_{f-1}^{-1}(\frac{\overline\alpha}{\overline\alpha'})^{-1}f^0&{\rm si}&\sigma_0\circ\varphi^{-(f-1)}\notin J
\end{array}
\end{equation*}
en se souvenant que $\sigma_0\in J$ par hypoth\`ese. Un calcul au cas par cas, un peu laborieux mais sans difficult\'e, donne alors $\varphi(e^{'0})=e^{'1}$ puis :
$$\begin{array}{ll}
\left\{\begin{array}{llcll}
\varphi(e^{'j})&=&e^{'j+1}&{\rm si}&\sigma_0\circ\varphi^{-(j-1)}\notin F(J)\\
\varphi(e^{'j})&=&-\overline a_{j-1}e^{'j+1}&{\rm si}&\sigma_0\circ\varphi^{-(j-1)}\in F(J)
\end{array}\right.
\ \ \ &1\leq j\leq f-2\\ \\
\left\{\begin{array}{llcll}
\varphi_{r_{j}+1}(f^{'j})&=&f^{'j+1}-\overline a_j e^{'j+1}&{\rm si}&\sigma_0\circ\varphi^{-j}\notin F(J)\\
\varphi_{r_{j}+1}(f^{'j})&=&\overline a_j^{-1} f^{'j+1}-\overline a_j^{-1} e^{'j+1}&{\rm si}&\sigma_0\circ\varphi^{-j}\in F(J)
\end{array}\right.
\ \ \ &0\leq j\leq f-2
\end{array}$$
et si $j=f-1$ :
$$\begin{array}{l}
\left\{\begin{array}{llcllll}
\varphi(e^{'f-1})&=&\overline\alpha'^{-1} e^{'0}&{\rm si}&\sigma_0\circ\varphi^{-(f-2)}\in J&{\rm et}&\sigma_0\circ\varphi^{-(f-1)}\in J\\
\varphi(e^{'f-1})&=&-\overline a_{f-2}\overline\alpha'^{-1} e^{'0}&{\rm si}&\sigma_0\circ\varphi^{-(f-2)}\notin J&{\rm et}&\sigma_0\circ\varphi^{-(f-1)}\in J\\
\varphi(e^{'f-1})&=&-\overline a_{f-1}\overline\alpha'^{-1} e^{'0}&{\rm si}&\sigma_0\circ\varphi^{-(f-2)}\notin J&{\rm et}&\sigma_0\circ\varphi^{-(f-1)}\notin J\\
\varphi(e^{'f-1})&=&\overline a_{f-2}\overline a_{f-1}\overline\alpha'^{-1} e^{'0}&{\rm si}&\sigma_0\circ\varphi^{-(f-2)}\in J&{\rm et}&\sigma_0\circ\varphi^{-(f-1)}\notin J
\end{array}\right.\\ \\
\left\{\begin{array}{llcll}
\varphi_{r_{f-1}+1}(f^{'f-1})&=&\overline\alpha^{-1}f^{'0}-\overline a_{f-1} \overline\alpha'^{-1} e^{'0}&{\rm si}&\sigma_0\circ\varphi^{-(f-1)}\in J\\
\varphi_{r_{f-1}+1}(f^{'f-1})&=&\overline a_{f-1}^{-1}\overline\alpha^{-1}f^{'0}+ \overline\alpha'^{-1} e^{'0}&{\rm si}&\sigma_0\circ\varphi^{-(f-1)}\notin J
\end{array}\right.
\end{array}$$
(lorsque $f=1$, on obtient simplement $\left\{\begin{array}{lcl}
\varphi(e^{'0})&=&\overline\alpha'^{-1} e^{'0}\\
\varphi_{r_{0}+1}(f^{'0})&=&\overline\alpha^{-1}f^{'0}-\overline a_{0} \overline\alpha'^{-1} e^{'0}
\end{array}\right.$). Le changement de base :
$$\begin{array}{l}
v^{\sigma_0}\=e^{'0},\ v^{\sigma_0\circ\varphi^{-1}}\=e^{'1},\ v^{\sigma_0\circ\varphi^{-j}}\=\big(\prod_{\stackrel{0\leq i\leq j-2 }{\sigma_0\circ\varphi^{-i}\in F(J)}}(-\overline a_i)\big)e^{'j}\ {\rm si}\ 2\leq j\leq f-1\\
w^{\sigma_0}\=f^{'0},\ w^{\sigma_0\circ\varphi^{-j}}\=\big(\prod_{\stackrel{0\leq i\leq j-1 }{\sigma_0\circ\varphi^{-i}\in F(J)}}\overline a_i^{-1}\big)f^{'j}\ {\rm si}\ 1\leq j\leq f-1
\end{array}$$
donne alors par un calcul facile exactement la pr\'esentation de l'\'enonc\'e (en remarquant que $\sigma_0\circ\varphi^{-(f-1)}\in F(J)$ si et seulement si $\sigma_0\circ\varphi^{-(f-1)}\notin J$ et que $\prod_{\stackrel{0\leq i\leq f-1 }{\sigma_0\circ\varphi^{-i}\in F(J)}}(-\overline a_i)=\prod_{\sigma\in {\mathcal S}}\overline a_{\sigma}$ car $\vert F(J)\vert$ est pair). Le calcul lorsque $\sigma_0\notin J$ est analogue (on peut aussi utiliser la remarque \ref{plus}(iv)).
\end{proof}

Si $J$, $\M$ et $\rhobar$ sont comme dans la proposition \ref{prop2} (ou la proposition \ref{prop1}), on voit avec (\ref{zrho}) et (\ref{aj}) que l'on a :
\begin{equation*}
Z(\rhobar)=\{\sigma_0\circ\varphi^{-j}, \overline a_{\sigma_0\circ\varphi^{-j}}=0\}=\{\sigma\in {\mathcal S}, \overline a_{\sigma}=0\}
\end{equation*}
(ce r\'esultat d\'ecoule aussi de \cite[Thm.8.1]{Br} (cas scind\'e) et de \cite[Prop.7.3]{Br} et (\ref{zj}) (cas non scind\'e)). En particulier, on a toujours $Z(\rhobar)\cap F(J)=\emptyset$ par la remarque \ref{plus}(iii) si $\rhobar$ provient d'un $\oE$-module fortement divisible de type $J$.

\subsection{Valeurs propres de Frobenius}\label{bt}

On calcule la valuation et le ``coefficient dominant'' des valeurs propres de Frobenius sur le module de Dieudonn\'e d'un groupe $p$-divisible provenant d'un $\oE$-module fortement divisible de type $J$ dont la repr\'esentation galoisienne r\'esiduelle associ\'ee est r\'eductible g\'en\'erique. 

On conserve les notations des sections pr\'ec\'edentes. Le th\'eor\`eme suivant est le r\'esultat local clef de cet article.

\begin{thm}\label{thm_reduction}
Soit $\rhobar$ r\'eductible g\'en\'erique, $J\subseteq {\mathcal S}$ et $\M$ un $\oE$-module fortement divisible de type $J$ tel que $\rhobar\simeq {\rm Hom}_{{\rm Fil}^1,\varphi_1}(\Mbar,\widehat A_{\rm cris}\otimes_{\Zp}\Fp)^{\vee}(1)$. Soit $G$ le groupe $p$-divisible avec donn\'ee de descente correspondant \`a $\M$ et $D$ le module de Dieudonn\'e contravariant associ\'e \`a $G$. Alors la valeur propre de $\varphi^f$ sur la partie isotypique de $D$ pour le caract\`ere $\eta'(J)$ de $\Gal(L[\sqrt[e]{-p}]/L)$ est \'egale \`a $p^{\vert J\vert}\alpha'$ o\`u $\alpha'\in \oE^{\times}$ a pour r\'eduction dans $k_E^{\times}$ :
\begin{eqnarray}\label{reduction}
\overline \alpha'=(-1)^{\frac{1}{2}\vert F(J)\vert}\Big(\prod_{\sigma\in J}\alpha_{\sigma}\prod_{\sigma\notin J}\beta_{\sigma}\Big)^{-1}\frac{\displaystyle{\prod_{\stackrel{\sigma\in J}{ \sigma\circ\varphi^{-1}\notin J}}}x_{\sigma}}{\displaystyle{\prod_{\stackrel{\sigma\notin J}{\sigma\circ\varphi^{-1}\in J}}}x_{\sigma}}
\end{eqnarray}
avec $(\alpha_{\sigma})_{\sigma\in {\mathcal S}}$, $(\beta_{\sigma})_{\sigma\in {\mathcal S}}$ et $(x_{\sigma})_{\sigma\in {\mathcal S}}$ comme en (\ref{norm}).
\end{thm}
\begin{proof}
Le module de Dieudonn\'e $D$ v\'erifie $D=\M\otimes_SE$ o\`u la fl\`eche $S\twoheadrightarrow \oE\subset E$ est le morphisme de $\oE$-alg\`ebres qui envoie $u$ et ses puissances divis\'ees sur $0$. En revenant \`a la d\'efinition de $I_{\eta}$ et $I_{\eta'}$ en (\ref{iii_1}) et (\ref{iii_2}) (pour $(\eta,\eta')=(\eta(J),\eta'(J))$), on voit que la valeur propre de $\varphi^f$ sur la partie isotypique de $D$ pour le caract\`ere $\eta'(J)$ de $\Gal(L[\sqrt[e]{-p}]/L)$ est $p^{\vert I_{\eta'(J)}\vert}\alpha'$. Mais on a $\vert I_{\eta'(J)}\vert=\vert J\vert$ par la remarque \ref{plus}(ii). Il reste donc \`a d\'emontrer (\ref{reduction}).\\
Notons d'abord que le terme de droite en (\ref{reduction}) est dans $k_E^{\times}$ (car $Z(\rhobar)\cap F(J)=\emptyset$, cf. fin de la section \ref{fl2}). Comme il ne d\'epend que de $\overline\rho$ et de $J$ par le lemme \ref{joli}, il suffirait de montrer (\ref{reduction}) avec le module de Fontaine-Laffaille de la proposition \ref{prop2}. N\'eanmoins, comme nous avons laiss\'e la preuve du lemme \ref{joli} au lecteur, donnons ici une preuve directe de l'identit\'e (\ref{reduction}) qui n'utilise pas ce lemme. Fixons un plongement $\sigma_0:\Fq\hookrightarrow k_E$. En explicitant le changement de base qui permet de passer du module de Fontaine-Laffaille contravariant de $\rhobar\otimes \theta^{-1}$ du \S\ \ref{prel} \`a celui de la proposition \ref{prop2}, on voit que l'on a les \'egalit\'es $\overline\alpha\overline\alpha'=\lambda\mu$ et $A=\lambda$ (notons que, si $\rhobar$ est scind\'ee, il y a deux choix de $\theta$ comme dans la Remarque \ref{comment}(ii), mais une fois que l'on a fix\'e l'un de ces choix, on a ces \'egalit\'es en consid\'erant le module de Fontaine-Laffaille correspondant). De plus par (\ref{ab}) on a :
\begin{eqnarray}\label{abc}
\overline \alpha'=\mu \!\prod_{\sigma\in F(J)}\!\overline a_{\sigma}{\rm\ \ si\ \ }\sigma_0\in J,\ \ \ \overline \alpha'=\lambda \!\prod_{\sigma\in F(J)}\!\overline a_{\sigma}^{-1}{\rm\ \ si\ \ }\sigma_0\notin J.
\end{eqnarray}
Nous allons expliciter $\prod_{\sigma\in F(J)}\overline a_{\sigma}$. Posons $\alpha_i\=\alpha_{\sigma_0\circ\varphi^{-i}}$, $\beta_i\=\beta_{\sigma_0\circ\varphi^{-i}}$ et $x_j\=x_{\sigma_0\circ\varphi^{-j}}$, on voit que l'on a aussi $A_j=-\big(\prod_{0\leq i\leq j}\frac{\beta_i}{\alpha_i}\big)x_j$ pour $0\leq j\leq f-2$ et $A_{f-1}=-\lambda^{-1}x_{f-1}$. Un calcul imm\'ediat \`a partir de (\ref{aj}) donne alors pour $j\in \{0,\cdots,f-1\}$ (avec $\overline a_j=\overline a_{\sigma_0\circ\varphi^{-j}}$)~:
\begin{eqnarray}\label{ai}
\left\{\begin{array}{llll}
\overline a_j &=& -x_j\displaystyle{\prod_{0\leq i\leq j}\frac{\beta_i}{\alpha_i}}\displaystyle{\prod_{\stackrel{0\leq i\leq j-1}{\sigma_0\circ\varphi^{-i}\in F(J)}}}\!\!(-\overline a_i^2)&\ \ {\rm si}\ \ \sigma_0\circ\varphi^{-j}\notin F(J)\\
\overline a_j &=& -x_j^{-1}\!\!\!\displaystyle{\prod_{0\leq i\leq j}\frac{\alpha_i}{\beta_i}}\displaystyle{\prod_{\stackrel{0\leq i\leq j-1}{ \sigma_0\circ\varphi^{-i}\in F(J)}}}\!\!(-\overline a_i^{-2})&\ \ {\rm si}\ \ \sigma_0\circ\varphi^{-j}\in F(J).
\end{array}\right.
\end{eqnarray}
Comme dans la preuve de la proposition \ref{prop1}, notons $0\leq j_1< j_2< \cdots < j_{2t-1}< j_{2t}\leq f-1$ les \'el\'ements de $\{0,\cdots,f-1\}$ tels que $\sigma_0\circ\varphi^{-j}\in F(J)$. Nous allons d\'eduire de (\ref{ai}) par r\'ecurrence les formules pour $1\leq s\leq \frac{1}{2}\vert F(J)\vert$ :
\begin{eqnarray}
\label{s-1}\overline a_{j_1}\overline a_{j_2}\cdots \overline a_{j_{2s-1}} &=& (-1)^s\frac{\bigg(\displaystyle{\prod_{0\leq i\leq j_{2s-1}}\frac{\alpha_i}{\beta_i}}\bigg)}{x_{j_{2s-1}}}\frac{\displaystyle{\prod_{i=1}^{s-1}\left[x_{j_{2i}}\bigg(\displaystyle{\prod_{0\leq k\leq j_{2i-1}}\frac{\alpha_k}{\beta_k}}\bigg)\right]}}{\displaystyle{\prod_{i=1}^{s-1}\left[x_{j_{2i-1}}\bigg(\displaystyle{\prod_{0\leq k\leq j_{2i}}\frac{\alpha_k}{\beta_k}}\bigg)\right]}}\\
\label{s}\overline a_{j_1}\overline a_{j_2}\cdots \overline a_{j_{2s}} &=& (-1)^s\frac{\displaystyle{\prod_{i=1}^{s}\left[x_{j_{2i-1}}\bigg(\displaystyle{\prod_{0\leq k\leq j_{2i}}\frac{\alpha_k}{\beta_k}}\bigg)\right]}}{\displaystyle{\prod_{i=1}^{s}\left[x_{j_{2i}}\bigg(\displaystyle{\prod_{0\leq k\leq j_{2i-1}}\frac{\alpha_k}{\beta_k}}\bigg)\right]}}.
\end{eqnarray}
Pour $s=1$, l'\'egalit\'e (\ref{s-1}) est juste la deuxi\`eme \'egalit\'e en (\ref{ai}) pour $j=j_1$. Montrons comment (\ref{s}) se d\'eduit de (\ref{s-1}). La deuxi\`eme \'egalit\'e en (\ref{ai}) pour $j=j_{2s}$ se r\'ecrit :
$$\overline a_{j_{2s}}= x_{j_{2s}}^{-1}\Big(\displaystyle{\prod_{0\leq i\leq j_{2s}}\frac{\alpha_i}{\beta_i}}\Big)\prod_{i=1}^{2s-1}\overline a_{j_i}^{-2}.$$
En multipliant des deux c\^ot\'es par $\overline a_{j_1}\overline a_{j_2}\cdots \overline a_{j_{2s-1}}$ on obtient :
$$\overline a_{j_1}\overline a_{j_2}\cdots \overline a_{j_{2s}}= x_{j_{2s}}^{-1}\Big(\displaystyle{\prod_{0\leq i\leq j_{2s}}\frac{\alpha_i}{\beta_i}}\Big)(\overline a_{j_1}\overline a_{j_2}\cdots \overline a_{j_{2s-1}})^{-1}.$$
En rempla\c cant $\overline a_{j_1}\overline a_{j_2}\cdots \overline a_{j_{2s-1}}$ \`a droite par la valeur donn\'ee en (\ref{s-1}), on obtient exactement (\ref{s}). En utilisant la deuxi\`eme \'egalit\'e en (\ref{ai}) pour $j=j_{2s+1}$, on montre de mani\`ere analogue (\ref{s-1}) pour $s+1$ \`a partir de (\ref{s}) pour $s$. En appliquant (\ref{s}) \`a $s=t$, on a en particulier :
$$\prod_{\sigma\in F(J)}\overline a_{\sigma}= (-1)^{\frac{1}{2}\vert F(J)\vert}\frac{\displaystyle{\prod_{i=1}^{\frac{1}{2}\vert F(J)\vert}}\left[x_{j_{2i-1}}\bigg(\displaystyle{\prod_{0\leq k\leq j_{2i}}\frac{\alpha_k}{\beta_k}}\bigg)\right]}{\displaystyle{\prod_{i=1}^{\frac{1}{2}\vert F(J)\vert}}\left[x_{j_{2i}}\bigg(\displaystyle{\prod_{0\leq k\leq j_{2i-1}}\frac{\alpha_k}{\beta_k}}\bigg)\right]}$$
qui se r\'ecrit, en distinguant les deux cas $\sigma_0\in J$ et $\sigma_0\notin J$ :
\begin{eqnarray*}
\prod_{\sigma\in F(J)}\overline a_{\sigma}&=& (-1)^{\frac{1}{2}\vert F(J)\vert}\bigg(\displaystyle{\prod_{\sigma\notin J}\frac{\alpha_{\sigma}}{\beta_{\sigma}}}\bigg)\frac{\displaystyle{\prod_{\stackrel{\sigma\in J}{\sigma\circ\varphi^{-1}\notin J}}}x_{\sigma}}{\displaystyle{\prod_{\stackrel{\sigma\notin J}{\sigma\circ\varphi^{-1}\in J}}}x_{\sigma}}\ \ {\rm si}\ \ \sigma_0\in J\\
\prod_{\sigma\in F(J)}\overline a_{\sigma}&=& (-1)^{\frac{1}{2}\vert F(J)\vert}\bigg(\displaystyle{\prod_{\sigma\in J}\frac{\alpha_{\sigma}}{\beta_{\sigma}}}\bigg)\frac{\displaystyle{\prod_{\stackrel{\sigma\notin J}{ \sigma\circ\varphi^{-1}\in J}}}x_{\sigma}}{\displaystyle{\prod_{\stackrel{\sigma\in J}{ \sigma\circ\varphi^{-1}\notin J}}}x_{\sigma}}\ \ \ \ {\rm si}\ \ \sigma_0\notin J.
\end{eqnarray*}
Avec (\ref{abc}) on en d\'eduit (\ref{reduction}) puisque $\mu=\prod_{\sigma\in \mathcal S}\alpha_{\sigma}^{-1}$ et $\lambda=\prod_{\sigma\in \mathcal S}\beta_{\sigma}^{-1}$.
\end{proof}

\begin{rem}\label{apartbis}
{\rm Pour $J=\mathcal S$, on a $\overline \alpha'=(\prod_{\sigma\in \mathcal S}\alpha_{\sigma})^{-1}=\mu$ et pour $J=\emptyset$, on a $\overline \alpha'=(\prod_{\sigma\in \mathcal S}\beta_{\sigma})^{-1}=\lambda$.}
\end{rem}

\subsection{S\'eries principales et sommes de Jacobi}\label{jacobi}

On calcule la r\'eduction modulo $p$ de certains invariants associ\'es aux s\'eries principales mod\'er\'ement ra\-mifi\'ees de $\G$ sur $E$ provenant des repr\'esentations de $\g$ issues de certains $\oE$-modules fortement divisibles de type $J$.

On conserve les notations pr\'ec\'edentes. On note $\val$ la valuation $p$-adique (sur une extension finie de $\Qp$) normalis\'ee par $\val(p)=1$ et $\norm\=\frac{1}{q^{\val(\cdot)}}$.

Rappelons d'abord le th\'eor\`eme suivant d\^u \`a Stickelberger qui donne la valuation $p$-adique ainsi que le ``coefficient dominant'' de certaines sommes de Jacobi.

\begin{thm}[\cite{La}]\label{stick}
Soit $a,b$ deux entiers tels que $0<a,b\leq q-1$ et $q-1$ ne divise pas $a+b$. \'Ecrivons :
\begin{eqnarray*}
a&=&\!\sum_{j=0}^{f-1}p^ja_j,\ \ \ b\ =\ \sum_{j=0}^{f-1}p^jb_j\\
a+b&=&\!\sum_{j=0}^{f-1}p^j(a+b)_j + (q-1)Q
\end{eqnarray*}
o\`u $a_j,b_j,(a+b)_j\in \{0,\cdots,p-1\}$ et $Q\geq 0$. Soit $\sigma_0:\Fq\hookrightarrow k_E$ un plongement quelconque. On a dans $\oE$ :
$$\sum_{s\in \Fq}[\sigma_0(s)]^{a}[1-\sigma_0(s)]^{b}=Up^u+C$$
o\`u :
\begin{eqnarray*}
u&=&\frac{1}{p-1}\bigg(\sum_{j=0}^{f-1}p-1-(a_j+b_j-(a+b)_j)\bigg)\in \Z_{\geq 0}\\
U&=&(-1)^{f-1+u}\frac{\prod_{j=0}^{f-1}a_j!b_j!}{\prod_{j=0}^{f-1}(a+b)_j!}\in \Zp^\times
\end{eqnarray*}
et o\`u $\val(C)>u$. 
\end{thm}

Si $\chi,\chi':L^{\times}\rightarrow E^{\times}$ sont deux caract\`eres multiplicatifs localement constants, on note $\chi'\otimes\chi: \B\rightarrow E^{\times}$ le caract\`ere :
\begin{eqnarray*}
\begin{pmatrix} a & b\\ 0 & d\end{pmatrix}\mapsto \chi'(a)\chi(d)
\end{eqnarray*}
et $\Ind_{\B}^{\G}\!\chi'\otimes\chi$ le $E$-espace vectoriel des fonctions localement constantes $f:\G\rightarrow E$ telles que :
$$f(bg)=(\chi'\otimes\chi)(b)f(g)$$
($b\in \B$, $g\in \G$) que l'on munit de l'action \`a gauche de $\G$ par translation \`a droite sur les fonctions.

\begin{thm}\label{main}
Soit $\rhobar$ r\'eductible g\'en\'erique, $J\subseteq {\mathcal S}$ et $\M$ un $\oE$-module fortement divisible de type $J$ tel que $\rhobar\simeq {\rm Hom}_{{\rm Fil}^1,\varphi_1}(\Mbar,\widehat A_{\rm cris}\otimes_{\Zp}\Fp)^{\vee}(1)$. Soit $G$ le groupe $p$-divisible avec donn\'ee de descente correspondant \`a $\M$, $D$ le module de Dieudonn\'e contravariant associ\'e \`a $G$ et $V$ (resp. $V'$) la valeur propre de $\varphi^f$ sur la partie $\eta(J)$-isotypique (resp. $\eta'(J)$-isotypique) de $D$. Consid\'erons la s\'erie principale mod\'er\'ement ramifi\'ee : 
$$\pi_p\=\Ind_{\B}^{\G}\!\eta'(J)\nr(V')\norm \otimes \eta(J)\nr(V)$$
o\`u $\eta(J)$ et $\eta'(J)$ sont vus comme des caract\`eres de $L^{\times}$ en envoyant $p$ et $1+p\oL$ vers $1$ et soit $\widehat v\in \pi_p^{\II}$ non nul sur lequel $\I$ agit par $\eta'(J)\otimes \eta(J)$ (un tel vecteur existe). Si $J\notin \{\emptyset, {\mathcal S}\}$, il existe un unique \'el\'ement $\widehat x(J)\in \oE^{\times}$ tel que dans $\pi_p$ :
$$\sum_{s\in \Fq}\!\bigg(\!\prod_{\sigma\in J}[\sigma(s)]^{p-1-r_{\sigma}}\!\bigg)\!\begin{pmatrix}[s] & 1\\ 1 & 0\end{pmatrix}\begin{pmatrix}0 & 1\\ p & 0\end{pmatrix}\!\widehat v = \widehat x(J)\!\!\sum_{s\in \Fq}\!\bigg(\!\prod_{\sigma\notin J}[\sigma(s)]^{p-1-r_{\sigma}}\!\bigg)\!\begin{pmatrix}[s] & 1\\ 1 & 0\end{pmatrix}\!\widehat v.$$
De plus, la r\'eduction de $\widehat x(J)$ dans $k_E$ est :
\begin{eqnarray}\label{reductionbis}
-\theta(-1)\Big(\prod_{\sigma\in J}\alpha_{\sigma}\prod_{\sigma\notin J}\beta_{\sigma}\Big)^{-1}\frac{\displaystyle{\prod_{\stackrel{\sigma\in J}{ \sigma\circ\varphi^{-1}\notin J}}}\!\!x_{\sigma}(r_{\sigma}+1)}{\displaystyle{\prod_{\stackrel{\sigma\notin J}{ \sigma\circ\varphi^{-1}\in J}}}\!\!x_{\sigma}(r_{\sigma}+1)}
\end{eqnarray}
avec $(\alpha_{\sigma})_{\sigma\in {\mathcal S}}$, $(\beta_{\sigma})_{\sigma\in {\mathcal S}}$ et $(x_{\sigma})_{\sigma\in {\mathcal S}}$ comme en (\ref{norm}).
\end{thm}
\begin{proof}
Notons que la s\'erie principale $\pi_p$ correspond \`a la repr\'esenta\-tion de Weil associ\'ee \`a $D$ muni de la donn\'ee de descente (selon \cite{Fo}) par la correspondance de Langlands locale convenablement normalis\'ee. Un calcul direct et \'el\'ementaire dans $\pi_p$ donne :
$$\begin{pmatrix}0 & 1\\ p & 0\end{pmatrix}\widehat v = \frac{1}{q}V'\sum_{t\in \Fq}\begin{pmatrix}[t] & 1\\ 1 & 0\end{pmatrix}\widehat v$$
ce qui am\`ene \`a calculer $X\=\sum_{s,t\in \Fq}\!\Big(\!\prod_{\sigma\in J}[\sigma(s)]^{p-1-r_{\sigma}}\!\Big)\!\begin{pmatrix}[s] & 1\\ 1 & 0\end{pmatrix}\begin{pmatrix}[t] & 1\\ 1 & 0\end{pmatrix}\widehat v$. On a :
\begin{eqnarray}
\begin{pmatrix}[s] & 1\\ 1 & 0\end{pmatrix}\begin{pmatrix}[t] & 1\\ 1 & 0\end{pmatrix}&=&\begin{pmatrix}1 & [s]\\ 0 & 1\end{pmatrix}\ \ {\rm si}\ \ t= 0\\
\nonumber&=&\begin{pmatrix}[s]+[t^{-1}] & 1\\ 1 & 0\end{pmatrix}\begin{pmatrix}[t] & 1\\ 0 & -[t^{-1}]\end{pmatrix}\ \ {\rm si}\ \ t\ne 0
\end{eqnarray}
d'o\`u :
\begin{multline*}
X=\sum_{s\in \Fq}\!\Big(\!\prod_{\sigma\in J}[\sigma(s)]^{p-1-r_{\sigma}}\!\Big)\!\begin{pmatrix}1 & [s]\\ 0 & 1\end{pmatrix}\widehat v \\
+ \sum_{s\in \Fq,t\in \Fq^{\times}}\!\Big(\!\prod_{\sigma\in J}[\sigma(s)]^{p-1-r_{\sigma}}\!\Big)\!\begin{pmatrix}[s]+[t] & 1\\ 1 & 0\end{pmatrix}\begin{pmatrix}[t^{-1}] & 1\\ 0 & -[t]\end{pmatrix}\widehat v.
\end{multline*}
Comme $\begin{pmatrix}1 & [s]\\ 0 & 1\end{pmatrix}\widehat v=\widehat v$ et comme $\sum_{s\in \Fq}\prod_{\sigma\in J}[\sigma(s)]^{p-1-r_{\sigma}}=0$ (car $J\ne\emptyset$), la premi\`ere somme dans $X$ est nulle. Par (\ref{type}) et la d\'efinition des $c_{\sigma}$ (lemme \ref{csigma}), on a : 
$$\begin{pmatrix}[t^{-1}] & 1\\ 0 & -[t]\end{pmatrix}\widehat v=[\theta(-1)](-1)^{\sum_{\sigma\in J}r_{\sigma}}\Big(\prod_{\sigma\in {\mathcal S}}[\sigma(t)]^{c_{\sigma}}\Big)\widehat v$$
et la deuxi\`eme somme dans $X$ se r\'ecrit avec un changement de variables \'evident (en utilisant encore que $\widehat v$ est invariant sous $\II$) :
\begin{eqnarray*}
X\!\!&=&\!\![\theta(-1)](-1)^{\sum_{\sigma\in J}r_{\sigma}}\!\!\!\sum_{s,t\in \Fq}\!\Big(\!\prod_{\sigma\in J}[\sigma(s)]^{p-1-r_{\sigma}}\!\Big)\!\Big(\!\prod_{\sigma\in {\mathcal S}}[\sigma(t)]^{c_{\sigma}}\!\Big)\begin{pmatrix}[s+t] & 1\\ 1 & 0\end{pmatrix}\widehat v\\
\!\!&=&\!\![\theta(-1)](-1)^{\sum_{\sigma\in J}r_{\sigma}}\!\!\sum_{s\in \Fq}\!\!\Bigg(\!\sum_{t\in \Fq}\Big(\!\prod_{\sigma\in J}[\sigma(t)]^{p-1-r_{\sigma}}\!\Big)\!\Big(\!\prod_{\sigma\in {\mathcal S}}[\sigma(s-t)]^{c_{\sigma}}\!\Big)\Bigg)\begin{pmatrix}[s] & 1\\ 1 & 0\end{pmatrix}\widehat v.
\end{eqnarray*}
Comme $J\ne {\mathcal S}$, on a $\sum_{t\in \Fq}\Big(\!\prod_{\sigma\in J}[\sigma(t)]^{p-1-r_{\sigma}}\!\Big)\!\Big(\!\prod_{\sigma\in {\mathcal S}}[\sigma(-t)]^{c_{\sigma}}\Big)=0$, d'o\`u :
\begin{multline*}
X=[\theta(-1)](-1)^{\sum_{\sigma\in J}r_{\sigma}}\sum_{s\in \Fq}\Bigg[\Big(\prod_{\sigma\in J}[\sigma(s)]^{p-1-r_{\sigma}}\Big)\Big(\prod_{\sigma\in {\mathcal S}}[\sigma(s)]^{c_{\sigma}}\Big)\\
\Bigg(\sum_{t\in \Fq}\Big(\prod_{\sigma\in J}[\sigma(t)]^{p-1-r_{\sigma}}\Big)\Big(\prod_{\sigma\in {\mathcal S}}[\sigma(1-t)]^{c_{\sigma}}\Big)\Bigg)\Bigg]\begin{pmatrix}[s] & 1\\ 1 & 0\end{pmatrix}\widehat v
\end{multline*}
soit en utilisant le lemme \ref{csigma} :
$$X=[\theta(-1)](-1)^{\sum_{\sigma\in J}r_{\sigma}}T\sum_{s\in \Fq}\Big(\prod_{\sigma\notin J}[\sigma(s)]^{p-1-r_{\sigma}}\Big)\begin{pmatrix}[s] & 1\\ 1 & 0\end{pmatrix}\widehat v$$
o\`u $T\=\sum_{t\in \Fq}\Big(\!\prod_{\sigma\in J}[\sigma(t)]^{p-1-r_{\sigma}}\!\Big)\!\Big(\!\prod_{\sigma\in {\mathcal S}}[\sigma(1-t)]^{c_{\sigma}}\!\Big)$. On a donc :
\begin{equation}\label{xhat}
\widehat x(J)=[\theta(-1)](-1)^{\sum_{\sigma\in J}r_{\sigma}}\frac{1}{p^f}V'T.
\end{equation}
Comme $J\notin \{\emptyset,{\mathcal S}\}$, \ on \ peut \ appliquer \ le \ th\'eor\`eme \ \ref{stick} \ \`a \ $T$ (avec $a\=\sum_{\sigma_0\circ\varphi^j\in J}p^j(p-1-r_{\sigma_0\circ\varphi^j})$ et $b\=\sum_{\sigma_0\circ\varphi^j\in {\mathcal S}}p^jc_{\sigma_0\circ\varphi^j}$). Un calcul \'el\'ementaire utilisant le lemme \ref{csigma} donne :
$$\begin{array}{llcll}
(a+b)_{j}&=&0&{\rm si}&\sigma_0\circ\varphi^j\in J\\
(a+b)_{j}&=&p-1-r_{\sigma_0\circ\varphi^j}&{\rm si}&\sigma_0\circ\varphi^j\notin J
\end{array}$$
d'o\`u on d\'eduit :
$$a_j+b_j-(a+b)_j=\left\{\begin{array}{llcll}
0&{\rm si}&\sigma_0\circ\varphi^j\notin J&{\rm et}&\sigma_0\circ\varphi^{j-1}\notin J\\
-1&{\rm si}&\sigma_0\circ\varphi^j\notin J&{\rm et}&\sigma_0\circ\varphi^{j-1}\in J\\
p-1&{\rm si}&\sigma_0\circ\varphi^j\in J&{\rm et}&\sigma_0\circ\varphi^{j-1}\in J\\
p&{\rm si}&\sigma_0\circ\varphi^j\in J&{\rm et}&\sigma_0\circ\varphi^{j-1}\notin J.
\end{array}\right.$$
On a ainsi :
\begin{eqnarray*}
u=\frac{1}{p-1}\Bigg(\sum_{\stackrel{\sigma_0\circ\varphi^j\notin J}{\sigma_0\circ\varphi^{j-1}\notin J}}\!\!\!p-1\ +\!\!\!\sum_{\stackrel{\sigma_0\circ\varphi^j\notin J}{\sigma_0\circ\varphi^{j-1}\in J}}\!\!\!p\ -\!\!\!\sum_{\stackrel{\sigma_0\circ\varphi^j\in J}{\sigma_0\circ\varphi^{j-1}\notin J}}\!\!\!1\Bigg)=\vert {\mathcal S}\backslash J\vert
\end{eqnarray*}
et :
\begin{eqnarray*}
U=(-1)^{\vert J\vert +1}\ \frac{\displaystyle{\prod_{\sigma\in J}}(p-1-r_{\sigma})!\displaystyle{\prod_{\sigma\in {\mathcal S}}}c_{\sigma}!}{\displaystyle{\prod_{\sigma\notin J}}(p-1-r_{\sigma})!}.
\end{eqnarray*}
En utilisant $n!(p-1-n)!\equiv (-1)^{n-1}$ modulo $p$ si $n\in \{0,\cdots,p-1\}$ et le lemme \ref{csigma} on obtient :
$$U\equiv (-1)^{1+\frac{\vert F(J)\vert}{2}+\sum_{\sigma\in J}r_{\sigma}}\frac{\displaystyle\prod_{\stackrel{\sigma\in J}{ \sigma\circ\varphi^{-1}\notin J}}r_{\sigma}+1}{\displaystyle\prod_{\stackrel{\sigma\notin J}{ \sigma\circ\varphi^{-1}\in J}}r_{\sigma}+1}\ \ {\rm modulo}\ \ p.$$
Par ailleurs les valeurs propres de $\varphi^f$ sur $D$ sont (avec les notations du \S\ \ref{fl1}) $V=p^{\vert I_{\eta(J)}\vert}\alpha=p^{\vert {\mathcal S}\backslash J\vert}\alpha$ et $V'=p^{\vert I_{\eta'(J)}\vert}\alpha'=p^{\vert J\vert}\alpha'$ (cf. remarque \ref{plus}(ii)). On a donc avec ce qui pr\'ec\`ede et (\ref{xhat}) :
\begin{eqnarray*}
\widehat x(J)&=&[\theta(-1)](-1)^{\sum_{\sigma\in J}r_{\sigma}}\frac{1}{p^f}p^{\vert J\vert}\alpha'(-1)^{1+\frac{\vert F(J)\vert}{2}+\sum_{\sigma\in J}r_{\sigma}}\frac{\displaystyle\prod_{\stackrel{\sigma\in J}{ \sigma\circ\varphi^{-1}\notin J}}r_{\sigma}+1}{\displaystyle\prod_{\stackrel{\sigma\notin J}{ \sigma\circ\varphi^{-1}\in J}}r_{\sigma}+1}p^{\vert {\mathcal S}\backslash J\vert} + \delta\\
&=&[\theta(-1)]\alpha'(-1)^{1+\frac{\vert F(J)\vert}{2}}\frac{\displaystyle\prod_{\stackrel{\sigma\in J}{ \sigma\circ\varphi^{-1}\notin J}}r_{\sigma}+1}{\displaystyle\prod_{\stackrel{\sigma\notin J}{ \sigma\circ\varphi^{-1}\in J}}r_{\sigma}+1}+\delta
\end{eqnarray*}
o\`u $\delta\in \oE$ a une valuation strictement positive. On obtient alors l'\'egalit\'e (\ref{reductionbis}) avec (\ref{reduction}). 
\end{proof}

\begin{rem}\label{apart}
{\rm Si $J={\mathcal S}$, un calcul analogue \`a celui de la preuve du th\'eor\`eme \ref{main} donne (rappelons que $V=p^{\vert {\mathcal S}\backslash J\vert}\alpha$ et $V'=p^{\vert J\vert}\alpha'$) :
\begin{multline*}
\sum_{s\in \Fq}\bigg(\prod_{\sigma\in {\mathcal S}}[\sigma(s)]^{p-1-r_{\sigma}}\bigg)\begin{pmatrix}[s] & 1\\ 1 & 0\end{pmatrix}\begin{pmatrix}0 & 1\\ p & 0\end{pmatrix}\widehat v=-[\theta(-1)]\alpha'\sum_{s\in \Fq}\!\begin{pmatrix}[s] & 1\\ 1 & 0\end{pmatrix}\!\widehat v+\\
[\theta(-1)]q\alpha'\begin{pmatrix}0 & 1\\ 1 & 0\end{pmatrix}\widehat v.
\end{multline*}
On en d\'eduit l'\'equation pour $J=\emptyset$ en rempla\c cant $\smat{0 & 1\\ p & 0}\widehat v$ par $\widehat v$ et en utilisant $\smat{p & 0\\ 0 & p}\widehat v=\alpha\alpha'\widehat v$ et la remarque \ref{plus}(iv) :
\begin{multline*}
\sum_{s\in \Fq}\begin{pmatrix}[s] & 1\\ 1 & 0\end{pmatrix}\begin{pmatrix}0 & 1\\ p & 0\end{pmatrix}\widehat v=-[\theta(-1)]\alpha'\sum_{s\in \Fq}\!\bigg(\prod_{\sigma\in {\mathcal S}}[\sigma(s)]^{p-1-r_{\sigma}}\bigg)\begin{pmatrix}[s] & 1\\ 1 & 0\end{pmatrix}\widehat v+\\
q\begin{pmatrix}p & 0\\ 0 & 1\end{pmatrix}\widehat v.
\end{multline*}}
\end{rem}

\subsection{Valeurs sp\'eciales de param\`etres}\label{spec}

On d\'efinit des param\`etres $x(J)$ sur certaines repr\'esentations lisses de $\G$ sur $k_E$ et on montre comment les r\'esultats des parties pr\'ec\'edentes permettent de ``distinguer'' des valeurs sp\'eciales de ces param\`etres.

On conserve les notations des sections pr\'ec\'edentes, en particulier on fixe une repr\'esentation $\rhobar$ r\'eductible comme au \S\ \ref{prel}. Pour $J\subseteq {\mathcal S}$, on note $\tau(J)$ l'unique poids de Serre tel que l'action de $\I$ sur $\tau(J)^{\II}$ est donn\'ee par le caract\`ere $\overline\eta'(J)\otimes\overline\eta(J)$ (cf. (\ref{type}), notons que $\tau(J)$ d\'epend aussi de $\rhobar$). Un tel poids est unique car $\overline\eta'(J)\ne\overline\eta(J)$. On a par exemple $\tau(\emptyset)=(\otimes_{\sigma\in {\mathcal S}}(\Sym^{r_{\sigma}}k_E^2)^{\sigma})\otimes \theta\circ\det\in {\mathcal D}(\rhobar)$ et $\tau({\mathcal S})=(\otimes_{\sigma\in {\mathcal S}}(\Sym^{p-1-r_{\sigma}}k_E^2)^{\sigma})\otimes \prod_{\sigma\in {\mathcal S}}\sigma\circ\det^{r_{\sigma}}\!\otimes \theta\circ\det$. De plus $\tau({\mathcal S}\backslash J)$ est le socle de la repr\'esentation $\ind_{\I}^{\K}\overline\eta'(J)\otimes\overline\eta(J)$ du lemme \ref{unik}, $\tau(J)$ est son co-socle (\cite[\S\ 2]{BP}) et $\tau(\emptyset)$ un de ses constituants (voir la preuve du lemme \ref{unik} ou celle du lemme \ref{const}(ii) ci-dessous).

La proposition suivante d\'efinit abstraitement les param\`etres qui sont au coeur de cet article.

\begin{prop}\label{xj}
Soit $J\subseteq {\mathcal S}$ et $\pi$ une repr\'esentation lisse de $\G$ sur $k_E$ avec un caract\`ere central et v\'erifiant les propri\'et\'es suivantes :\\
(i) $\tau(\emptyset)$ appara\^\i t dans le $\K$-socle de $\pi$ avec multiplicit\'e $1$\\
(ii) $\tau(\emptyset)$ est le seul constituant commun \`a $\ind_{\I}^{\K}\overline\eta'(J)\otimes\overline\eta(J)$ et au $\K$-socle de $\pi$\\
(iii) $\pi$ contient l'unique quotient de $\ind_{\I}^{\K}\overline\eta'(J)\otimes\overline\eta(J)$ de socle $\tau(\emptyset)$.\\
Alors il existe (\`a multiplication par un scalaire non nul pr\`es) un unique vecteur $v\in \pi^{\II}$ non nul sur lequel $\I$ agisse par $\overline\eta'(J)\otimes\overline\eta(J)$ et un unique \'el\'ement $x(J)\in k_E^{\times}$ tel que l'on ait l'\'egalit\'e dans $\pi$ :
$$\sum_{s\in \Fq}\bigg(\prod_{\sigma\in J}\sigma(s)^{p-1-r_{\sigma}}\!\bigg)\!\begin{pmatrix}[s] & 1\\ 1 & 0\end{pmatrix}\begin{pmatrix}0 & 1\\ p & 0\end{pmatrix}v = x(J)\sum_{s\in \Fq}\bigg(\prod_{\sigma\notin J}\sigma(s)^{p-1-r_{\sigma}}\!\bigg)\!\begin{pmatrix}[s] & 1\\ 1 & 0\end{pmatrix}v.$$
\end{prop}
\begin{proof}
Par (i), (ii) et (iii), on a un unique (\`a homoth\'etie pr\`es) morphisme non nul $\K$-\'equivariant : 
\begin{eqnarray}\label{unimor}
\ind_{\I}^{\K}\overline\eta'(J)\otimes\overline\eta(J)\longrightarrow \pi
\end{eqnarray}
et ce morphisme se factorise par l'unique quotient de l'induite de socle $\tau(\emptyset)$. Il envoie de plus l'unique vecteur non nul de $\big(\ind_{\I}^{\K}\overline\eta'(J)\otimes\overline\eta(J)\big)^{\II}$ sur lequel $\I$ agit par $\overline\eta'(J)\otimes\overline\eta(J)$ vers un vecteur non nul $v\in \pi$ avec la m\^eme propri\'et\'e. L'unicit\'e d'un $v\in \pi$ comme dans l'\'enonc\'e r\'esulte par r\'eciprocit\'e de Frobenius de l'unicit\'e du morphisme (\ref{unimor}). Comme $\tau(\emptyset)$ est en ``position ${\mathcal S}\backslash J$'' dans l'induite $\ind_{\I}^{\K}\overline\eta'(J)\otimes\overline\eta(J)$, par \cite[Lem.2.7]{BP} le vecteur :
$$\sum_{s\in \Fq}\Big(\prod_{\sigma\in {\mathcal S}\backslash J}\sigma(s)^{p-1-r_{\sigma}}\!\Big)\smat{[s] & 1\\ 1 & 0}v$$
est un g\'en\'erateur de $\tau(\emptyset)^{\II}$ dans $\pi$. Com\-me $\I$ agit sur $\smat{0 & 1\\ p & 0}v\in \pi^{\II}$ par $\overline\eta(J)\otimes\overline\eta'(J)$, on en d\'eduit par r\'eciprocit\'e de Frobenius une surjection canonique :
$$\ind_{\I}^{\K}\overline\eta(J)\otimes\overline\eta'(J)\twoheadrightarrow \left\langle \K\smat{0 & 1\\ p & 0}v\right\rangle\subset \pi$$
vers la sous-$\K$-repr\'esentation de $\pi$ engendr\'ee par $\smat{0 & 1\\ p & 0}v$. Comme les constituants de $\ind_{\I}^{\K}\overline\eta'(J)\otimes\overline\eta(J)$ et de $\ind_{\I}^{\K}\overline\eta(J)\otimes\overline\eta'(J)$ sont les m\^emes (mais dans l'``ordre inverse''), par (ii) cette surjection se factorise n\'ecessairement par l'unique quotient de l'induite de socle $\tau(\emptyset)$. Par (i), ce socle est aussi celui de l'image du morphisme (\ref{unimor}). Par le m\^eme raisonnement que pr\'ec\'edemment (en rempla\c cant $\ind_{\I}^{\K}\overline\eta'(J)\otimes\overline\eta(J)$ par $\ind_{\I}^{\K}\overline\eta(J)\otimes\overline\eta'(J)$ et $v$ par $\smat{0 & 1\\ p & 0}v$) le vecteur :
$$\sum_{s\in \Fq}\Big(\prod_{\sigma\in J}\sigma(s)^{p-1-r_{\sigma}}\!\Big)\smat{[s] & 1\\ 1 & 0}\smat{0 & 1\\ p & 0}v$$
est donc aussi un g\'en\'erateur de $\tau(\emptyset)^{\II}$ dans $\pi$. Comme $\dim_{k_E}\tau(\emptyset)^{\II}=1$, il doit exister un scalaire non nul $x(J)$ comme dans l'\'enonc\'e.
\end{proof}

Notons que si $J$ et $\pi$ v\'erifient les hypoth\`eses de la proposition \ref{xj}, alors ${\mathcal S}\backslash J$ et $\pi$ les v\'erifient aussi et on a la relation $x(J)x({\mathcal S}\backslash J)=\omega_{\pi}(p)$ o\`u $\omega_{\pi}$ est le caract\`ere central de $\pi$ (cela vient du fait que $\smat{0 & 1\\p & 0}\smat{0 & 1\\p & 0}$ agit par la multiplication par $\omega_{\pi}(p)$).

Bien que nous n'en ayons pas strictement besoin pour les applications globales, il est \'eclairant de replacer la proposition \ref{xj} dans le contexte de \cite{BP}. 

On rappelle que $\rhobar$ est r\'eductible g\'en\'erique et que ${\mathcal D}(\rhobar)$ d\'esigne l'ensemble de ses poids de Serre (cf. \S\ \ref{prel}). Soit $D(\rhobar)$ la repr\'esentation maximale (pour l'inclusion) de $\GL_2(\Fq)$ sur $k_E$ v\'erifiant les deux conditions :\\
(i) le socle de $D(\rhobar)$ est $\oplus_{\tau\in {\mathcal D}(\rhobar)}\tau$\\
(ii) les constituants de ce socle n'apparaissent pas ailleurs dans $D(\rhobar)$.\\
On peut montrer qu'une telle repr\'esentation existe (cf. \cite[Prop.13.1]{BP}). De plus, $D(\rhobar)$ est alors de la forme $\oplus_{\tau\in {\mathcal D}(\rhobar)}D_{\tau}(\rhobar)$ o\`u $D_{\tau}(\rhobar)$ est l'unique facteur direct de $D(\rhobar)$ de socle $\tau$ et les caract\`eres de $\I$ qui apparaissent sur $D(\rhobar)^{\II}$ sont tous distincts (i.e. apparaissent avec multiplicit\'e $1$).

\begin{lem}\label{const}
Soit $J\subseteq {\mathcal S}$.\\
(i) Le poids de Serre $\tau(J)$ est un constituant de $D(\rhobar)$ (rappelons que $\tau(J)$ d\'epend de $J$ {\it et} $\rhobar$).\\
(ii) Le poids de Serre $\tau(J)$ est un constituant de $D_{\tau(\emptyset)}(\rhobar)$ si et seulement si $Z(\rhobar)\cap \{\sigma\in {\mathcal S}, \sigma\notin J, \sigma\circ\varphi^{-1}\in J\}=\emptyset$.
\end{lem}
\begin{proof}
(i) Comme $\ind_{\I}^{\K}\overline\eta'(J)\otimes\overline\eta(J)$ a certains de ses constituants dans ${\mathcal D}(\rhobar)$ et comme $\tau(J)$ est son co-socle, on en d\'eduit qu'elle poss\`ede un quotient de socle un poids de Serre dans ${\mathcal D}(\rhobar)$, de co-socle $\tau(J)$ et dont aucun constituant \`a part le socle n'est dans ${\mathcal D}(\rhobar)$. Par maximalit\'e de $D(\rhobar)$, on en d\'eduit que ce quotient est un sous-objet de $D(\rhobar)$, d'o\`u (i).\\
(ii) On fixe $\sigma_0\in {\mathcal S}$ et on reprend les notations de la preuve du lemme \ref{unik}. Les constituants de $\ind_{\I}^{\K}\overline\eta'(J)\otimes\overline\eta(J)$ qui sont dans ${\mathcal D}(\rhobar)$ sont les poids de Serre index\'es par les parties de ${\mathcal S}$ contenant $\{\sigma_0\circ\varphi^j,j\in J^{\rm min}\}$ et contenues dans $\{\sigma_0\circ\varphi^j,j\in J^{\rm max}\}$ o\`u $J^{\rm min}$ et $J^{\rm max}$ sont comme suit (cf. \cite[Prop.4.3]{Br} et les \'egalit\'es (19) de la preuve de {\it loc.cit.}) :
\begin{eqnarray*}
J^{\rm min}\!\!&=&\!\!\{j,\lambda_{j+1}(x_{j+1})=p-1-x_{j+1}\ {\rm ou}\ \big(\lambda_{j+1}(x_{j+1})=x_{j+1}+1\\
&&\ \ \ \ \ \ \ \ \ \ \ \ \ \ \ \ \ \ \ \ \ \ \ \ \ \ \ \ \ \ \ \ \ \ \ \ \ \ \ \ \ \ \ \ \ \ \ \ \ \ \ \ \ \ \ \ \ \ \ \ \ \ \ {\rm et}\ \sigma_0\circ \varphi^{j+1}\notin Z(\rhobar)\big)\}\\
J^{\rm max}\!\!&=&\!\!\{j,\lambda_{j+1}(x_{j+1})\in \{p-1-x_{j+1}, x_{j+1}+1\}\ {\rm ou}\ \big(\lambda_{j+1}(x_{j+1})=p-2-x_{j+1}\\
&&\ \ \ \ \ \ \ \ \ \ \ \ \ \ \ \ \ \ \ \ \ \ \ \ \ \ \ \ \ \ \ \ \ \ \ \ \ \ \ \ \ \ \ \ \ \ \ \ \ \ \ \ \ \ \ \ \ \ \ \ \ \ \ {\rm et}\ \sigma_0\circ \varphi^{j+1}\in Z(\rhobar)\big)\}
\end{eqnarray*}
avec $(\lambda_j(x_j))_{j\in \{0,\cdots,f-1\}}$ comme en (\ref{lambda}). En particulier, on a toujours (cf. fin de la preuve du lemme \ref{unik}) :
$$\{\sigma_0\circ\varphi^j,j\in J^{\rm min}\}\subseteq {\mathcal S}\backslash J\subseteq \{\sigma_0\circ\varphi^j,j\in J^{\rm max}\}.$$
De plus $\tau(\emptyset)$ est index\'e par ${\mathcal S}\backslash J$ dans $\ind_{\I}^{\K}\overline\eta'(J)\otimes\overline\eta(J)$. Par construction de $D(\rhobar)$ et par \cite[Prop.4.3]{Br} (avec \cite[Thm.2.4]{BP}), on voit donc que $\tau(J)$ est un constituant de $D_{\tau(\emptyset)}(\rhobar)$ si et seulement si ${\mathcal S}\backslash J=\{\sigma_0\circ\varphi^j,j\in J^{\rm max}\}$. En proc\'edant comme dans la preuve du lemme \ref{unik}, cela est \'equivalent \`a $\sigma_0\circ \varphi^{j+1}\notin Z(\rhobar)$ si $\lambda_{j+1}(x_{j+1})=p-2-x_{j+1}$, c'est-\`a-dire si $\sigma_0\circ \varphi^{j+1}\notin J$ et $\sigma_0\circ \varphi^{j}\in J$ par (\ref{lambda}). C'est donc \'equivalent \`a $Z(\rhobar)\cap \{\sigma\in {\mathcal S}, \sigma\notin J, \sigma\circ\varphi^{-1}\in J\}=\emptyset$. 
\end{proof}

La proposition suivante n'est pas utilis\'ee dans la suite, mais donne une variante plus pr\'ecise de la proposition \ref{xj}.

\begin{prop}\label{xjbis}
Soit $\pi$ une repr\'esentation lisse de $\G$ sur $k_E$ avec un caract\`ere central et v\'eri\-fiant les propri\'et\'es suivantes :\\
(i) $\pi$ contient $D(\rhobar)$\\
(ii) le $\K$-socle de $\pi$ est celui de $D(\rhobar)$\\
(iii) $D(\rhobar)^{\II}$ est stable par $\smat{0 & 1\\ p & 0}$ dans $\pi$.\\
Soit $J\subseteq {\mathcal S}$ tel que $Z(\rhobar)\cap \{\sigma\in {\mathcal S}, \sigma\notin J, \sigma\circ\varphi^{-1}\in J\}=\emptyset$. Alors il existe (\`a multiplication par un scalaire non nul pr\`es) un unique vecteur $v\in D_{\tau(\emptyset)}(\rhobar)^{\II}\subseteq \pi^{\II}$ non nul sur lequel $\I$ agisse par $\overline\eta'(J)\otimes\overline\eta(J)$. De plus, si $Z(\rhobar)\cap \{\sigma\in {\mathcal S}, \sigma\in J, \sigma\circ\varphi^{-1}\notin J\}=\emptyset$, alors il existe un unique \'el\'ement $x(J)\in k_E^{\times}$ tel que l'on ait l'\'egalit\'e de la proposition \ref{xj}. Si $Z(\rhobar)\cap \{\sigma\in {\mathcal S}, \sigma\in J, \sigma\circ\varphi^{-1}\notin J\}\ne \emptyset$, alors on a dans $\pi$ :
$$\sum_{s\in \Fq}\bigg(\prod_{\sigma\in J}\sigma(s)^{p-1-r_{\sigma}}\!\bigg)\begin{pmatrix}[s] & 1\\ 1 & 0\end{pmatrix}\begin{pmatrix}0 & 1\\ p & 0\end{pmatrix}v=0.$$
\end{prop}
\begin{proof}
Notons que les hypoth\`eses (i) et (ii) et la d\'efinition de $D(\rhobar)$ font que $\Hom_{\K}(D_\tau(\rhobar),\pi)=k_E$ pour tout $\tau\in {\mathcal D}(\rhobar)$ de sorte que $\pi$ contient chaque $D_\tau(\rhobar)$ de mani\`ere unique. Par le lemme \ref{const}, $\tau(J)$ est un constituant de $D_{\tau(\emptyset)}(\rhobar)$ et les deux cas correspondent respectivement \`a $\tau({\mathcal S}\backslash J)$ est un constituant de $D_{\tau(\emptyset)}(\rhobar)$ et $\tau({\mathcal S}\backslash J)$ n'est pas un constituant de $D_{\tau(\emptyset)}(\rhobar)$. On ne d\'emontre que le deuxi\`eme cas, la preuve du premier \'etant analogue \`a celle de la proposition \ref{xj}. Notons que par le lemme \ref{const}(ii), on a alors $J\notin \{\emptyset,{\mathcal S}\}$. Par les hypoth\`eses et par les d\'efinition et structure de $D(\rhobar)$, on a $\smat{0 & 1\\ p & 0}v\in D_{\tau}(\rhobar)$ pour un $\tau\in {\mathcal D}(\rhobar)$ qui n'est pas $\tau(\emptyset)$. La sous-$\K$-repr\'esentation $\left\langle \K\smat{0 & 1\\ p & 0}v\right\rangle$ de $D(\rhobar)$ engendr\'ee par $\smat{0 & 1\\ p & 0}v$ est contenue dans $D_{\tau}(\rhobar)$ et en particulier n'a donc pas $\tau(\emptyset)$ dans ses constituants. Consid\'erons la surjection canonique : $$\ind_{\I}^{\K}\overline\eta(J)\otimes\overline\eta'(J)\twoheadrightarrow \left\langle \K\smat{0 & 1\\ p & 0}v\right\rangle$$
donn\'ee par r\'eciprocit\'e de Frobenius. Dans cette surjection, le constituant $\tau(\emptyset)$ de $\ind_{\I}^{\K}\overline\eta(J)\otimes\overline\eta'(J)$ est n\'ecessairement envoy\'e vers $0$ \`a droite puisqu'il n'y appara\^\i t plus comme constituant. Par \cite[Lem.2.7]{BP}, le vecteur :
$$\sum_{s\in \Fq}\Big(\prod_{\sigma\in J}\sigma(s)^{p-1-r_{\sigma}}\!\Big)\begin{pmatrix}[s] & 1\\ 1 & 0\end{pmatrix}\begin{pmatrix}0 & 1\\ p & 0\end{pmatrix}v\in \left\langle \K\smat{0 & 1\\ p & 0}v\right\rangle\subset \pi,$$ 
s'il est non nul, est tel que ${\rm T}(\oL)$ (le tore de $\I$) agit sur lui comme sur $\tau(\emptyset)^{\II}$. Or le caract\`ere correspondant de ${\rm T}(\oL)$ n'appara\^\i t pas dans $\left\langle \K\smat{0 & 1\\ p & 0}v\right\rangle$ car aucun autre constituant que $\tau(\emptyset)$ ne peut le donner (utiliser $J\notin \{\emptyset,{\mathcal S}\}$ et \cite[lem.2.5]{BP}). Ce vecteur est donc nul.
\end{proof}

\begin{rem}
{\rm (i) Soit $\pi$ comme dans la proposition \ref{xjbis} et $J\subseteq {\mathcal S}$ tel que $Z(\rhobar)\cap F(J)=\emptyset$. Alors par le lemme \ref{unik} (et la proposition \ref{xjbis}) le couple $(J,\pi)$ v\'erifie les hypoth\`eses de la proposition \ref{xj}, et en particulier il existe un unique $v\in \pi^{\II}$ non nul (forc\'ement dans $D_{\tau(\emptyset)}(\rhobar)^{\II}$ par la proposition \ref{xjbis}) sur lequel $\I$ agisse par $\overline\eta'(J)\otimes\overline\eta(J)$. Si $Z(\rhobar)\cap \{\sigma\in {\mathcal S}, \sigma\notin J, \sigma\circ\varphi^{-1}\in J\}=\emptyset$ mais $Z(\rhobar)\cap \{\sigma\in {\mathcal S}, \sigma\in J, \sigma\circ\varphi^{-1}\notin J\}\ne \emptyset$, un tel vecteur $v$ n'est pas forc\'ement unique, i.e. il peut exister $w\in \pi^{\II}\backslash D_{\tau(\emptyset)}(\rhobar)^{\II}$ (non nul) sur lequel $\I$ agit par $\overline\eta'(J)\otimes\overline\eta(J)$. Cela vient du fait que $\ind_{\I}^{\K}\overline\eta'(J)\otimes\overline\eta(J)$ poss\`ede alors plusieurs constituants dans ${\mathcal D}(\rhobar)$ (e.g. les poids de Serre $\tau(\emptyset)$ et $\tau$, cf. preuve pr\'ec\'edente).\\
(ii) Lorsque $Z(\rhobar)=\emptyset$, i.e. lorsque ${\mathcal D}(\rhobar)=\{\tau(\emptyset)\}$, les $x(J)$ (pour tout $J\subseteq {\mathcal S}$) sont les seuls invariants que l'on peut d\'eduire de \cite{BP} (en proc\'edant comme dans la Proposition \ref{xjbis}) car ils d\'eterminent dans ce cas compl\`etement l'action de $\smat{0 & 1\\ p & 0}$ sur $D(\rhobar)^{\II}=D_{\tau(\emptyset)}(\rhobar)^{\II}$. Pour plus de d\'etails sur les invariants que l'on peut d\'efinir de mani\`ere syst\'ematique \`a partir des constructions de \cite{BP}, nous renvoyons le lecteur \`a \cite{Hu}.}
\end{rem}

Par \cite[Thm.1.2]{BP}, on peut trouver des repr\'esentations $\pi$ v\'erifiant les hypoth\`eses de la proposition \ref{xjbis} (et donc de la proposition \ref{xj} pour $J$ tel que $Z(\rhobar)\cap F(J)=\emptyset$) et avec des valeurs {\it presque quelconques} pour les $x(J)$. Plus pr\'ecis\'ement, pour tout $\nu\in k_E^{\times}$ qui est un carr\'e dans $k_E$ et tout uplet $(x(J))_J$ d'\'el\'ements de $k_E$ (o\`u $J$ parcourt les parties de ${\mathcal S}$ v\'erifiant $Z(\rhobar)\cap \{\sigma\in {\mathcal S}, \sigma\notin J, \sigma\circ\varphi^{-1}\in J\}=\emptyset$) qui est tel que :
$$\begin{array}{lll}
x(J)=0&\ {\rm si}\ &Z(\rhobar)\cap \{\sigma\in {\mathcal S}, \sigma\in J, \sigma\circ\varphi^{-1}\notin J\}\ne\emptyset\\
x(J)=\nu x({\mathcal S}\backslash J)^{-1}&\ {\rm si}\ &Z(\rhobar)\cap \{\sigma\in {\mathcal S}, \sigma\in J, \sigma\circ\varphi^{-1}\notin J\}=\emptyset
\end{array}$$
(notons que la condition dans le deuxi\`eme cas est \'equivalente \`a $Z(\rhobar)\cap F(J)=\emptyset$), alors il existe (au moins) une repr\'esentation $\pi$ lisse admissible avec un caract\`ere central $\omega_{\pi}$ tel que $\omega_{\pi}(p)=\nu$, qui v\'erifie toutes les propri\'et\'es de la proposition \ref{xjbis} et telle que pour tout $J$ v\'erifiant $Z(\rhobar)\cap \{\sigma\in {\mathcal S}, \sigma\notin J, \sigma\circ\varphi^{-1}\in J\}=\emptyset$ :
$$\sum_{s\in \Fq}\bigg(\prod_{\sigma\in J}\sigma(s)^{p-1-r_{\sigma}}\!\bigg)\!\begin{pmatrix}[s] & 1\\ 1 & 0\end{pmatrix}\!\begin{pmatrix}0 & 1\\ p & 0\end{pmatrix}v = x(J)\sum_{s\in \Fq}\bigg(\prod_{\sigma\notin J}\sigma(s)^{p-1-r_{\sigma}}\!\bigg)\!\begin{pmatrix}[s] & 1\\ 1 & 0\end{pmatrix}v.$$

Le corollaire suivant, qui r\'esume l'essentiel des r\'esultats locaux qui pr\'ec\`edent, montre que sous certaines conditions sur $\pi$ (qui seront v\'erifi\'ees dans un cadre global), les scalaires $x(J)$ {\it ne} peuvent {\it pas} prendre n'importe quelles valeurs.

\begin{cor}\label{mainlocal}
Soit $\rhobar$ r\'eductible g\'en\'erique et $J\subseteq {\mathcal S}$ tel que $Z(\rhobar)\cap F(J)=\emptyset$. Soit $\pi$ une repr\'esentation lisse de $\G$ sur $k_E$ telle que $(J,\pi)$ v\'erifie les hypoth\`eses de la proposition \ref{xj} de sorte que le scalaire $x(J)\in k_E^{\times}$ est d\'efini. Soit $\M$ un $\oE$-module fortement divisible de type $\eta(J)\otimes\eta'(J)$ tel que $\rhobar\simeq {\rm Hom}_{{\rm Fil}^1,\varphi_1}(\Mbar,\widehat A_{\rm cris}\otimes_{\Zp}\Fp)^{\vee}(1)$ et d\'efinissons $\pi_p$ et $\widehat v$ comme dans le th\'eor\`eme \ref{main}. On suppose qu'il existe un $\oE$-r\'eseau stable $\pi_p^0$ dans $\pi_p$ contenant $\widehat v$ ainsi qu'un morphisme $\oE$-lin\'eaire $\G$-\'equivariant $\pi_p^0\rightarrow \pi$ tel que l'image de $\widehat v$ dans $\pi$ est non nulle. Alors on a :
$$x(J)=-\theta(-1)\Big(\prod_{\sigma\in J}\alpha_{\sigma}\prod_{\sigma\notin J}\beta_{\sigma}\Big)^{-1}\frac{\displaystyle{\prod_{\stackrel{\sigma\in J}{ \sigma\circ\varphi^{-1}\notin J}}}\!\!x_{\sigma}(r_{\sigma}+1)}{\displaystyle{\prod_{\stackrel{\sigma\notin J}{ \sigma\circ\varphi^{-1}\in J}}}\!\!x_{\sigma}(r_{\sigma}+1)}$$
avec $(\alpha_{\sigma})_{\sigma\in {\mathcal S}}$, $(\beta_{\sigma})_{\sigma\in {\mathcal S}}$ et $(x_{\sigma})_{\sigma\in {\mathcal S}}$ comme en (\ref{norm}).
\end{cor}
\begin{proof}
Cela d\'ecoule de la proposition \ref{J'}, du th\'eor\`eme \ref{main} avec les remarques \ref{apart} et \ref{apartbis}, et de la proposition \ref{xj}.
\end{proof}

\begin{rem}
{\rm (i) On voit que le corollaire \ref{mainlocal} appliqu\'e avec $J$ et ${\mathcal S}\backslash J$ pour $\rhobar$ et $\pi$ fix\'ees (rappelons que $Z(\rhobar)\cap F(J)=\emptyset$ est \'equivalent \`a $Z(\rhobar)\cap F({\mathcal S}\backslash J)=\emptyset$) redonne bien $x(J)x({\mathcal S}\backslash J)=(\omega^{-1}\det\rhobar)(p)=(\det\rhobar)(p)$ comme attendu.\\
(ii) Lorsque $Z(\rhobar)=\emptyset$, on peut r\'ecrire la formule ci-dessus pour $x(J)$ sous la forme plus simple :
$$x(J)=-\theta(-1)\Big(\prod_{\sigma\in J}\alpha_{\sigma}\prod_{\sigma\notin J}\beta_{\sigma}\Big)^{-1}\prod_{\sigma\in J}\frac{x_{\sigma}(r_{\sigma}+1)}{x_{\sigma\circ\varphi}(r_{\sigma\circ\varphi}+1)}.$$}
\end{rem}

\section{R\'esultats globaux}

\subsection{Quelques pr\'eliminaires}\label{globalprelim}

On commence par quelques pr\'eliminaires, notations et d\'efinitions.

Soit $F$ un corps totalement r\'eel. Pour chaque place finie $v$ de $F$, on note $F_v$ le compl\'et\'e de $F$ en $v$ et on identifie $\gFv$ \`a un sous-groupe de $\gF$ via un choix de plongement $\Qbar \hookrightarrow \Fvbar$. On note $k_v$ le corps r\'esiduel de $F_v$, $q_v \= |k_v|$, $|\cdot|\=\frac{1}{q_v^{\val(\cdot)}}$ (o\`u $\val(\varpi_v)=1$ si $\varpi_v$ est une uniformisante de $F_v$), $I_v\subset W_{v}$ les sous-groupes d'inertie et de Weil de $\gFv$ et $\Fr_v\in \gFvbar$ un Frobenius g\'eom\'etrique $x\mapsto x^{q_v^{-1}}$ en $v$. On normalise la th\'eorie du corps de classes local de telle sorte que les relev\'es des Frobenius g\'eom\'etriques correspondent aux uniformisantes. On normalise la correspondance locale de Langlands de telle sorte que, si $\xi_1$ et $\xi_2$ sont des caract\`eres lisses de $F_v^\times$ tels que $\xi_1\xi_2^{-1} \neq \norm^{\pm1}$, alors la repr\'esentation de $W_v$ qui correspond \`a $\Ind_{{\rm B}(F_v)}^{\GL_2(F_v)}\xi_1\otimes\xi_2\norm^{-1}$ est $\xi_1\oplus\xi_2$. 

Soit $D$ une alg\`ebre de quaternions sur $F$ qui est d\'eploy\'ee en une seule des places archim\'ediennes de $F$ not\'ee $\tau$ (on exclut le cas $F=\Q$ et $D={\mathrm{GL}}_2$). On fixe un isomorphisme $D_\tau \=D\otimes_{F,\tau}\R\cong {\rm M}_2(\R)$ et on pose $D_f^\times \= (D\otimes_{\Q}\A_f)^\times$ o\`u $\A_f \= \widehat{\Z}\otimes\Q$ est l'anneau des ad\`eles finis de $\Q$. Si $v$ est une place finie de $F$, on note $D_v\=D\otimes_F\Fv$. Pour chaque sous-groupe ouvert compact $U \subset D_f^\times$, on note $X_U$ la courbe alg\'ebrique projective lisse sur $F$
associ\'ee \`a $U$ par la th\'eorie des mod\`eles cano\-niques de Shimura \cite{Sh} (voir \cite[\S\ 1.1]{Ca1} pour un r\'esum\'e des r\'esultats sous la forme que l'on utilise ici, sauf que l'on prend la convention associ\'ee au choix de signe $\varepsilon = 1$ comme dans \cite[\S\ 3.3]{CV} plut\^ot que la convention de \cite{Ca1} adopt\'ee dans \cite{BDJ}). Les points complexes de $X_U$ relativement \`a $\tau$ s'identifient \`a :
$$D^\times\backslash ((\C\backslash\R) \times D_f^\times/U)$$
o\`u l'action de $D^\times$ sur $\C\backslash\R$ se fait via $D^\times \hookrightarrow D_\tau^\times \cong \GL_2(\R)$. Pour $g\in D_f^\times$ et pour des sous-groupes ouverts compacts $U,V\subset D_f^\times$ tels que $g^{-1}Vg\subseteq U$, l'application sur les points complexes d\'efinie par la multiplication \`a droite par $g$ provient d'un morphisme de courbes alg\'ebriques $X_V \to X_U$ d\'efini sur $F$, et tous ces morphismes induisent des applications $\gF$-\'equivariantes sur la cohomologie \'etale :
$$H^1_{\et}(X_{U,\Qbar},\Qpbar) \to H^1_{\et}(X_{V,\Qbar},\Qpbar).$$
On obtient comme cela des actions de $\gF$ et de $D_f^\times$ qui commutent sur :
$$\Pi_D \= \varinjlim H^1_{\et}(X_{U,\Qbar},\Qpbar)(1)$$
o\`u la limite inductive est prise sur les sous-groupes ouverts compacts $U\subset D_f^\times$, o\`u les applications de transition pour $V\subseteq U$ sont induites par $g=1$ et o\`u $(1)$ signifie le tordu par le caract\`ere cyclotomique $p$-adique de $\gF$. L'action de $\gF$ est continue et celle de $D_f^\times$ est lisse et admissible. (On pourrait aussi prendre la cohomologie \'etale \`a coefficients dans une extension finie $E$ de $\Qp$, mais pour l'instant, il est plus pratique d'avoir $\Qpbar$).

Fixons des plongements $\Qbar \hookrightarrow \C$ et $\Qbar \hookrightarrow \Qpbar$ et soit $\sigma_2$ la repr\'esentation de $D_\infty^\times \= (D\otimes_{\Q}\R)^\times$ qui, en $\tau$ est la repr\'esentation de la s\'erie discr\`ete holomorphe de poids $2$ de $D_\tau^\times\cong \GL_2(\R)$ de caract\`ere central trivial et qui aux autres places infinies est la repr\'esentation triviale. On a une d\'ecomposition :
$$\Pi_D = \bigoplus_{\pi} (\pi_f\otimes_{\Qbar} \rho_\pi)$$
o\`u $\pi$ parcourt les repr\'esentations automorphes de $(D\otimes_{\Q}\A)^\times$ telles que : $$\Hom_{\C[D_\infty^\times]}(\sigma_2,\pi)\neq 0.$$
Dans la d\'ecomposition ci-dessus, $\pi_f$ d\'esigne la repr\'esentation de $D_f^\times$ sur $\Qbar$ telle que :
$$\C\otimes_{\Qbar} \pi_f  \cong \Hom_{\C[D_\infty^\times]}(\sigma_2,\pi)$$
et $\rho_\pi\= \Hom_{\Qbar[D_f^\times]}(\pi_f, \Pi_D)$ est une repr\'esentation de $\gF$ de dimension $2$ sur $\Qpbar$. Pour une extension finie $E$ suffisamment grande, on \'ecrit encore $\rho_\pi$ pour la repr\'esentation $\gF \to \GL_2(E)$.

Dans ce cadre, la compatibilit\'e local-global de la correspondance de Langlands est due \`a Carayol \cite{Ca2} pour $v\!\nmid \!p$ et \`a T. Saito \cite{Sai} pour $v|p$. Chaque $\pi_f$ se factorise comme un produit tensoriel restreint sur toutes les places finies $v$ de $F$ :
$$\pi_f \cong \otimes_v' \pi_v$$
o\`u $\Qpbar\otimes_{\Qbar}\pi_v = \pi_{D_v}\big(\WD(\rho_\pi|_{\gFv})\big)$ est la repr\'esentation lisse irr\'eductible de $D_v^\times$ correspondant (par la correspondance de Langlands locale) \`a la ``Frobenius-semi-simplifi\'ee'' de $\WD(\rho_\pi|_{\gFv})$, la repr\'esentation de Weil-Deligne associ\'ee \`a $\rho_\pi|_{\gFv}$. En particulier, la repr\'esentation $\rho_\pi|_{\gFv}$ est potentiellement semi-stable pour tout $v|p$ et la repr\'esentation $\WD(\rho_\pi|_{\gFv})$ est ind\'ecomposa\-ble lorsque $D_v$ n'est pas d\'eploy\'ee.

Passons maintenant \`a la caract\'eristique $p$. Soit $\rhobar:\gF \to \GL_2(k_E)$ une repr\'esentation continue, irr\'eductible et totalement impaire. On suppose toujours que le corps fini $k_E$ contient l'extension quadratique d'un corps (fini) sur lequel $\rhobar$ est d\'efinie, de sorte que (i) $k_E$ contient les valeurs propres de $\rhobar(g)$ pour tout $g\in \gF$ et (ii) si $H$ est un sous-groupe de $\gF$ tel que la restriction $\rho|_H$ est absolument r\'eductible alors elle est r\'eductible sur $k_E$. On suppose de plus $\rhobar$ {\em modulaire} au sens o\`u elle provient d'une forme modulaire de Hilbert propre (de poids et niveau quelconques). Comme pr\'ec\'edemment en caract\'eristique $0$, la limite inductive sur les sous-groupes ouverts compacts $U$ de $D_f^\times$ permet de d\'efinir une repr\'esentation de $\gF\times D_f^\times$ :
$$\overline\Pi_D \= \varinjlim H^1_{\et}(X_{U,\Qbar},k_E)(1).$$
On d\'efinit alors une repr\'esentation lisse de $D_f^\times$ sur $k_E$ par :
$$\pi_D(\rhobar) \= \Hom_{k_E[\gF]} (\rhobar,\overline{\Pi}_D)$$
(notons que $\pi_D(\rhobar)$ est la repr\'esentation associ\'ee au dual de $\rhobar$ dans \cite[\S\ 4]{BDJ}). Par \cite{Ta}, la repr\'esentation $\pi_D(\rhobar)$ est non nulle pour certains choix de $D$ que l'on explicitera au \S\ \ref{lifts0} (voir le corollaire \ref{nonzero}). Un argument classique (voir par exemple \cite[Lem.4.11]{BDJ}) montre que si $U$ est suffisamment petit, alors on a :
$$\pi_D(\rhobar)^U = \Hom_{k_E[\gF]}(\rhobar, H^1_{\et}(X_{U,\Qbar},k_E)(1)),$$
en particulier la repr\'esentation $\pi_D(\rhobar)$ est admissible. \`A la suite du travail d'Emer\-ton \cite{Em} pour $\GL_2/\Q$, il est conjectur\'e dans \cite[Conj.4.7]{BDJ} que $\pi_D(\rhobar)$ se factorise comme un produit tensoriel restreint :
\begin{equation}\label{factorization}
\pi_D(\rhobar) =  \otimes_v' \pi_{D,v}(\rhobar)
\end{equation}
o\`u chaque $\pi_{D,v}(\rhobar)$ est une repr\'esentation lisse admissible de $D_v^\times$ sur $k_E$. Au moins pour $v$ ne divisant pas $p$, on s'attend de plus \`a ce que $\pi_{D,v}(\rhobar)$ ne 
d\'epende que des classes d'isomorphisme de $\rhobar|_{\gFv}$ et de $D_v$. Si $v$ divise $p$ et si $\tau$ est une repr\'esentation lisse irr\'eductible de $\oDv^\times$ sur $k_E$ (o\`u $\oDv$ est un ordre maximal dans $D_v$), on dit que $\rhobar$ est {\em modulaire de poids} $\tau$ (en $v$, par rapport \`a $D$) si : 
$$\Hom_{k_E[\oDv^\times]}(\tau,\pi_D(\rhobar)) \neq 0.$$
Nous utiliserons le r\'esultat r\'ecent suivant de Gee et Kisin (cf. \cite[Thm.B]{GeK}) sur les poids dans la conjecture de Serre de \cite{BDJ}.

\begin{thm}\label{cor:Geeplus} 
Supposons $p>2$, $\rhobar|_{{\Gal}(\Qbar/F(\sqrt[p]{1}))}$ irr\'eductible et, si $p=5$, l'image de $\rhobar({\Gal}(\Qbar/F(\sqrt[p]{1})))$ dans ${\rm PGL}_2(k_E)$ non isomorphe \`a ${\rm PSL}_2(\F_5)$. Soit $v$ une place de $F$ divisant $p$ telle que $F_v$ est non ramifi\'ee et $D_v$ est d\'eploy\'ee. Si $\rhobar$ est modulaire de poids $\tau$ (en $v$, par rapport \`a $D$) alors $\tau\in {\mathcal D}(\rhobar|_{\gFv})$.
\end{thm}

Il y a \'egalement une action de $\gF\times D_f^\times$ sur la cohomologie \'etale \`a coefficients entiers :
$$\Pi_D^0 \= \varinjlim H^1_{\et}(X_{U,\Qbar},\oE)(1).$$
Soit $\pi$ une repr\'esentation automorphe de $(D\otimes\A)^\times$ telle que $\pi_\infty \cong \sigma_2$ et $\rhobar_\pi \cong \rhobar$ o\`u $\rhobar_\pi$ est la r\'eduction de l'unique $\oE$-r\'eseau stable par $\gF$ \`a homoth\'etie pr\`es dans $\rho_\pi$. On suppose que $E$ est suffisamment grand pour que $\rho_\pi$ soit d\'efinie sur $E$ et on pose~:
\begin{eqnarray}\label{pi0}
\pi^0 \= \Hom_{\oE[\gF]}(\rho_\pi^0,\Pi_D^0)
\end{eqnarray}
o\`u $\rho_\pi^0$ est un $\oE$-r\'eseau stable par $\gF$ dans $\rho_\pi$. L'application naturelle $\Qpbar \otimes_{\oE}\pi^0 \to \Qpbar\otimes_{\Qbar}\pi_f$ est un isomorphisme. De plus, comme $\gF$ agit sur $\varinjlim H^0_{\et}(X_{U,\Qbar},\oE)(1)$ via $\gF^{\mathrm{ab}}$, l'irr\'eductibilit\'e de $\rhobar$ implique que :
$$(\pi^0)^U = \Hom_{\oE[\gF]}(\rho_\pi^0, H^1_{\et}(X_{U,\Qbar},\oE)(1))$$
pour tout sous-groupe ouvert compact $U\subset D_f^\times$. En particulier (le $\oE$-module sous-jacent \`a) $\pi^0$ n'a pas de vecteurs divisibles. Par ailleurs, l'injection :
$$k_E\otimes_{\oE}H^1_{\et}(X_{U,\Qbar},\oE)\hookrightarrow H^1_{\et}(X_{U,\Qbar},k_E)$$
(en fait un isomorphisme) montre que l'application naturelle :
$$k_E\otimes_{\oE}(\pi^0)^U \to \pi_D(\rhobar)^U$$
est injective. En prenant la limite inductive, on obtient le lemme suivant.

\begin{lem}\label{pi0bar} 
L'application naturelle $k_E \otimes_{\oE} \pi^0 \rightarrow \pi_D(\rhobar)$ est injective.
\end{lem}

\subsection{Relev\'es de type fix\'e}\label{lifts0}

On rappelle quelques cons\'equences de r\'esultats de Gee et de Barnet-Lamb, Gee et Geraghty (\cite{Ge}, \cite{BLGG}, \cite{BLGG2}) et on en d\'eduit quelques autres. 

Le r\'esultat principal de Gee est inspir\'e d'une technique due \`a Khare et Wintenberger (\cite{KW}) pour d\'emontrer l'existence et la modularit\'e de relev\'es avec comportement local prescrit de repr\'esentations $\rhobar:\gF \to \GL_2(k_E)$ (continues, irr\'eductibles, totalement impaires). Il g\'en\'eralise des r\'esultats de \cite{DT} (pour $v\neq p$) et de \cite{Kh} (pour $v = p$) dans le cas $F=\Q$, eux-m\^emes inspir\'es des r\'esultats de changement de niveau de Ribet (\cite{Ri1}, \cite{Ri2}). Ce r\'esultat de Gee est \'etendu et renforc\'e par celui clef de Barnet-Lamb, Gee et Geraghty qui donne l'existence de tels relev\'es qui sont ordinaires en des places prescrites au-dessus de $p$.

On conserve les notations du \S\ \ref{globalprelim}. Pour $v$ place finie de $F$ on note $\nu$ la surjection canonique $\gFv\twoheadrightarrow \gFvbar\cong \widehat\Z$ o\`u l'isomorphisme de droite est d\'efini en envoyant le Frobenius g\'eom\'etrique $\Fr_v$ sur $1$. On rappelle qu'une {\em repr\'esentation de Weil-Deligne} $(r,N)$ en $v$ (d\'efinie sur $\Qpbar$) est un $\Qpbar$-espace vectoriel $V$ de dimension finie muni d'une repr\'esentation lisse $r:W_v \to \Aut_{\Qpbar}V$ et d'un endomorphisme nilpotent $N:V \to V$ tels que $Nr(g) = q_v^{\nu(g)}r(g)N$ pour tout $g\in W_v$. On d\'efinit un {\em type de Weil-Deligne} en $v$ comme une classe d'\'equivalence $[r,N]$ de repr\'esentations de Weil-Deligne en $v$ de dimension $2$ o\`u $(r,N) \sim (r',N')$ si $(r|_{I_v},N)$ est isomorphe \`a $(r'|_{I_v},N')$. On dit qu'une repr\'esentation lin\'eaire continue $\rho:\gFv \to \GL_2(\Qpbar)$ est de type de Weil-Deligne $[r,N]$ si $\rho$ est potentiellement semi-stable et si $\WD(\rho) \sim (r,N)$. On dit que $\rho:\gFv \to \GL_2(E)$ est de type de Weil-Deligne $[r,N]$ si $\Qpbar\otimes_E\rho$ en est.

Soit $[r,N]$ un type de Weil-Deligne en $v$. On dit qu'une repr\'esentation lisse irr\'eductible $\vartheta_v$ de $\oDv^\times$ sur $E$ est un {\em $K$-type} pour $[r,N]$ si l'on a l'\'equivalence suivante pour toute repr\'esentation de Weil-Deligne $(r',N')$ en $v$ de dimension $2$ :
$$\Hom_{E[\oDv^\times]}(\vartheta_v, \pi_{D_v}(r',N'))\neq 0 \quad \Leftrightarrow \quad (r',N') \sim (r,N).$$
Supposant $E$ suffisamment grand, un $K$-type pour $[r,N]$ existe (sans \^etre en g\'en\'eral unique, voir \cite{He}) sous l'une des deux hypoth\`eses suivantes :\\
(i) $D_v$ est non ramifi\'ee et $N = 0$\\
(ii) $D_v$ est ramifi\'ee et ou bien $r$ est irr\'eductible ou bien $N\neq 0$.

Le lemme ci-dessous relie poids de Serre et types et se d\'emontre par un argument classique que l'on omet (voir par exemple \cite[Prop.2.10]{BDJ} pour un \'enonc\'e et une preuve sous des hypoth\`eses l\'eg\`erement diff\'erentes).

\begin{lem}\label{wts_types}
Soit $v$ une place de $F$ divisant $p$, $[r,N]$ un type de Weil-Deligne en $v$ et $\vartheta_v$ un $K$-type pour $[r,N]$. On a $\rhobar \cong \rhobar_\pi$ pour une repr\'esentation automorphe $\pi$ de $(D\otimes_{\Q}\A)^\times$ telle que $\pi_\infty \cong \sigma_2$ et $\rho_\pi|_{\gFv}$ est de type $[r,N]$ si et seulement si $\rhobar$ est modulaire de poids $\tau$ (en $v$, par rapport \`a $D$) pour un constituant $\tau$ du semi-simplifi\'e de ${\vartheta_v}$ sur $k_E$.
\end{lem}

On rappelle maintenant la cons\'equence suivante des r\'esultats de Barnet-Lamb, Gee et Geraghty (\cite{Ge}, \cite{BLGG}, \cite{BLGG2}).

\begin{thm}\label{lifts}
Supposons $p>2$, $\rhobar:\gF \to \GL_2(k_E)$ modulaire, $\rhobar|_{{\Gal}(\Qbar/F(\sqrt[p]{1}))}$ irr\'eductible et, si $p=5$, l'image de $\rhobar({\Gal}(\Qbar/F(\sqrt[p]{1})))$ dans ${\rm PGL}_2(k_E)$ non isomorphe \`a ${\rm PSL}_2(\F_5)$. Soit $\psi:\gF \to E^\times$ un caract\`ere qui rel\`eve $\det\rhobar$ et tel que $\psi\varepsilon^{-1}$ est d'ordre fini, $T$ un sous-ensemble de l'ensemble des places de $F$ divisant $p$ et $S$ un ensemble fini de places finies de $F$ contenant les places divisant $p$ et les places o\`u $\rhobar$ ou $\psi$ sont ramifi\'es. Pour chaque $v\in S$, soit $[r_v,N_v]$ un type de Weil-Deligne en $v$ et pour chaque $v\in T \cup\{v\!\nmid\!p,N_v\neq 0\}$, soit $\overline\mu_v:\gFv\to k_E^\times$ un caract\`ere. Supposons que, pour chaque $v\in S$, $\rhobar|_{\gFv}$ admet un relev\'e $\rho_v:\gFv \to \GL_2(E)$ tel que :
\begin{enumerate}
\item[(i)]si $v|p$ alors $\rho_v$ est potentiellement semi-stable de poids de Hodge-Tate $(0,1)$ pour tout $F_v \hookrightarrow \Qpbar$
\item[(ii)]si $v|p$ alors $\rho_v$ est potentiellement ordinaire si et seulement si $v\in T$
\item[(iii)]$\rho_v$ est de type de Weil-Deligne $[r_v,N_v]$ ($v\in S$)
\item[(iv)]si $v\in T\cup\{v\!\nmid\!p,N_v\neq 0\}$ alors $\rho_v$ a une sous-repr\'esentation $\sigma_v$ de dimension~$1$ telle que $\sigma_v$ rel\`eve $\overline\mu_v\omega$ et $\sigma_v\varepsilon^{-1}|_{I_v}$ est d'ordre fini
\item[(v)]$\det\rho_v|_{I_v} = \psi|_{I_v}$ ($v\in S$).
\end{enumerate}
Alors, quitte \`a agrandir $E$, $\rhobar$ poss\`ede un relev\'e $\rho:\gF \to \GL_2(E)$ continu non ramifi\'e en dehors de $S$ et tel que :
\begin{enumerate}
\item[(i)]si $v|p$ alors $\rho|_{\gFv}$ est potentiellement semi-stable de poids de Hodge-Tate $(0,1)$ pour tout $F_v \hookrightarrow \Qpbar$
\item[(ii)]si $v|p$ alors $\rho|_{\gFv}$ est potentiellement ordinaire si et seulement si $v\in T$
\item[(iii)]$\rho|_{\gFv}$ est de type de Weil-Deligne $[r_v,N_v]$ ($v\in S$)
\item[(iv)]si $v\in T\cup\{v\!\nmid\!p,N_v\neq 0\}$ alors $\rho|_{\gFv}$ a une sous-repr\'esentation $\sigma'_v$ de dimension~$1$ telle que $\sigma'_v$ rel\`eve $\overline\mu_v\omega$ et $\sigma'_v\varepsilon^{-1}|_{I_v}$ est d'ordre fini
\item[(v)]$\det\rho = \psi$.
\end{enumerate}
De plus, un tel relev\'e $\rho$ de $\rhobar$ provient d'une forme modulaire de Hilbert de poids $(2,2,\cdots,2)$.
\end{thm}
\begin{proof}
Ce th\'eor\`eme se d\'eduit des r\'esultats principaux de \cite{Ge} et \cite{BLGG2} comme expliqu\'e dans la preuve de \cite[Lem.5.3.2]{GeK}. Il diff\`ere de {\it loc.cit.} seulement parce que le cas $N_v\neq 0$ n'y est pas consid\'er\'e lorsque $v|p$ et n'y est pas explicit\'e lorsque $v\!\nmid\!p$. Pour obtenir le th\'eor\`eme \ref{lifts}, il suffit de modifier les arguments dans \cite{Ge} comme expliqu\'e ci-dessous.\\
D'abord, on v\'erifie qu'il existe une extension r\'esoluble totalement r\'eelle $F_0/F$ de degr\'e pair telle que :\\
(i) les restrictions $r_v|_{I_w}$, $\rhobar|_{{\Gal}(\overline{F_v}/F_{0,w})}$ et $\psi\varepsilon^{-1}|_{{\Gal}(\overline{F_v}/F_{0,w})}$ sont triviales pour tout $v\in S$ et tout $w|v$\\
(ii) les hypoth\`eses de \cite[Lem.5.3.2]{GeK} s'appliquent \`a $\overline{r} =\rhobar|_{{\Gal}(\Qbar/F_{0})}$ o\`u (changeant $\ell$ en $p$, $S$ en $S^+$ et $\psi$ en 
$\psi|_{{\Gal}(\Qbar/F_{0})}$) :
\begin{itemize}
\item $S^+ = \{w|p\} \amalg S_0$ o\`u $S_0$ est l'ensemble des places de $F_0$ au-dessus des places $v$ de $F$ telles que $v\!\nmid\! p$ et $N_v \neq 0$
\item $\tau_w$ est trivial pour tout $w\in S^+$
\item si $w|p$ alors $R_w$ est la composante irr\'eductible (de l'anneau de d\'eformation) param\'etrant les d\'eformations ordinaires si et seulement si $w$ divise $v$ pour un $v$ dans $T$
\item si $w\in S_0$ alors $R_w$ est la composante irr\'eductible (de l'anneau de d\'eforma\-tion) param\'etrant les d\'eformations de la forme $\smat{\varepsilon &*\\0&1}$ dans une base convenable.
\end{itemize}
Notant $D_0$ l'alg\`ebre de quaternions sur $F_0$ ramifi\'ee exactement aux places infinies et aux places de $S_0$, on obtient alors une repr\'esentation automorphe $\pi_0$ de 
$(D_0\otimes_{\Q}\A)^\times$ telle que :
\begin{itemize}
\item $\pi_{0,w}$ est triviale si $w$ est une place infinie ou si $w\in S_0$
\item $\pi_{0,w}$ est non ramifi\'ee si $w \not\in S_0$
\item $\rhobar_{\pi_0} \cong \rhobar|_{{\Gal}(\Qbar/F_{0})}$
\item si $w|p$ alors $\pi_0$ est ordinaire en $w$ si et seulement si $w|v$ pour un $v$ dans $T$.
\end{itemize}
On utilise alors l'argument d'augmentation du niveau de \cite{Ta} comme dans la preuve de \cite[(3.5.3)]{Ki} pour remplacer $S_0$ par l'ensemble de toutes les places de $F_0$ au-dessus des places $v$ de $F$ telles que $N_v\neq 0$. Plus pr\'ecis\'ement, soit $w_1,\cdots,w_r$ les places de $F_0$ divisant $p$ et au-dessus d'une place $v$ de $F$ telle que $N_v\neq 0$. Pour $i=1,\cdots,r$ on d\'efinit par r\'ecurrence des extensions $F_i$ de $F_0$ telles que :
\begin{itemize}
\item $F_i/F_{i-1}$ est quadratique totalement r\'eelle
\item $\rhobar|_{{\rm Gal}(\Qbar/F_i(\sqrt[p]{1}))}$ est irr\'eductible
\item si $w \in S_i'$ alors $w$ est totalement d\'ecompos\'ee dans $F_i$ o\`u $S_i'$ est l'ensemble des places de $F_{i-1}$ au-dessus des places de $\{w_1,\cdots,w_i\}\amalg S_0$
\item si $w$ est une place de $F_{i-1}$ au-dessus d'une place de $\{w_{i+1},\cdots,w_r\}$ alors $w$ est inerte dans $F_i$.
\end{itemize}
Soit $S_i$ l'ensemble des places de $F_i$ au-dessus de celles de $S_i'$, on a $S_i' = \{w_i'\}\amalg S_{i-1}$ o\`u $w_i'$ est l'unique place de $F_{i-1}$ au-dessus de $w_i$ (en particulier on a $w_1' = w_1$ et $S_1' = \{w_1\} \amalg S_0$). Soit $D_i$ l'alg\`ebre de quaternions sur $F_i$ ramifi\'ee exactement aux places infinies et aux places de $S_i$. On montre par r\'ecurrence sur $i$ qu'il existe une repr\'esentation automorphe $\pi_i$ de $(D_i\otimes_{\Q}\A)^\times$ v\'erifiant la m\^eme liste de propri\'et\'es que $\pi_0$ (voir ci-dessus) avec $F_0$ remplac\'e par $F_i$, $D_0$ par $D_i$ et $S_0$ par $S_i$ (cela se d\'emontre en choisissant un premier auxiliaire $\ell$ satisfaisant \cite[(3.5.6)]{Ki} et en appliquant le m\^eme argument qu'en \cite[(3.5.3)]{Ki}, l'existence d'un relev\'e de $\rhobar_{\pi_{i-1}}|_{{\rm Gal}(\overline{F_{w_i'}}/F_{w_i'})}$ de type de Weil-Deligne $[r,N]$ avec $r$ non ramifi\'ee et $N\neq 0$ combin\'ee avec l'ordinarit\'e de $\pi_{i-1}$ en la place $w'_i$ et la compatibilit\'e local-global en $w_i'$ assurent que les hypoth\`eses de \cite[(3.1.11)]{Ki} sont satisfaites, l'hypoth\`ese $w_i'\!\nmid \!p$ \'etant superflue dans la preuve lorsque la repr\'esentation $\tau$ de {\it loc.cit.} est triviale).\\
Pour finir, on proc\`ede exactement comme dans la preuve de \cite[(3.1.5)]{Ge} : pour toutes les places finies $v\in S$ telles que $N_v \neq 0$, on choisit la composante irr\'eductible (de l'anneau de d\'eformation locale ``cadr\'ee'' (framed)) param\'etrant les relev\'es conjugu\'es \`a $\mu_v\smat{\varepsilon&*\\0&1}$ avec $\mu_v:\gFv\to \oE^\times$ choisi tel que le relev\'e $\rho_v$ de l'\'enonc\'e v\'erifie $\rho_v\sim \mu_v\smat{\varepsilon&*\\0&1}$ (cf. \cite[(3.2.6)]{KW} pour une analyse de cette composante lorsque $v$ divise $p$). En travaillant au-dessus de $F' \= F_r$, on voit que l'anneau de d\'eformation qui en r\'esulte poss\`ede un point modulaire, et donc tous ses points \`a valeurs dans $\Zpbar$ sont modulaires et il est fini sur $\Z_p$. Il s'ensuit que $R_{F',S}^{\psi,\tau}$ (avec la notation de \cite{Ge}) est fini sur $\Z_p$, donc a un point \`a valeurs dans $\Zpbar$. De plus, si $\rho$ est la d\'eformation correspondante, alors $\rho|_{{\rm Gal}(\Qbar/F')}$ est modulaire, et donc aussi $\rho$. Cela conclut la preuve.
\end{proof}

Notons {\bf(H0)} l'hypoth\`ese de compatibilit\'e suivante entre $D$ et $\rhobar$ :
\begin{itemize}
\item[{\bf(H0)}]Pour chaque place $v$ de $F$ telle que $v\!\nmid\!p$ et $D_v$ est ramifi\'ee, la repr\'esentation $\rhobar|_{\gFv}$ est soit irr\'eductible soit isomorphe \`a une repr\'esentation de la forme $\overline\mu_v\smat{\omega & *\\ 0& 1}$ pour un caract\`ere $\overline\mu_v: \gFv \to k_E^\times$.
\end{itemize}

\begin{cor}\label{nonzero}
Supposons $p>2$, $\rhobar:\gF \to \GL_2(k_E)$ modulaire, $\rhobar|_{{\Gal}(\Qbar/F(\sqrt[p]{1}))}$ irr\'eductible et, si $p=5$, l'image de $\rhobar({\Gal}(\Qbar/F(\sqrt[p]{1})))$ dans ${\rm PGL}_2(k_E)$ non isomorphe \`a ${\rm PSL}_2(\F_5)$. Alors $\pi_D(\rhobar) \neq 0$ (voir \S~\ref{globalprelim}) si et seulement si l'hypoth\`ese {\bf(H0)} est v\'erifi\'ee.
\end{cor}
\begin{proof}
Notons d'abord que $\pi_D(\rhobar)$ est non nul si et seulement si $\rhobar=\rhobar_\pi$ pour une repr\'esentation automorphe $\pi$ de $(D\otimes_{\Q}\A)^\times$ telle que $\pi_\infty \cong \sigma_2$. En effet, si une telle $\pi$ existe alors le lemme \ref{pi0bar} montre que $\pi_D(\rhobar) \neq 0$. R\'eciproquement, si $\pi_D(\rhobar) \neq 0$ alors $\rhobar$ est une sous-repr\'esentation de la r\'eduction d'un r\'eseau dans $H^1_{\et}(X_{U,\Qbar},E)(1)$ pour un $U$ convenable, et la repr\'esentation $H^1_{\et}(X_{U,\Qbar},E)(1)$ est une somme directe de telles repr\'esentations $\rho_\pi$ (quitte \`a agrandir $E$). En parti\-culier, si $\pi_D(\rhobar)\neq 0$ et $D$ est ramifi\'ee en $v$, alors $\rhobar$ poss\`ede un relev\'e modulaire $\rho$ tel que $\rho|_{\gFv}$ est soit irr\'eductible, soit un tordu d'une repr\'esentation de la forme $\smat{\varepsilon & *\\ 0& 1}$. De plus, il est bien connu que si $\rho|_{\gFv}$ est irr\'eductible mais a une r\'eduction 
r\'eductible, alors c'est une induite d'un caract\`ere $\psi:{\rm Gal}(\overline{F_v}/L) \to \oE^\times$ tel que $\overline{\psi}^\sigma = \overline{\psi}$ o\`u $q_v\equiv -1\bmod p$, $L$ est l'extension quadratique non ramifi\'ee de $F_v$ et $\sigma$ le g\'en\'erateur de $\Gal(L/F_v)$. On en d\'eduit que $\rhobar|_{\gFv}$ est dans tous les cas de la forme requise pour que {\bf(H0)} soit v\'erifi\'ee.\\
R\'eciproquement, supposons {\bf(H0)} satisfaite. Comme $\rhobar$ est modulaire, on a $\rhobar \cong \rhobar_{\pi'}$ pour une repr\'esentation automorphe $\pi'$ de $(D'\otimes_{\Q}\A)^\times$ telle que $\pi'_\infty \cong \sigma_2$ o\`u $D'$ est une alg\`ebre de quaternions sur $F$ non ramifi\'ee aux places divisant $p$ et en une place infinie exactement (voir par exemple \cite[(2.12)]{BDJ}). Soit $\psi \= \det\rho_{\pi'}$ et, pour chaque $v|p$, soit $\tau_v$ un poids de Serre tel que $\rhobar$ est modulaire de poids $\tau_v$ (en $v$) par rapport \`a $D'$. Par la proposition \ref{app:SCoGL2} (ou la proposition \ref{app:SCeGL2}) de l'appendice, $\tau_v$ est un constituant de la r\'eduction d'un $K$-type supercuspidal, donc par le lemme \ref{wts_types} pour chaque $v|p$ la repr\'esentation $\rhobar|_{\gFv}$ admet un relev\'e potentiellement semi-stable $\rho_v$ avec tous ses poids de Hodge-Tate $(0,1)$ et de type de Weil-Deligne correspondant \`a une repr\'esentation supercuspidale. Pour chaque $v\!\nmid\!p$ o\`u $D$ est ramifi\'ee, l'hypoth\`ese {\bf(H0)} assure que $\rhobar|_{\gFv}$ admet un relev\'e $\rho_v$ de type de Weil-Deligne sp\'ecial ou supercuspidal. Comme $\det\rhobar|_{I_v} = \overline\psi|_{I_v}$ pour tout $v$, on peut tordre chaque $\rho_v$ par un caract\`ere d'ordre une puissance de $p$ de sorte que $\det\rho_v|_{I_v} = \psi|_{I_v}$. Maintenant par le th\'eor\`eme \ref{lifts} et la correspondance de Jacquet-Langlands, on a $\rhobar\cong\rhobar_{\pi}$ pour une repr\'esentation automorphe $\pi$ de $(D\otimes_{\Q}\A)^\times$ telle que $\pi_\infty \cong \sigma_2$. On en d\'eduit $\pi_D(\rhobar)\neq 0$.
\end{proof}

M\^eme si l'objectif principal de cet article concerne les propri\'et\'es de l'hypoth\'e\-tique facteur local $\pi_{D,v}(\rhobar)$ en (\ref{factorization}) pour $v|p$ lorsque $D_v^{\times} \cong \GFv$ et $F_v$ est non ramifi\'ee sur $\Q_p$, on en profite pour signaler le r\'esultat surprenant suivant lorsque $D_v$ reste une alg\`ebre de quaternions en $v|p$.

\begin{cor}\label{extra}
Supposons $p>2$, $\rhobar:\gF \to \GL_2(k_E)$ modulaire, l'hypoth\`ese {\bf(H0)} satisfaite, $\rhobar|_{{\Gal}(\Qbar/F(\sqrt[p]{1}))}$ irr\'eductible et, si $p=5$, l'image de $\rhobar({\Gal}(\Qbar/F(\sqrt[p]{1})))$ dans ${\rm PGL}_2(k_E)$ non isomorphe \`a ${\rm PSL}_2(\F_5)$. Si  (\ref{factorization}) est vrai alors pour toute place $v$ de $F$ divisant $p$ o\`u $D$ est ramifi\'ee la repr\'esentation $\pi_{D,v}(\rhobar)$ est de longueur infinie.
\end{cor}
\begin{proof}
Soit $\pi$ une repr\'esentation automorphe de $(D\otimes_{\Q}\A)^\times$ telle que $\pi_\infty \cong \sigma_2$ et $\rhobar_\pi \cong \rhobar$. Soit $S$ l'ensemble des places $w\neq v$ de $F$ telles que ou bien $w|p$ ou bien $\rho_\pi|_{\gFw}$ est ramifi\'ee ($S$ contient donc en particulier les places finies distinctes de $v$ o\`u $D$ est ramifi\'ee) et soit $\psi \= \det\rho_\pi$. Pour chaque $w\in S$, soit $[r_w,N_w]$ le type de Weil-Deligne de $\rho_\pi|_{\gFw}$. Soit $\tau_v$ un poids de Serre quelconque pour lequel $\rhobar$ est modulaire en $v$ par rapport \`a $D$. Par la proposition \ref{app:SCoD} (cf. appendice), pour tout $m \ge 0$ il existe un $K$-type supercuspidal $\vartheta_{v,m}$ de conducteur essentiel $2m+3$ dont la r\'eduction contient $\tau_v$ comme constituant. Par le lemme \ref{wts_types}, on en d\'eduit que $\rhobar|_{\gFv}$ admet un relev\'e $\rho_{v,m}$ potentiellement semi-stable \`a poids de Hodge-Tate $(0,1)$ de type de Weil-Deligne irr\'eductible $[r_{v,m},0]$ de conducteur essentiel $2m+3$. Quitte \`a tordre par un caract\`ere, on peut de plus supposer $\det\rho_{v,m}|_{I_v} = \psi|_{I_v}$. Par le th\'eor\`eme \ref{lifts} et la correspondance de Jacquet-Langlands, on a donc une repr\'esentation automorphe $\pi_m$ de $(D\otimes_{\Q}\A)^\times$ telle que $\pi_{m,\infty}\cong \sigma_2$, $\rhobar_{\pi_m} \cong \rhobar$, $\rho_{\pi_m}|_{{\rm Gal}(\overline{F_w}/F_w)}$ est de type de Weil-Deligne $[r_w,N_w]$ pour tout $w \in S$ et $\rho_{\pi_m}|_{\gFv}$ est de type de Weil-Deligne $[r_{v,m},0]$. Pour $w\in S$ et pour un sous-groupe ouvert compact $U_w \subseteq {\mathcal O}_{D_w}^\times$, la dimension (sur $\Qbar$) de $\pi_{m,w}^{U_w}$ est ind\'ependante de $m$ (car la compatibilit\'e local-global et le fait que $\rho_{\pi_m}|_{{\rm Gal}(\overline{F_w}/F_w)}$ est de type de Weil-Deligne ind\'ependant de $m$ impliquent que la restriction $\pi_{m,w}|_{{\mathcal O}_{D_w}^\times}$ ne d\'epend pas de $m$). Soit $U^v \= \prod_{w\neq v} U_w$ o\`u, pour $w\in S$, $U_w$ est suffisamment petit pour que $\pi_{m,w}^{U_w}\neq 0$ et o\`u $U_w = {\mathcal O}_{D_w}^\times$ pour $w\not\in S$. Alors on a $\dim_{\Qbar} \pi_{m,f}^{U^v} = C q_v^m$ pour un entier $C$ strictement positif et ind\'ependant de $m$ (voir la proposition \ref{app:SCoD} de l'appendice). Par le lemme \ref{pi0bar} on a donc $\dim_{k_E} \pi_D(\rhobar)^{U^v} \ge Cq_v^m$. Comme c'est vrai pour tout $m$, on en d\'eduit que $\pi_D(\rhobar)^{U^v}$ est de dimension infinie. Si (\ref{factorization}) est vrai, on voit donc que $\pi_{D,v}(\rhobar)$ est aussi de dimension infinie. Comme $D_v^\times$ est compact modulo son centre, toutes ses repr\'esentations lisses irr\'eductibles sur $k_E$ sont de dimension finie. On en d\'eduit que $\pi_{D,v}(\rhobar)$ est de longueur infinie.
\end{proof}

Nous donnerons au \S\ \ref{unI} ci-dessous une variante du corollaire \ref{extra} (cf. corollaire \ref{extra+}) qui s'applique, elle, \`a une ``vraie'' repr\'esentation $\pi_{D,v}(\rhobar)$ (et non plus conjecturale) que nous d\'efinissons maintenant.

\subsection{Le facteur local}\label{facteur}

Si $\rhobar$ est modulaire et si $v$ est une place de $F$ divisant $p$, on d\'efinit de mani\`ere ad hoc (sous quelques hypoth\`eses techniques assez faibles) un ``facteur local'' $\pi_{D,v}(\rhobar)$ d\'ependant de la repr\'esentation globale $\rhobar$.

On suppose $p>2$ et on fixe une repr\'esentation $\rhobar:\gF \to \GL_2(k_E)$ continue, modulaire et irr\'eductible en restriction \`a ${\Gal}(\Qbar/F(\sqrt[p]{1}))$ telle que, si $p=5$, l'image de $\rhobar({\Gal}(\Qbar/F(\sqrt[p]{1})))$ dans ${\rm PGL}_2(k_E)$ est non isomorphe \`a ${\rm PSL}_2(\F_5)$. On fixe une place $v$ de $F$ au-dessus de $p$ (on ne fait pas d'hypoth\`ese suppl\'ementaire sur $v$ dans cette section). On fixe aussi une alg\`ebre de quaternions $D$ sur $F$ v\'erifiant {\bf(H0)}, de sorte que $\pi_D(\rhobar)\neq 0$ par le corollaire \ref{nonzero}. On note $\oD$ un ordre maximal dans $D$ et $\oDw$ son adh\'erence dans $D_w$ pour une place finie quelconque $w$ de $F$. Dans toute cette section, on suppose de plus satisfaites les hypoth\`eses {\bf(H1)} et {\bf(H2)} suivantes :

\begin{itemize}
\item[{\bf(H1)}] Si $w$ est une place finie de $F$ o\`u $D$ est ramifi\'ee, alors $\rhobar|_{\gFw}$ est non scalaire.
\item[{\bf(H2)}] Si $w\ne v$ est une place de $F$ divisant $p$, alors $D_w$ est d\'eploy\'ee et $\rhobar|_{\gFw}$ est r\'eductible non scalaire.
\end{itemize}

L'hypoth\`ese {\bf(H1)} et la condition d'\^etre non scalaire dans l'hypoth\`ese {\bf(H2)} sont n\'ecessaires afin d'\^etre s\^ur que les facteurs locaux en dehors de $v$ ont des propri\'et\'es convenables de ``multiplicit\'e $1$''. Il devrait \^etre possible de supprimer l'hypoth\`ese de r\'eductibilit\'e dans {\bf(H2)}, i.e. de traiter des cas o\`u $\rhobar|_{\gFw}$ est irr\'eductible pour $w \neq v$ divisant $p$, en utilisant le travail r\'ecent de Cheng (\cite{Ch}), voir la remarque \ref{cheng}.

On ignore \`a l'heure actuelle si l'on dispose d'une factorisation (\ref{factorization}), mais on d\'efinit ci-dessous de mani\`ere ad hoc une repr\'esentation lisse admissible $\pi_{D,v}(\rhobar)$ de $D_v^{\times}$ sur $k_E$ (d\'ependant {\it a priori} de toute la repr\'esentation $\rhobar$) et dont on montrera au \S\ \ref{resprinc} que, au moins lorsque $F_v$ est non ramifi\'ee sur $\Qp$, $D_v$ est d\'eploy\'ee et $\rhobar|_{\gFv}$ est r\'eductible g\'en\'erique, elle co\"\i ncide avec le ``facteur local'' en $v$ dans (\ref{factorization}) si (\ref{factorization}) est v\'erifi\'ee (corollaire \ref{ok}). La repr\'esentation $\pi_{D,v}(\rhobar)$ est d\'efinie comme suit :
$$\pi_{D,v}(\rhobar) = \Hom_{k_E[U^v]}(\overline{M}^v,\pi_D(\rhobar))[\gm']$$
o\`u $\overline{M}^v$ est une repr\'esentation lisse irr\'eductible sur $k_E$ d'un sous-groupe ouvert compact $U^v$ de $\prod_{w\neq v}\oDw^\times$ et o\`u $\gm'$ est un id\'eal maximal dans une alg\`ebre de Hecke agissant sur $\Hom_{k_E[U^v]}(\overline{M}^v,\pi_D(\rhobar))$. Plus pr\'ecis\'ement, on d\'efinit ci-dessous des ensembles finis $S'\subseteq S$ de places finies de $F$, des repr\'esentations $\overline{M}_w$ (sur $k_E$) de sous-groupes ouverts compacts $U_w  \subset \oDw^\times$ pour $w\in S$, et des op\'erateurs de Hecke $T_w$ et des scalaires $\alpha_w \in k_E^\times$ pour $w\in S'$ tels que : 
$$U^v = \prod_{w\in S} U_w \!\!\!\prod_{w\not\in S\cup\{v\}}\!\!\!\!\oDw^\times,\ \ \ \ \overline{M}^v = \otimes_{w\in S}  \overline{M}_w$$
et tels que $\gm'$ est l'id\'eal maximal de $k_E[T_w,w\in S']$ engendr\'e par les $T_w-\alpha_w$ pour $w\in S'$. Si $w$ est une place finie de $F$, $\varpi_w$ une uniformisante de $F_w$ et $n\in \Z_{>0}$, on \'ecrit modulo $w^n$ pour modulo $\varpi_w^n$. 

On pose :
$$S \= \{w \neq v,\ \mbox{$w|p\Disc(D)$ ou bien $\rhobar|_{\gFw}$ est ramifi\'ee}\},$$
et :
$$S' \= \{w\neq v,\ w|p\} \ \ \coprod \ \ \{w\neq v,\ w|\Disc(D)\ {\rm et}\ \mbox{$\rhobar|_{\gFw}$ est r\'eductible}\}\subseteq S.$$
On d\'efinit $U_w$, $\overline{M}_w$ (pour $w\in S$) et $T_w$, $\alpha_w$ (pour $w\in S'$) au cas par cas en distinguant les quatre cas suivants (pour $w\in S$) :
\begin{eqnarray*}
{\rm Cas\ I} &:& w|p\\
{\rm Cas\ II} &:& w|\Disc(D)\ {\rm et}\ \rhobar|_\gFw\ {\rm est\ r\acute eductible}\\
{\rm Cas\ III} &:& w\!\nmid \!p\Disc(D)\ {\rm et}\ \rhobar|_\gFw\ {\rm est\ r\acute eductible\ (ramifi\acute ee)}\\
{\rm Cas\ IV} &:& \rhobar|_\gFw\ {\rm est\ irr\acute eductible\ (donc\ ramifi\acute ee).}
\end{eqnarray*}

Notons que $w\in S'$ correspond aux cas I et II et $w\in S\backslash S'$ aux cas III et IV (par {\bf(H2)}).

Cas I : Par {\bf(H2}), d'une part l'alg\`ebre de quaternions $D$ est d\'eploy\'ee en $w$ et on a donc un isomorphisme $\oDw^\times\cong \GL_2(\oFw)$, d'autre part $\rhobar|_\gFw$ est r\'eductible et on peut donc \'ecrire :
$$\rhobar|_\gFw \cong \begin{pmatrix} \overline{\xi}_w\omega & * \\ 0 & \overline{\xi}_w' \end{pmatrix}$$
pour des caract\`eres $\overline{\xi}_w,\overline{\xi}_w':\gFw \to k_E^\times$.

Si $\overline{\xi}_w|_{I_w} = \overline{\xi}_w'|_{I_w}$ et $\overline{\xi}_w^{-1}\otimes \rhobar|_\gFw$ est la fibre g\'en\'erique d'un sch\'ema en groupes fini et plat sur $\oFw$, on pose $U_w \= \GL_2(\oFw)$. Sinon, on d\'efinit $U_w \subset \GL_2(\oFw)$ comme le sous-groupe des matrices triangulaires sup\'erieures modulo $w$.

En voyant $\overline{\xi}_w,\overline{\xi}_w'$ comme des caract\`eres de $F_w^\times$, on pose $\overline{M}_w \= k_E(\overline{\vartheta}_w)$ o\`u $\overline{\vartheta}_w:U_w \to k_E^\times$ est le caract\`ere d\'efini par $\overline{\vartheta}_w(g) \= \overline{\xi}_w(\det(g))$ si $U_w = \GL_2(\oFw)$ et $\overline{\vartheta}_w(g) \= \overline{\xi}_w(a)\overline{\xi}_w'(d)$ pour $g \equiv \smat{a & b \\ 0 & d} \bmod w$ sinon.

Pour $T_w$ et $\alpha_w$, on choisit une uniformisante $\varpi_w$ de $\oFw$, puis on d\'efinit $T_w$ comme la double classe $V_w\smat{\varpi_w & 0 \\ 0 & 1} V_w$ o\`u $V_w \= \ker(\overline{\vartheta}_w)$ et on pose $\alpha_w \= \overline{\xi}_w(\varpi_w)$.

Cas II : Par \ {\bf(H0)}, \ on \ a \ $\rhobar|_{\gFw} \cong \overline{\xi}_w\smat{\omega & *\\ 0& 1}$ \ pour \ un \ caract\`ere \ $\overline{\xi}_w: \gFw \to k_E^\times$. De plus, par {\bf(H1)}, $\rhobar|_{\gFw}$ est non scind\'ee si $|k_w| \equiv 1 \bmod p$ (car alors $\omega=1$).

On pose $U_w \= \oDw^\times$ et $\overline{M}_w \= k_E(\overline{\vartheta}_w)$ o\`u $\overline{\vartheta}_w \= \overline{\xi}_w\circ\det$ ($\det$ d\'esigne ici la norme r\'eduite). On choisit une uniformisante $\Pi_{D_w}$ de $\oDw$ puis on d\'efinit $T_w$ comme la double classe $V_w \Pi_{D_w} V_w$ o\`u $V_w \= \ker(\overline{\vartheta}_w)$ et on pose $\alpha_w \= \overline{\xi}_w(\det(\Pi_{D_w}))$.

Cas III : On fixe un isomorphisme $\oDw^\times\cong \GL_2(\oFw)$ ainsi qu'un caract\`ere $\overline{\xi}_w:\gFw \to k_E^\times$ tel que $\overline{\sigma}_w \= \overline{\xi}_w^{-1}\rhobar|_\gFw$ est de conducteur minimal parmi ses tordus. On note $n_w$ l'exposant du conducteur de $\overline{\sigma}_w$ (un entier positif ou nul) et $\overline{\mu}_w:F_w^\times \to k_E^\times$ le caract\`ere correspondant \`a $\det(\overline{\sigma}_w)$.

Si $\overline{\sigma}_w$ est non ramifi\'ee, on pose $U_w \= \GL_2(\oFw)$. Sinon, on d\'efinit $U_w \subset \GL_2(\oFw)$ comme le sous-groupe des matrices triangulaires sup\'erieures modulo $w^{n_w}$. On pose $\overline{M}_w \= k_E(\overline{\vartheta}_w)$ o\`u $\overline{\vartheta}_w:U_w \to k_E^\times$ est le caract\`ere d\'efini par $\overline{\vartheta}_w(g) \= \overline{\xi}_w(\det(g))$ si $U_w = \GL_2(\oFw)$ et $\overline{\vartheta}_w(g) \= \overline{\xi}_w(\det(g))\overline{\mu}_w(d)$ pour $g \equiv \smat{a & b \\ 0 & d} \bmod w^{n_w}$ sinon, et $V_w \= \ker(\overline{\vartheta}_w)$.

Cas IV : Soit $\rho_w:\gFw \to \GL_2(E)$ un relev\'e quelconque de $\rhobar|_\gFw$ et soit $\vartheta_w$ un $K$-type pour $[r,N]$ (voir d\'ebut du \S\ \ref{lifts0}) o\`u $r=\rho_w|_{I_w}$ et $N=0$. Par des r\'esultats de Vign\'eras, la repr\'esentation $\overline{\vartheta}_w$ reste toujours irr\'eductible (si $D$ est d\'eploy\'ee en $w$, cela suit de \cite[1.6]{Vi2} et \cite[3.11]{Vi2}, et si $D$ est ramifi\'ee en $w$, cela suit de \cite[Prop.9]{Vi1}, \cite[Prop.11]{Vi1} et \cite[Cor.12]{Vi1}). On pose $U_w \= \oDw^\times$, $V_w \= \ker(\overline\vartheta_w)$ et $\overline{M}_w \= \overline{\vartheta}_w$.

Ayant d\'efini $U_w$, $V_w$ et $\overline{M}_w$ pour tout $w\in S$, on pose :
$$U^v \= \prod_{w\in S}U_w\prod_{w\not\in S\cup\{v\}}\oDw^\times,\quad V^v \= \prod_{w\in S}V_w\prod_{w\not\in S\cup\{v\}}\oDw^\times$$ 
et $\overline{M}^v \= \otimes_{w\in S}\overline{M}_w$. Notons que $\overline M^v$ est une repr\'esentation de $U^v/V^v$. 

\begin{lem}\label{ws'}
Pour $w\in S'$, l'action naturelle de $T_w$ sur $\pi_D(\rhobar)^{V^v}$ induit une action sur :
$$\Hom_{k_E[U^v]}(\overline{M}^v,\pi_D(\rhobar)) = \Hom_{k_E[U^v/V^v]}(\overline{M}^v,\pi_D(\rhobar)^{V^v}).$$
\end{lem}
\begin{proof}
Il suffit de montrer que (pour $w\in S'$) l'action de $U_w$ sur $\pi_D(\rhobar)^{V^v}$ commute avec celle de $T_w$. Lorsque $w|p$ (cas I), on peut choisir des repr\'esentants dans $U_w$ de $U_w/V_w$ qui commutent avec $\smat{ \varpi_w & 0 \\ 0 & 1}$, et donc pour $g \in U_w$ on a :
$$g \begin{pmatrix} \varpi_w & 0 \\ 0 & 1 \end{pmatrix} g^{-1}\in V_w\begin{pmatrix} \varpi_w & 0 \\ 0 & 1 \end{pmatrix} V_w.$$
Comme $V_w$ est distingu\'e dans $U_w$, on en d\'eduit que $g$ commute avec $T_w$. Lorsque $w|\Disc(D)$ (cas II), le caract\`ere $\overline{\vartheta}_w$ s'\'etend \`a $D_w^\times$. On a donc $\overline{\vartheta}_w(g\Pi_{D_w}g^{-1})=\overline{\vartheta}_w(\Pi_{D_w})$ pour $g\in U_w=\oDw^\times$ d'o\`u $g\Pi_{D_w}g^{-1}\Pi_{D_w}^{-1}\in \ker(\overline{\vartheta}_w)\cap \oDw^\times=V_w$ et en particulier $g\Pi_{D_w}g^{-1} \in V_w \Pi_{D_w} V_w$. Donc $g$ commute encore avec $T_w$.
\end{proof}

Pour $w\in S'$, les op\'erateurs $T_w$ agissent sur $\Hom_{k_E[U^v]}(\overline{M}^v,\pi_D(\rhobar))$ par le lemme \ref{ws'} et ils commutent de plus entre eux puisque tel est le cas sur $\pi_D(\rhobar)^{V^v}$. On a donc une action de $\T' \= k_E[T_w,w\in S']$ sur $\Hom_{k_E[U^v]}(\overline{M}^v,\pi_D(\rhobar))$ et on pose~:
\begin{eqnarray}\label{piDv}
\pi_{D,v}(\rhobar) \= \Hom_{k_E[U^v]}\big(\overline{M}^v,\pi_D(\rhobar)\big)[\gm']
\end{eqnarray}
o\`u $\gm'$ est l'id\'eal maximal de $\T'$ engendr\'e par les $T_w-\alpha_w$ pour $w\in S'$. C'est une repr\'esentation lisse admissible de $D_v^\times$ sur $k_E$ de caract\`ere central $\omega^{-1}\det\rhobar|_\gFv$.

\begin{rem}\label{choix}
{\rm (i) La d\'efinition de $\pi_{D,v}(\rhobar)$ en I et II ne d\'epend pas des choix des uniformisantes $\varpi_w$ et $\Pi_{D_w}$. Si, dans le cas I, $\rhobar|_\gFw$ est scind\'ee, ses sous-repr\'esentations de dimension $1$ sont par hypoth\`ese distinctes et il y a donc deux choix possibles pour d\'efinir $\overline{\xi}_w$ et $\overline{\xi}_w'$. De m\^eme dans le cas II avec $\overline{\xi}_w$ si $\rhobar|_\gFw$ est scind\'ee et si $|k_w|\equiv -1 \bmod p$ (car alors $\omega=\omega^{-1}$). Dans ces situations nous choisissons \`a chaque fois une des deux possibilit\'es. Il est probable que la repr\'esentation obtenue $\pi_{D,v}(\rhobar)$ soit ind\'ependante de ces choix, mais nous ignorons comment le d\'emontrer pour l'instant.\\
(ii) Il y a d'autres choix possibles pour les d\'efinitions de $U_w$ et $\overline{M}_w$. Par exemple, on aurait pu proc\'eder en III comme on l'a fait en IV, c'est-\`a-dire introduire des $K$-types. Mais il aurait alors fallu utiliser une notion de type plus \'etendue de mani\`ere \`a inclure la repr\'esentation de Steinberg lorsque $\rhobar|_\gFw$ est 
sp\'eciale. De plus, pour avoir une repr\'esentation $\overline{M}_w$ irr\'eductible, il aurait fallu choisir un relev\'e irr\'eductible de $\rhobar|_\gFw$ lorsque $\rhobar|_\gFw$ est sp\'eciale et $|k_w| \equiv -1\bmod p$. Inversement, on aurait pu aussi utiliser en IV une d\'efinition si\-milaire \`a celle de III (sauf dans le cas $|k_w| \equiv -1\bmod p$ et $\rhobar|_\gFw$ est induite de l'extension quadratique non ramifi\'ee de $F_w$). Il n'est pas difficile, l\`a, de v\'erifier que ces variantes donneraient la m\^eme repr\'esentation $\pi_{D,v}(\rhobar)$.\\
(iii) Enfin, on peut remarquer que l'op\'erateur de Hecke $T_w$ est vraiment n\'ecessaire dans le cas I seulement si $\rhobar|_{I_w}$ est une repr\'esentation scalaire, et dans le cas II seulement si $\rhobar|_\gFw$ est scind\'ee et $|k_w| \equiv - 1 \bmod p$.}
\end{rem}

\subsection{D\'eformations}\label{deformation}

On d\'efinit les anneaux de d\'eformations de repr\'esenta\-tions galoisiennes locales et globales que l'on utilisera dans la section suivante.

On suppose dans cette section $p>2$, $\rhobar:\gF \to \GL_2(k_E)$ irr\'eductible et $\rhobar|_{\gFw}$ r\'eductible non scalaire pour toutes les places $w$ de $F$ divisant $p$.

Soit $\Sigma$ un ensemble fini de places finies de $F$ contenant les places divisant $p$ et les places o\`u $\rhobar$ est ramifi\'ee. Soit $\Sigma'$ un sous-ensemble de $\Sigma$ tel que :
\begin{itemize}
\item si $w\vert p$ ou si $\rhobar|_{I_w}$ est\ r\'eductible\ non\ scind\'ee (donc en particulier $\rhobar|_\gFw$ est alors r\'eductible ramifi\'ee), alors $w \in \Sigma'$
\item si $w\!\nmid \!p$ et $w\in \Sigma'$, alors $\rhobar|_{\gFw}$ est isomorphe \`a une repr\'esentation non\ scalaire de la forme $\rhobar|_\gFw \cong \smat{\overline{\xi}_w\omega & * \\ 0 & \overline{\xi}_w}$.
\end{itemize}
Donc, pour tout $w\in \Sigma'$, on peut \'ecrire :
$$\rhobar|_\gFw \cong \begin{pmatrix} \overline{\xi}_w\omega & * \\ 0 & \overline{\xi}_w' \end{pmatrix}$$
pour des caract\`eres $\overline{\xi}_w,\overline{\xi}_w':\gFw \to k_E^\times$. Lorsqu'il y a plusieurs choix possibles pour le caract\`ere $\overline{\xi}_w$, on en fixe un (voir la remarque \ref{choix}(i)).

On commence par les anneaux de d\'eformations de repr\'esentations locales. On note $\CNL$ la cat\'egorie des $\oE$-alg\`ebres locales noeth\'eriennes compl\`etes de corps r\'esiduel $k_E$. Si $A$ est un objet de $\CNL$, on note $\gm_A$ son id\'eal maximal.

On note $\psi:\gF \to \oE^\times$ le relev\'e de Teichm\"uller de $\omega^{-1}\det\rhobar$. On rappelle qu'il existe un anneau de d\'eformation ``cadr\'ee'' (framed) $R_w^\square$ qui param\`etre les relev\'es de $\rhobar|_\gFw$ de d\'eterminant $\varepsilon\psi|_{\gFw}$. Plus 
pr\'ecis\'ement $R_w^\square$ repr\'esente le foncteur de $\CNL$ vers les ensembles qui envoie $A$ vers l'ensemble des relev\'es $\sigma: \gFw \to \GL_2(A)$ de $\rhobar|_\gFw$ tels que $\det\sigma = \varepsilon\psi|_{\gFw}$.

Pour $w\in \Sigma$ on d\'efinit un foncteur $\CF_w$ de $\CNL$ vers les ensembles comme suit. Soit $A$ un object de $\CNL$. Si $w\not\in \Sigma'$, alors $\CF_w(A)$ est par d\'efinition l'ensemble des relev\'es $\sigma:\gFw \to \GL_2(A)$ de $\rhobar|_{\gFw}$ tels que $\det\sigma = \varepsilon\psi|_{\gFw}$ et $\sigma(I_w) \stackrel{\sim}{\to} \rhobar(I_w)$. Si $w\in \Sigma'$, alors $\CF_w(A)$ est par d\'efinition l'ensemble des paires ordonn\'ees $(\sigma,L)$ o\`u $\sigma:\gFw \to \GL_2(A)$ est un relev\'e de $\rhobar|_\gFw$ tel que $\det\sigma = \varepsilon\psi|_{\gFw}$ et $L$ est un facteur direct de rang un de $A^2$ (l'espace de $\sigma$) sur lequel $\gFw$ agit par un caract\`ere de la forme $\eta\varepsilon$, o\`u $\eta$ est un relev\'e de $\overline{\xi}_w$ tel que $\eta(I_w) \stackrel{\sim}{\to} \overline{\xi}_w(I_w)$. Dans le cas o\`u $w|p$, $\overline{\xi}_w = \overline{\xi}'_w$ et 
$\overline{\xi}_w^{-1}\otimes\rhobar|_{\gFw}$ est la fibre g\'en\'erique d'un sch\'ema en groupes fini et plat sur $\oFw$, on demande de plus que $\eta^{-1}\otimes\sigma\otimes_A A/\gm_A^n$ soit aussi la fibre g\'en\'erique d'un sch\'ema en groupes fini et plat sur $\oFw$ pour tout $n\ge 1$. On d\'efinit $\CF_w$ sur les morphismes de la mani\`ere \'evidente.

\begin{lem} \label{lem_deformations}
Le foncteur $\CF_w$ pour $w\in \Sigma$ est repr\'esentable par un anneau $R_w^\triangle$ qui est formellement lisse sur $\oE$ de dimension relative $3+[F_w:\Q_p]$ (resp.~$3$) si $w|p$ (resp. $w\!\nmid\!p$). Si $w\not\in\Sigma'$ ou si $\rhobar|_{I_w}$ n'est pas scalaire, alors le morphisme $R_w^\square\to R_w^\triangle$ est surjectif. Si $w\in \Sigma'$ et $\rhobar|_{I_w}$ est scalaire, alors $R_w^\triangle$ est topologiquement engendr\'e sur $R_w^\square$ par $\eta_w^\univ(g)$ o\`u $\eta_w^\univ:\gFw\to (R_w^\triangle)^\times$ est d\'efini par l'action de $\gFw$ sur $L^\univ$ pour la paire universelle $(\sigma^\univ,L^\univ)$ sur $R_w^\triangle$ et o\`u $g \in \gFw$ est un relev\'e quelconque de $\Fr_w$.
\end{lem}
\begin{proof}
La preuve est essentiellement dans \cite{KW}. On explique ci-dessous comment la d\'eduire des r\'esultats qui y sont \'enonc\'es.
 
(i) Supposons d'abord $w\not\in \Sigma'$, donc en particulier $w\!\nmid \!p$. Dans ce cas, la repr\'esentabilit\'e de $\CF_w$ par un objet $R_w^\triangle$ de $\CNL$ et l'existence d'un relev\'e $\rho_0$ de $\rhobar|_{\gFw}$ dans $\CF_w(\oE)$ sont des r\'esultats standards. Consid\'erons la $\oE$-alg\`ebre $\overline{R}^\square_{\psi,0,\fl}$ associ\'ee \`a $\rho_0$ dans \cite[\S\ 2.7]{KW}. Comme $\overline{R}^\square_{\psi,0,\fl}$ est un quotient de $R_w^\triangle$ (par \cite[Prop.2.11]{KW}) et que $\overline{R}^\square_{\psi,0,\fl}[1/p]$ est r\'egulier de dimension $3$ (par \cite[Prop.2.10]{KW}), il suffit de montrer que l'espace tangent de $R_w^\triangle \otimes_{\oE}k_E$ est de dimension au plus $3$. Comme $p > 2$, on peut identifier cet espace tangent $\CF_w(k_E[\varepsilon]/(\varepsilon^2))$ avec le noyau de l'application naturelle $Z^1(\gFw,M) \to H^1(I_w,M)$ o\`u $M\=\ad^0(\rhobar|_{\gFw})$. Le quotient de ce noyau par les cobords $B^1(\gFw,M)$ est isomorphe \`a $H^1(\gFw/I_w,M^{I_w})$, qui a m\^eme dimension (sur $k_E$) que $H^0(\gFw,M)$. On en d\'eduit que la dimension de l'espace tangent est bien~: 
$$\dim_{k_E}B^1(\gFw,M) + \dim_{k_E}H^0(\gFw,M)= \dim_{k_E}M = 3.$$
La surjectivit\'e du morphisme $R_w^\square \to R_w^\triangle$ dans ce cas est claire.

(ii) On suppose maintenant $w\in \Sigma'$. Comme $\CF_w$ est isomorphe au foncteur obtenu en rempla\c cant la repr\'esentation $\rhobar|_{\gFw}$ par une de ses conjugu\'ees et en la tordant par un caract\`ere, on peut supposer :
$$\rhobar|_\gFw = \begin{pmatrix} \overline{\xi}_w\omega & z_{\rhobar} \\ 0 & 1 \end{pmatrix}$$
pour un cocycle $z_{\rhobar} \in Z^1(\gFw,k_E(\overline{\xi}_w\omega))$.

On consid\`ere d'abord le cas $w|p$ et $\overline{\xi}_w \neq 1$. On va construire explicitement un relev\'e universel $(\sigma^\univ,L^\univ) \in \CF_w(R_w^\triangle)$ o\`u 
$R_w^\triangle \= \oE[[T,X_1,\cdots,X_{d+1},Y]]$ et $d\=[F_w:\Q_p]$.

Soit $\eta_w^\univ : \gFw \to \oE[[T]]$ le caract\`ere $\xi_w\nr(1+T)$ o\`u $\xi_w\=[\overline{\xi}_w]$ est le relev\'e de Teichm\"uller de $\overline{\xi}_w$ et $\nr(1+T)$ est le caract\`ere non ramifi\'e envoyant un Frobenius g\'eom\'etrique vers $1+T$. On note $\Xi^\univ =\nr(1+T)\eta_w^\univ\varepsilon$. Par \cite[Lem.3.7]{KW}, le $\oE[[T]]$-module des cocycles :
$$Z^\univ \= Z^1(\gFw,\oE[[T]](\Xi^\univ))$$
est libre de rang $d+1$ et pour tout morphisme $\beta:\oE[[T]] \to A$ dans $\CNL$, on a :
$$Z^1(\gFw,A(\beta\circ\Xi^\univ)) = Z^\univ\otimes_{\oE[[T]]}A.$$
(Notons que \cite[Lem.3.7]{KW} suppose en fait $F_w$ non ramifi\'ee sur $\Q_p$, mais cela n'est pas n\'ecessaire dans la preuve.) En particulier, on a $Z^1(\gFw,k_E(\overline{\xi}_w\omega)) = Z^\univ \otimes_{\oE[[T]]} k_E$ et on peut choisir un relev\'e $\tilde{z}_{\rhobar} \in Z^\univ$ de $z_{\rhobar}$ ainsi qu'une base $\{z_1,\cdots,z_{d+1}\}$ de $Z^\univ$ sur $\oE[[T]]$. On d\'efinit alors : 
$$\sigma^\univ : \gFw \to \GL_2(R_w^\triangle)=\GL_2(\oE[[T,X_1,\cdots,X_{d+1},Y]])$$
par :
$$\sigma^\univ \=\nr(1+T)^{-1} \begin{pmatrix} 1 & 0 \\ Y & 1 \end{pmatrix}\begin{pmatrix} \Xi^\univ & \tilde{z}_{\rhobar} + \sum_{i=1}^{d+1} X_i z_i \\ 0 & 1 \end{pmatrix}\begin{pmatrix} 1 & 0 \\ - Y & 1 \end{pmatrix},$$
et $L^\univ$ comme le sous-$R_w^\triangle$-module engendr\'e par $\begin{pmatrix}1\\Y\end{pmatrix} \in \left(R_w^\triangle\right)^2$ (en particulier on a bien $\det(\sigma^\univ)=\xi_w\varepsilon$).

Maintenant que nous avons d\'efini un \'el\'ement $\sigma^\univ$ de $\CF_w(R_w^\triangle)$, nous devons montrer que, pour $A$ dans $\CNL$ et $(\sigma,L) \in \CF_w(A)$, il existe un unique $\alpha:R_w^\triangle \to A$ tel que $\sigma = \alpha\circ\sigma^\univ$ et $L = L^\univ\otimes_{R_w^\triangle}A$. On remarque d'abord qu'il y a des 
\'el\'ements $t,y\in \gm_A$ uniques tels que $L$ est engendr\'e par $\begin{pmatrix}1\\y\end{pmatrix}$ et $\gFw$ agit sur $L$ par $\eta\varepsilon$ o\`u $\eta = \xi_w\nr(1+t)$. Comme $\det(\sigma) = \xi_w\varepsilon$, on a :
$$\nr(1+t)\begin{pmatrix} 1 & 0 \\ -y & 1 \end{pmatrix}\circ \sigma \circ \begin{pmatrix} 1 & 0 \\ y & 1 \end{pmatrix} = \begin{pmatrix} \beta\circ \Xi^\univ & z \\ 0 & 1 \end{pmatrix}$$
pour un $z \in Z^1(\gFw,A(\beta\circ\Xi^\univ))$ o\`u $\beta:\oE[[T]]\rightarrow A$ est d\'efini par $\beta(T) = t$. Comme $z$ est un relev\'e de $z_{\rhobar}$, on a $z = 1 \otimes \tilde{z}_{\rhobar} + \sum_{i=1}^{d+1} x_i \otimes z_i$ pour des $x_1,\cdots,x_{d+1} \in \gm_A$ uniques. On en d\'eduit que $\alpha:R_w^\triangle \to A$ d\'efini par $T \mapsto t$, $Y \mapsto y$ et $X_i \mapsto x_i$ pour $i=1,\cdots,d+1$ est l'unique $\alpha$ transformant $(\sigma^\univ,L^\univ)$ en $(\sigma,L)$.

Les cas restants o\`u $w\in \Sigma'$ se traitent par des variantes de l'argument ci-dessus. Si $w|p$, $\overline{\xi}_w = 1$ et $\rhobar|_{\gFw}$ est la fibre g\'en\'erique d'un sch\'ema en groupes fini et plat sur $\oFw$, on remplace le module des cocycles par le module :
$$Z^1_f(\gFw,\oE[[T]](\Xi^\univ))$$
comme d\'efini dans \cite{KW}. Si $w|p$, $\overline{\xi}_w = 1$ et $\rhobar|_{\gFw}$ ne provient pas d'un sch\'ema en groupes fini et plat, on remplace $\oE[[T]]$ par $\oE$, $\eta_w^\univ$ par le caract\`ere trivial et $d+1$ par $d+2$. Enfin si $w\!\nmid \!p$, l'argument est le m\^eme que dans le cas pr\'ec\'edent mais en rempla\c cant $d+2$ par $2$.

(iii) Pour montrer la derni\`ere assertion, supposons que $\rhobar|_{I_w}$ est scalaire et soit $S$ la sous-$R_w^\square$-alg\`ebre de $R_w^\triangle$ topologiquement engendr\'ee par $\eta_w^\univ(g)$. Alors $\eta_w^\univ$ est \`a valeurs dans $S^\times$ et $\sigma^\univ$ est \`a valeurs dans  l'image de $\GL_2(R_w^\square)$, donc dans $\GL_2(S)$. Comme $\rhobar(g)$ n'est pas scalaire, on voit que la matrice $B\=\sigma^\univ(g) - (\eta_w^\univ\varepsilon)(g)I \in M_2(S)$ a une r\'eduction non nulle modulo $\gm_S$. Comme $\det(B) = 0$, on en d\'eduit que $L\=\ker(B)$ est un facteur direct de $S^2$ de rang $1$ et donc que $L^\univ = R_w^\triangle\otimes_S L$. La propri\'et\'e universelle de $(\sigma^\univ,L^\univ)$ donne donc une section $R_w^\triangle \to S$ de l'inclusion $S\subseteq R_w^\triangle$, d'o\`u il suit que $S = R_w^\triangle$.

Si $\rhobar|_{I_w}$ n'est pas scalaire, soit $S$ l'image de $R_w^\square$ dans $R_w^\triangle$. Alors $\sigma^\univ$ est \`a valeurs dans $\GL_2(S)$ et $\eta_w^\univ|_{I_w}$ dans $S^\times$. On choisit $g\in I_w$ tel que $\rhobar(g)$ n'est pas scalaire et on montre exactement comme ci-dessus que $S=R_w^\triangle$.
\end{proof}

\begin{rem} \label{rem_scalars}
{\rm Si l'on remplace $E$ par une extension finie $E'$, alors on a une application naturelle $R_w^{\prime\triangle} \to {\mathcal{O}}_{E'} \otimes_{\oE}R_w^\triangle$ o\`u $R_w^{\prime\triangle}$ est d\'efini en utilisant $E'$ au lieu de $E$. Un argument standard utilisant le sous-anneau de $R_w^{\prime\triangle}$ des \'el\'ements se r\'eduisant dans $k_E$ montre que cette application est un isomorphisme.}
\end{rem}

\begin{rem}
{\rm Lorsque $w|p$, l'anneau $R_w^\triangle$ co\"\i ncide avec l'anneau $\overline{R}_w^{\square,\psi}$ consid\'er\'e dans \cite{KW} sauf si $\rhobar|_{\gFw}$ est la tordue par un caract\`ere d'une repr\'esentation non ramifi\'ee r\'eductible non scind\'ee, auquel cas l'application naturelle $\overline{R}_w^{\square,\psi}\rightarrow 
R_w^\triangle$ n'est pas surjective.}
\end{rem}

On d\'efinit maintenant les anneaux de d\'eformations de repr\'esentations globales qui vont intervenir. Pour un ensemble fini $Q$ de places de $F$, on note $\CD_Q$ le foncteur de $\CNL$ vers les ensembles qui envoie un objet $A$ vers l'ensemble des classes d'\'equivalence des relev\'es $\sigma:\gF \to \GL_2(A)$ de $\rhobar$ non ramifi\'es en dehors de $\Sigma\cup Q$ et tels que $\det\sigma = \varepsilon\psi$, deux relev\'es \'etant \'equivalents s'ils sont conjugu\'es par une matrice $g\in \GL_2(A)$ dont la r\'eduction modulo $\gm_A$ est la matrice identit\'e. Comme $\rhobar$ est irr\'eductible, le foncteur $\CD_Q$ est repr\'esentable par un anneau de d\'eformation universelle $R_Q$. On note $R_Q^\square$ l'objet de $\CNL$ repr\'esentant le foncteur qui associe \`a $A$ l'ensemble des classes d'\'equivalence de paires ordonn\'ees $(\sigma,\{B_w\}_{w\in\Sigma})$ o\`u $\sigma$ est un relev\'e de $\rhobar$ comme dans la d\'efinition de $\CD_Q(A)$ et chaque $B_w$ est une base de $A^2$ qui se r\'eduit sur la base standard de $k_E^2$, et o\`u $(\sigma,\{B_w\}_{w\in\Sigma}) \sim (g\sigma g^{-1},\{gB_w\}_{w\in\Sigma})$ pour toute matrice $g\in \GL_2(A)$ qui se r\'eduit modulo $\gm_A$ sur l'identit\'e.

On pose $R_\loc^\square \= \widehat{\otimes}_{w\in\Sigma} R_w^\square$ et $R_\loc^\triangle \= \widehat{\otimes}_{w\in\Sigma} R_w^\triangle$. On a des applications naturelles $R_Q \to R_Q^\square$, $R_\loc^\square \to R_Q^\square$ et $R_\loc^\square \to R_\loc^\triangle$. On note enfin $R_Q^\triangle \= R_Q^\square \widehat{\otimes}_{R_\loc^\square} R_\loc^\triangle$. Lorsque $Q = \emptyset$, on supprime l'indice $Q$.

\subsection{Multiplicit\'e un I}\label{unI}

Dans cette section et la suivante, on applique la m\'ethode de Taylor-Wiles comme modifi\'ee par Diamond, Fujiwara et Kisin (cf. \cite{Di1} et \cite{Ki}). Cette premi\`ere section contient essentiellement des pr\'eliminaires.

On conserve les notations et hypoth\`eses du \S\ \ref{facteur} et on suppose de plus que la place $v\vert p$ fix\'ee est telle que $F_v$ est non ramifi\'ee sur $\Qp$, $D_v$ est d\'eploy\'ee et $\rhobar|_{\gFv}$ est r\'eductible g\'en\'erique (cf. \S\ \ref{prel} et rappelons que la g\'en\'ericit\'e implique $p>3$). On fixe un isomorphisme de $\oFw$-alg\`ebres $\oDw \cong M_2(\oFw)$ pour chaque place finie $w$ o\`u $D$ est d\'eploy\'ee.

Notons qu'avec {\bf (H2)}, on suppose en particulier que pour toute place $w\vert p$, $D_w$ est d\'eploy\'ee et $\rhobar|_{\gFw}$ est r\'eductible non scalaire. (Comme d\'ej\`a remarqu\'e au \S\ \ref{facteur}, les r\'esultats de \cite{Ch} devraient permettre de traiter les cas o\`u $\rhobar|_{\gFw}$ est irr\'eductible pour $w \neq v$ divisant $p$, voir la remarque \ref{cheng}.)

Rappelons qu'au \S\ \ref{facteur}, on a d\'efini des sous-groupes ouverts compacts $U_w \supset V_w$ et des repr\'esentations $\overline{M}_w$ de $U_w/V_w$ sur $k_E$ pour $w \in S$. La place $v$ n'est pas dans $S$, mais on d\'efinit maintenant de mani\`ere similaire $U_v$ comme le sous-groupe de $\oDv^\times\cong \GL_2(\oFv)$ des matrices triangulaires sup\'erieures modulo $v$ (cf. le cas I ci-dessus). On \'ecrit :
$$\rhobar|_\gFv \cong \begin{pmatrix} \overline{\xi}_v\omega & * \\ 0 & \overline{\xi}_v' \end{pmatrix}$$
pour des caract\`eres $\overline{\xi}_v,\overline{\xi}_v':\gFv \to k_E^\times$ et on pose $V_v \= \ker(\overline{\vartheta}_v)$ et $\overline{M}_v \= k_E(\overline{\vartheta}_v)$ o\`u $\overline{\vartheta}_v: U_v \to k_E^\times$ est d\'efini par $\overline{\vartheta}_v(g) \= \overline{\xi}_v(a)\overline{\xi}_v'(d)$ pour $g \equiv \smat{a & b \\ 0 & d}$ modulo $v$.

Pour $w\in S\cup \{v\}$, on rel\`eve la repr\'esentation $\overline{M}_w$ en caract\'eristique $0$ comme suit. Dans les cas I, II et III du \S\ \ref{facteur}, on a d\'efini $\overline{M}_w = k_E(\overline{\vartheta}_w)$ pour un caract\`ere $\overline{\vartheta}_w:U_w \to k_E^\times$, on pose alors $M_w \= \oE(\vartheta_w)$ o\`u $\vartheta_w$ est le relev\'e de Teichm\"uller de $\overline{\vartheta}_w$ et $\oE$ est suppos\'e suffisamment grand. Dans le cas IV (o\`u $w\!\nmid \!p$), on a defini $\overline{M}_w$ comme la r\'eduction d'un $K$-type $\vartheta_w$ pour un relev\'e $\sigma$ de $\rhobar|_{\gFw}$. On suppose maintenant que ce relev\'e est choisi tel que $\det\sigma = \varepsilon\psi|_{\gFw}$ et $\sigma(I_w) \stackrel{\sim}{\to} \rhobar(I_w)$ o\`u $\psi$ est le relev\'e de Teichm\"uller de $\omega^{-1}\det\rhobar$ (voir \S~\ref{deformation}). On d\'efinit $M_w$ comme un mod\`ele de $\vartheta_w$ sur $\oE$ (en supposant $\oE$ suffisamment grand pour que $\vartheta_w$ ait un tel mod\`ele). Dans tous les cas on a $\overline{M}_w = k_E \otimes_{\oE} M_w$ comme module pour $k_E[U_w/V_w]$.

On d\'efinit aussi des sous-groupes $U_w \subset U_w^0$ de $\oDw^\times \cong\GL_2(\oFw)$ pour les places finies $w \not\in S \cup \{v\}$ qui v\'erifient la condition :
\begin{itemize}
\item[{\bf(Q)}] le quotient des racines du polyn\^ome caract\'eristique de $\rhobar(\Fr_w)$ n'est pas dans $\{1,{\mathrm{N}}(w),{\mathrm{N}}(w)^{-1}\}$
\end{itemize}
o\`u ${\mathrm{N}}(w)\=|k_w| = \varepsilon^{-1}(\Fr_w)$. Pour une telle place $w$, on d\'efinit $U_w^0\subset \GL_2(\oFw)$ comme le sous-groupe des matrices triangulaires sup\'erieures modulo $w$ et $U_w$ comme l'ensemble des $g\in U_w^0$ tels que $g \equiv \smat{ a & b \\ 0 & a} \bmod w$.

Si maintenant $Q$ est un ensemble fini quelconque de places finies de $F$ qui ne sont pas dans $S \cup \{v\}$ et qui satisfont toutes ({\bf{Q}}), on pose :
$$U_Q \= \prod_{w\in S\cup\{v\}\cup Q} \!\!\!\!\!U_w\!\!\prod_{w\not\in S\cup\{v\}\cup Q}\!\!\!\!\!\oDw^\times\quad\mbox{et}\quad V_Q \= \prod_{w\in S\cup\{v\}} \!\!\!\!V_w  \prod_{w\in Q} U_w \!\!\prod_{w\not\in S\cup \{v\}\cup Q}\!\!\!\!\!\oDw^\times.$$
On voit $U_Q/V_Q = \prod_{w\in S\cup\{v\}}U_w/V_w$ comme agissant sur le $\oE$-module $M \= \otimes_{w\in S\cup\{v\}} M_w$ (le produit tensoriel est sur $\oE$). Comme $\psi$ est totalement pair (car $\rhobar$ est modulaire), on peut l'identifier par la th\'eorie du corps de classes \`a un caract\`ere des ad\`eles finis $\A_{F,f}^\times$ trivial sur $F^\times$. Comme $\psi$ co\"\i ncide avec $\vartheta_w$ sur $\oFw^\times$ pour $w\in S\cup\{v\}$ et est trivial sur $\oFw^\times$ sinon, on peut \'etendre l'action de $U_Q/V_Q$ sur $M$ via $\psi$ pour obtenir une action du groupe fini $G \= U_Q\A_{F,f}^\times/V_Q F^\times$. Notons que ce groupe est ind\'ependant de $Q$ et agit naturellement sur la courbe de Shimura $X_{V_Q}$, donc sur la cohomologie $H^1_\et(X_{V_Q,\Qbar},\oE)(1)$. On pose :
$$C_Q \= \Hom_{\oE[G]}\big(M, H^1_\et(X_{V_Q,\Qbar},\oE)(1)\big).$$

Pour $w\not\in \{v\} \cup S \cup Q$ on d\'efinit l'op\'erateur de Hecke $T_w$ agissant sur le $\oE$-module $H^1_{\et}(X_{V_Q,\Qbar},\oE)(1)$ par la double classe $\oDw^\times\smat{\varpi_w & 0 \\ 0 & 1} \oDw^\times$ pour un choix quelconque d'uniformisante $\varpi_w$ de $\oFw$. De m\^eme pour $w \in Q$, on d\'efinit $T_w$ agissant sur $H^1_{\et}(X_{V_Q,\Qbar},\oE)(1)$ par la double classe $U_w\smat{ \varpi_w & 0 \\ 0 & 1 } U_w$. Attention que l'op\'erateur $T_w$ pour $w\in Q$ {\it d\'epend} du choix de $\varpi_w$. Pour $w\in S'$ on d\'efinit $T_w$ exactement comme au \S\ \ref{facteur}. On v\'erifie comme dans la preuve du lemme~\ref{ws'} (et plus simplement m\^eme lorsque $w\not\in S \cup Q$) que chaque $T_w$ induit une action sur $C_Q$. De plus, ces op\'erateurs commutent deux \`a deux ce qui permet de d\'efinir une action de la $\oE$-alg\`ebre commutative $\tilde{\T}\=\oE[T_w,w\not\in \{v\} \cup S \backslash S']$ engendr\'ee par des variables formelles $T_w$ pour $w\not\in \{v\} \cup S \backslash S'$. On note $\T_Q$ l'image de $\tilde{\T}$ dans $\End_{\oE}(C_Q)$ : c'est un $\oE$-module de type fini.

Rappelons que, pour $w\in S'$, on a d\'efini des \'el\'ements $\alpha_w \in k_E^\times$. Pour $w\in Q$, on pose $\alpha_w \= {\mathrm{N}}(w)\alpha_w'$ pour un choix de valeur propre $\alpha_w'$ de $\rhobar(\Fr_w)$. On note $\phi_Q$ l'homomorphisme de $\oE$-alg\`ebres $\tilde{\T} \to k_E$ d\'efini par $\phi_Q(T_w) = \alpha_w$ pour $w\in S' \cup Q$ et $\phi_Q(T_w) = {\mathrm{N}}(w)\tr(\rhobar(\Fr_w))$ pour $w \not \in \{v\} \cup S \cup Q$, et on pose ${\gm}_Q \= \ker(\phi_Q)$.

On pose maintenant $\Sigma \= S \cup \{v\}$ et :
$$\Sigma' \=  S' \cup \{v\} \cup \{w,\ \mbox{$\rhobar|_{I_w}$ est\ r\'eductible\ non\ scind\'ee}\}.$$
Rappelons que $S$ contient toutes les places o\`u $\rhobar$ est ramifi\'ee, donc on a $\Sigma' \subset \Sigma$, et notons que les hypoth\`eses du \S\ \ref{deformation} sont satisfaites pour $\rhobar$, $\Sigma$ et $\Sigma'$. En particulier, pour tout $w\in \Sigma'$, on peut \'ecrire :
$$\rhobar|_\gFw \cong \begin{pmatrix} \overline{\xi}_w\omega & * \\ 0 & \overline{\xi}_w' \end{pmatrix}$$
pour des caract\`eres $\overline{\xi}_w,\overline{\xi}_w':\gFw \to k_E^\times$. De plus, on a toujours $\overline{\xi}_w=\overline{\xi}_w'$ si $w\!\nmid \!p$ (par {\bf(H0)} ou car l'extension est non 
scind\'ee). Lorsqu'il y a plusieurs choix possibles pour le caract\`ere $\overline{\xi}_w$, on en fixe un (voir la remarque \ref{choix}(i)).  

\begin{prop}\label{CQmQ} 
On a des isomorphismes $\Qpbar \otimes_{\oE} C_{Q,\gm_Q} \cong \oplus \rho_\pi$ et $\Qpbar \otimes_{\oE} \T_{Q,\gm_Q} \cong \oplus \Qpbar$ o\`u $\Qpbar \otimes_{\oE} \T_{Q,\gm_Q}$ agit composante par composante sur $\Qpbar \otimes_{\oE} C_{Q,\gm_Q}$ et o\`u les sommes directes sont sur les repr\'esentations automorphes $\pi$ de $(D\otimes \A)^\times$ telles que :
\begin{itemize}
\item $\pi_\infty \cong \sigma_2$ (cf. \S\ \ref{globalprelim})
\item $\det\rho_\pi = \psi\varepsilon$
\item $\rho_\pi$ est un relev\'e de $\rhobar$
\item $\rho_\pi$ est non ramifi\'ee en dehors de $\Sigma \cup Q$
\item si $w\in \Sigma\backslash \Sigma'$, alors $\rho_\pi(I_w) \stackrel{\sim}{\to} \rhobar(I_w)$
\item si $w\in \Sigma'$, alors $\rho_\pi|_\gFw \cong  \smat{\eta_w\varepsilon & * \\ 0 & \eta_w'}$ pour un caract\`ere $\eta_w$ relevant $\overline{\xi}_w$ tel que $\eta_w(I_w) \stackrel{\sim}{\to} \overline{\xi}_w(I_w)$
\item si $w|p$, $\overline{\xi}_w|_{I_w} = \overline{\xi}'_w|_{I_w}$ et $\overline{\xi}_w^{-1}\rhobar|_{\gFw}$ est la fibre g\'en\'erique d'un sch\'ema en groupes fini et plat sur $\oFw$, alors $\eta_w^{-1}\rho_\pi|_{\gFw}$ est la fibre g\'en\'erique d'un groupe $p$-divisible sur $\oFw$.
\end{itemize}
De plus, l'ensemble de telles $\pi$ est non vide.
\end{prop}
\begin{proof} 
(i) On a des isomorphismes :
$$\Qpbar \otimes_{\oE} C_{Q,\gm_Q} \cong \Hom_{\oE[G]}\big(M, \Qpbar\otimes_{\oE} H^1_\et(X_{V_Q,\Qbar},\oE)(1)_{{\gm}_Q}\big)$$ 
et : 
$$\Qpbar\otimes_{\oE} H^1_\et(X_{V_Q,\Qbar},\oE)(1)_{{\gm}_Q} \cong \bigoplus_{\gp} H^1_\et(X_{V_Q,\Qbar},\Qpbar)(1)_\gp$$
o\`u la somme est sur les id\'eaux premiers $\gp$ de $\Qpbar\otimes_{\oE} \tilde{\T}$ tels que $\gp \cap \tilde{\T} \subset {\gm}_Q$ et qui sont dans le support de $H^1_\et(X_{V_Q,\Qbar},\Qpbar)$. En tant que $\Qpbar\otimes_{\oE}\tilde{\T}$-module, on a une d\'ecomposition :
$$H^1_\et(X_{V_Q,\Qbar},\Qpbar)(1)  = \oplus_{\pi} \big(\rho_\pi \otimes_{\Qpbar} \pi_f^{V_Q}\big)$$
o\`u la somme est sur les repr\'esentations automorphes $\pi$ de $(D\otimes_\Q\A)^\times$ telles que $\pi_\infty \cong \sigma_2$ et $\pi_f$ est comme au \S\ \ref{globalprelim} mais avec les scalaires \'etendus \`a $\Qpbar$ au lieu de $\C$. On en d\'eduit que $\Qpbar \otimes_{\oE} C_{Q,\gm_Q}$ est la somme directe sur de telles $\pi$ des $\Qpbar$-espaces vectoriels :
$$\rho_\pi \otimes_{\Qpbar} \Hom_{\oE[G]}(M,\oplus_\gp(\pi_f^{V_Q})_\gp)$$
pour $\gp$ comme ci-dessus. Comme cet espace vectoriel est nul sauf si $\psi$ est le caract\`ere central de $\pi$, ou de mani\`ere \'equivalente $\det\rho_\pi = \psi\varepsilon$, on a seulement besoin de sommer sur ces $\pi$ l\`a et on peut alors remplacer $G$ par $U_Q/V_Q$. Dans le produit tensoriel restreint usuel $\pi_f = \otimes_w' \pi_w$ l'op\'erateur de Hecke $T_w$ agit par l'identit\'e en dehors de $w$. Comme on a :
$$\pi_f^{V_Q} = \left(\otimes_{w\in\Sigma} \pi_w^{V_w}\right)\otimes\left(\otimes_{w\in Q} \pi_w^{U_w}\right) \otimes \Big(\otimes'_{w\not\in \Sigma\cup Q}\pi_w^{\oDw^\times}\Big)$$
on en d\'eduit que les id\'eaux premiers $\gp$ de $\Qpbar\otimes_{\oE} \tilde{\T}$ tels que $(\pi_f^{V_Q})_\gp\ne 0$ sont pr\'ecis\'ement les noyaux des homomorphismes d\'efinis par $T_w \mapsto a_w$ ($w \not\in \Sigma \backslash S'$) o\`u $a_w\in \Qpbar$ est une valeur propre de $T_w$ sur le facteur en $w$. Notons que $\gp\cap\tilde{\T}\subset \gm_Q$ (pour le $\gp$ correspondant) si et seulement si $a_w$ se r\'eduit sur $\alpha_w$ pour $w\in S'\cup Q$ et sur ${\mathrm{N}}(w)\tr(\rhobar(\Fr_w))$ pour $w \not\in \Sigma\cup Q$. On en d\'eduit :
$$\Hom_{\oE[G]}(M,\oplus_\gp(\pi_f^{V_Q})_\gp)\cong \otimes'_w X_w$$
o\`u :
$$X_w \= \left\{\begin{array}{ll}
\Hom_{\oE[U_w]}(M_w,\pi_w^{V_w}) & \mbox{si $w\in \Sigma\backslash S'$}\\
\Hom_{\oE[U_w]}(M_w,(\pi_w^{V_w})_{\alpha_w}) & \mbox{si $w\in S'$}\\
(\pi_w^{U_w})_{\alpha_w} & \mbox{si $w\in Q$}\\
(\pi_w^{\oDw^\times})_{{\mathrm{N}}(w)\tr(\rhobar(\Fr_w))} & \mbox{sinon,}
\end{array}\right.$$
et o\`u $\alpha$ en indice veut dire la somme des espaces propres g\'en\'eralis\'es pour $T_w$ associ\'es aux valeurs propres $a_w \in \Zpbar$ qui se r\'eduisent sur $\alpha\in k_E$. Pour $w\not\in \Sigma \cup Q$, les espaces $\pi_w^{\oDw^\times}$ sont de dimension $1$ et $a_w = {\mathrm{N}}(w)\tr(\rho_\pi(\Fr_w))$, donc la dimension de $\otimes'_{w\not\in\Sigma \cup Q}(\pi_w^{\oDw^\times})_{{\mathrm{N}}(w)\tr(\rhobar(\Fr_w))}$ est $1$ si $\rho_\pi$ se r\'eduit sur $\rhobar$ et $0$ sinon. On peut donc supposer que $\rho_\pi$ est un relev\'e de $\rhobar$ en d\'eterminant les autres facteurs.

(ii) Supposons d'abord $w\in S \backslash S'$ (donc en particulier $w\!\nmid \!p$). Si $\rhobar|_{\gFw}$ est irr\'eductible, alors le facteur en $w$ est $\Hom_{\oDw^\times}(\vartheta_w,\pi_w)$, qui est non nul (et n\'ecessairement de dimension $1$) si et seulement si $\rho_\pi|_{I_w} \cong \sigma|_{I_w}$, i.e. $\rho_\pi(I_w) \stackrel{\sim}{\to} \rhobar(I_w)$ (cf. le choix de $\sigma$ au d\'ebut de cette section).

Si $\rhobar|_{\gFw}$ est r\'eductible et $w \not\in \Sigma'$, alors $\rhobar|_{I_w}\cong \overline{\xi}_w \oplus \overline{\xi}_w'$ pour des caract\`eres $\overline{\xi}_w,\overline{\xi}_w':I_w \to k_E^\times$ et $\ind_{U_w}^{\oDw^\times}\vartheta_w$ est un $K$-type pour $[\sigma,0]$ (cf. \S\ \ref{lifts0}) o\`u $\sigma|_{I_w} \cong [\overline{\xi}_w] \oplus [\overline{\xi}_w']$. On en d\'eduit comme dans le cas pr\'ec\'edent que :
$$\Hom_{U_w}(\vartheta_w,\pi_w) \cong \Hom_{\oDw^\times}(\ind_{U_w}^{\oDw^\times}\vartheta_w,\pi_w)$$
est non nul (n\'ecessairement de dimension $1$) si et seulement si $\rho_\pi(I_w) \stackrel{\sim}{\to} \rhobar(I_w)$.

Si $\rhobar|_{\gFw}$ est r\'eductible et $w \in \Sigma'$, alors $\rhobar|_{\gFw} \cong \overline{\xi}_w \otimes \smat{\omega & * \\ 0 & 1 }$ pour un caract\`ere $\overline{\xi}_w: \gFw \to k_E^\times$ et $\overline{\xi}_w^{-1}\otimes\rhobar|_{\gFw}$ est ramifi\'ee. Dans ce cas $\Hom_{U_w}(\vartheta_w,\pi_w)$ est non trivial si et seulement si $\pi_w$ est de la forme $\eta_w\circ\det \otimes \pi'_w$ o\`u $\eta_w$ est un relev\'e de $\overline{\xi}_w$ tel que $\eta_w(I_w) = \overline{\xi}_w(I_w)$ et $\pi'_w$ est soit une s\'erie principale non ramifi\'ee soit la repr\'esentation de Steinberg. Comme $\rho_\pi|_{\gFw}$ rel\`eve $\rhobar|_{\gFw}$ et que $\overline{\xi}_w^{-1}\otimes\rhobar|_{\gFw}$ est ramifi\'ee, on ne peut avoir pour $\pi_w'$ une s\'erie principale non ramifi\'ee. On en d\'eduit que $\Hom_{U_w}(\vartheta_w,\pi_w)$ est non nul (n\'ecessairement de dimension $1$) si et seulement si $\rho_\pi|_\gFw \cong  \eta_w\otimes \smat{\varepsilon & * \\ 0 & 1}$, qui est \'equivalent \`a $\rho_\pi|_{\gFw}$ de la forme demand\'ee dans l'\'enonc\'e.

(iii) Supposons maintenant $w \in S'$. Consid\'erons d'abord le cas $w\!\nmid \!p$. On a alors $\overline{\xi}'_w = \overline{\xi}_w$, $D_w$ ramifi\'ee et $U_w = \oDw^\times$. D'apr\`es les d\'efinitions de $M_w$ et de $T_w$, on voit que $\Hom_{\oE[U_w]}(M_w,\pi_w)_{\alpha_w}\neq 0$ (n\'ecessairement de dimension $1$) si et seulement si $\pi_w = \eta_w\circ\det$ pour un caract\`ere $\eta_w:F_w^\times\to \Zpbar^\times$ tel que $\eta_w|_{\oFw^\times} = [\overline{\xi}_w]|_{\oFw^\times}$ et $a_w = \eta_w(\det(\Pi_{D_w}))$ rel\`eve $\alpha_w$ (rappelons que $\Pi_{D_w}$ est une uniformisante de $\oDw$), ce qui par compatibilit\'e local-global est \'equivalent \`a $\rho_\pi|_{\gFw}$ de la forme demand\'ee.

Supposons $w|p$, $w\neq v$ et $\overline{\xi}_w|_{I_w}\neq \overline{\xi}'_w|_{I_w}$. Dans ce cas $\ind_{U_w}^{\oDw^\times}\vartheta_w$ est un $K$-type pour $[[\overline{\xi}_w]\oplus[\overline{\xi}'_w],0]$, et donc $\Hom_{\oE[U_w]}(M_w,\pi_w)$ est non nul si et seulement si $\WD(\rho_\pi|_{\gFw}) = \eta_w\norm\oplus \eta_w'$ (avec $N=0$) pour des caract\`eres $\eta_w,\eta_w':W_{w} \to \Qpbar^\times$ tels que $\eta_w|_{I_w} = [\overline{\xi}_w]|_{I_w}$, $\eta_w'|_{I_w} = [\overline{\xi}'_w]|_{I_w}$ et $\eta_w\eta_w' = \psi|_{\gFw}$. De plus dans ce cas $\Hom_{\oE[U_w]}(M_w,\pi_w)$ est de dimension $1$ avec $T_w$ agissant par $a_w = \eta_w(\varpi_w)$, de sorte que $\Hom_{\oE[U_w]}(M_w,\pi_w)_{\alpha_w}$ est non nul (n\'ecessairement de dimension $1$) si et seulement si la repr\'esentation de Weil-Deligne $\WD(\rho_\pi|_{\gFw})$ est comme ci-dessus avec $\eta_w(\varpi_w)\in \Zpbar^\times$ relevant $\alpha_w=\overline{\xi}_w(\varpi_w)$ (et donc avec aussi $\eta_w'(\varpi_w)\in \Zpbar^\times$). Une repr\'esentation potentiellement semi-stable $\rho_\pi|_{\gFw}$ a une telle repr\'esentation de Weil-Deligne si et seulement si $\rho_\pi|_{\gFw}$ est comme dans l'\'enonc\'e.

Supposons $w|p$, $w\neq v$ et $\overline{\xi}_w|_{I_w}= \overline{\xi}'_w|_{I_w}$. La preuve est similaire au cas pr\'ec\'edent avec les modifications suivantes. Si $\overline{\xi}_w^{-1}\otimes \rhobar|_{\gFw}$ n'est pas la fibre g\'en\'erique \ d'un sch\'ema \ en \ groupes fini et plat sur $\oFw$, alors on a $\WD(\rho_\pi|_{\gFw}) = \eta_w\norm\oplus \eta_w$ avec $N\neq 0$ et $T_w$ agissant par $\eta_w(\varpi_w)$. Si $\overline{\xi}_w^{-1}\otimes \rhobar|_{\gFw}$ est la fibre g\'en\'erique d'un sch\'ema en groupes fini et plat sur $\oFw$, alors $U_w = \GL_2(\oFw)$ et on a $\WD(\rho_\pi|_{\gFw}) = \eta_w\norm\oplus \eta_w'$ (avec $N=0$) pour des $\eta_w,\eta_w'$ comme dans le cas pr\'ec\'edent, mais maintenant $T_w$ agit par $\eta_w(\varpi_w) + \eta_w'(\varpi_w){\mathrm{N}}(w)$. C'est une unit\'e $p$-adique si et seulement si soit $\eta_w(\varpi_w)$ soit $\eta_w'(\varpi_w){\mathrm{N}}(w)$ l'est. Quitte \`a \'echanger $\eta_w\norm$ et $\eta_w'$ on voit comme pr\'ec\'edemment que $\rho_\pi|_{\gFw}$ est comme dans l'\'enonc\'e. 

Supposons maintenant $w = v$. Comme $\rhobar|_{\gFv}$ est (r\'eductible) g\'en\'erique, on a $\overline{\xi}_v|_{I_v} \neq \overline{\xi}'_v|_{I_v}$ et encore une fois $\Hom_{\oE[U_v]}(M_v,\pi_v)$ non nul (n\'ecessaire\-ment de dimension $1$) si et seulement si $\WD(\rho_\pi|_{\gFv}) = \eta_v\norm\oplus \eta_v'$ avec $N=0$ et des caract\`eres $\eta_v,\eta_v'$ comme ci-dessus. Comme $\rho_\pi|_{\gFv}$ rel\`eve $\rhobar|_{\gFv}$, on en d\'eduit par la proposition~\ref{J'} que dans ce cas le $\oE$-module fortement divisible qui donne $\rho_\pi|_{\gFv}$ est de type $\emptyset$ (cf. d\'efinition \ref{J}), et par le th\'eor\`eme~\ref{thm_reduction} que $\eta_v(p)$ est une unit\'e $p$-adique qui rel\`eve $\overline{\xi}_v(p )$. (Notons que le caract\`ere $\eta'(\emptyset)$ du \S\ \ref{prel} est la restriction \`a $I_v$ du caract\`ere not\'e ici $\eta_v$.) On conclut exactement comme pr\'ec\'edemment.

(iv) Supposons enfin $w \in Q$ (donc en particulier $w\!\nmid \!p$). Par la condition {\bf(Q)}, on a $\rhobar|_{\gFw} \cong \overline{\xi}_w \oplus \overline{\xi}_w'$ pour des caract\`eres distincts $\overline{\xi}_w,\overline{\xi}_w':\gFw \to k_E^\times$ dont le quotient n'est pas $\omega^{\pm1}$. Comme $\rho_\pi$ rel\`eve $\rhobar$, on en d\'eduit $\rho_\pi|_{\gFw}\cong \eta_w \oplus \eta'_w$ pour des caract\`eres $\eta_w$ et $\eta'_w$ qui rel\`event $\overline{\xi}_w$ et $\overline{\xi}'_w$. On a donc $\pi_w \cong \Ind_{{\rm B}(F_w)}^{\GL_2(F_w)}(\eta_w\norm^{-1}\otimes \eta'_w)$ de sorte que $\Hom_{U_w}(M_w,\pi_w)$ est de dimension $2$ avec $\eta_w(\varpi_w){\mathrm{N}}(w)$ et $\eta_w'(\varpi_w){\mathrm{N}}(w)$ pour valeurs propres de $T_w$. Comme une seule de ces valeurs propres se r\'eduit sur $\alpha_w$, on obtient que $\Hom_{U_w}(M_w,\pi_w)_{\alpha_w}$ est de dimension $1$.

Ceci ach\`eve la preuve de $\Qpbar \otimes_{\oE} C_{Q,\gm_Q} \cong \oplus \rho_\pi$ avec $\pi$ comme dans l'\'enonc\'e. Notons que $\tilde{\T}$ agit sur la composante en $\pi$ via un homomorphisme $\tilde{\T} \to \Qpbar$ envoyant $T_w$ sur ${\mathrm{N}}(w)\tr(\rho_\pi(\Fr_w))$ pour $w\not\in \Sigma\cup Q$. Comme ces homomorphismes sont tous distincts, la conclusion concernant $\T_{Q,\gm_Q}$ s'ensuit.

(v) Le fait que l'ensemble de ces $\pi$ est non vide d\'ecoule du th\'eor\`eme~\ref{lifts} appliqu\'e avec $S=\Sigma$, $\psi=\psi\varepsilon$ et $T=$ l'ensemble des places de $F$ divisant $p$. Posons $\overline{\mu}_w = \overline{\xi}_w$ pour $w\in\Sigma'$ et choisissons les types de Weil-Deligne $[r_w,N_w]$ pour $w\in \Sigma$ comme ceux apparaissant dans la preuve ci-dessus. Un relev\'e $\rho_w$ provenant d'un point \`a valeurs dans $\Qpbar$ quelconque de l'anneau $R_w^\triangle$ du lemme~\ref{lem_deformations} est du type demand\'e. Le th\'eor\`eme~\ref{lifts} produit alors un relev\'e $\rho = \rho_\pi$ qui a les propri\'et\'es requises.
\end{proof}

\begin{rem}
{\rm Notons que, dans la preuve ci-dessus, le cas $w=v$ se distingue des cas $w|p$, $w\neq v$ car nous n'introduisons pas d'op\'erateur de Hecke $T_v$ en $v$. Nous aurions pu, comme pour les places $w|p$, $w\neq v$, rajouter un tel op\'erateur et remplacer dans l'\'enonc\'e les hypoth\`eses $F_v$ non ramifi\'ee, $D_v$ d\'eploy\'ee et $\rhobar|_{\gFv}$ r\'eductible g\'en\'erique par juste $D_v$ d\'eploy\'ee et $\rhobar|_{\gFv}$ r\'eductible non scalaire. Mais cela aurait alors compliqu\'e l'\'enonc\'e et la preuve du th\'eor\`eme \ref{thm:multone} ci-apr\`es qui, de toute mani\`ere, requi\`erent ces hypoth\`eses plus fortes en $v$.}
\end{rem}

Nous aurons besoin d'une variante de la proposition \ref{CQmQ} dans le cas $Q=\emptyset$. \'Ecrivons :
\begin{eqnarray}\label{ecrit}
\rhobar|_\gFv \cong \begin{pmatrix}\overline{\xi}_v\omega & * \\ 0 & \overline{\xi}_v'\end{pmatrix}\cong \begin{pmatrix}\overline{\xi}_v\theta_v^{-1}\omega & * \\ 0 & \overline{\xi}_v'\theta_v^{-1}\end{pmatrix}\otimes \theta_v
\end{eqnarray}
o\`u $\theta_v$ est l'unique caract\`ere de $\gFv$ qui co\"\i ncide avec $\overline{\xi}_v'$ sur l'inertie et qui vaut $1$ en $p$ (si $\rhobar|_\gFv$ est scind\'ee on fait un choix pour $\overline{\xi}_v$, $\overline{\xi}_v'$, et donc pour $\theta_v$, cf. les remarques \ref{comment}(ii) et \ref{choix}(i)). Notons ${\mathcal S}_v$ l'ensemble des plongements de $k_v$ dans $k_E$, on a :
$$\overline{\xi}_v\theta_v^{-1}=\nr(\lambda_v)\prod_{\sigma\in {\mathcal S}_v}\omega_{\sigma}^{r_{v,\sigma}}\ \ \ {\rm et}\ \ \ \ \overline{\xi}_v'\theta_v^{-1}=\nr(\mu_v)$$
pour $\lambda_v,\mu_v \in k_E^{\times}$ et des $r_{v,\sigma}\in \{0,\cdots,p-3\}$ uniques avec $(r_{v,\sigma})_{\sigma}\ne (0,\cdots,0),(p-3,\cdots,p-3)$. Rappelons que l'on a alors d\'efini en (\ref{carJ}) le caract\`ere $\eta'(J)\otimes\eta(J)$ de ${\rm I}({\mathcal O}_{F_v})=U_v$ pour tout $J\subseteq {\mathcal S}_v$. D\'efinissons les $\oE$-modules $M(J)$ et $C(J)$ exactement comme l'on a d\'efini $M$ et $C_\emptyset$ au \S~3.5 mais en rempla\c cant $M_v$ par $M_v(J) \= \oE(\eta'(J)\otimes\eta(J))$ (en particulier $M_v(\emptyset) = M_v$). On note $\T(J)$ le $\oE$-module de type fini image de $\tilde{\T}$ dans $\End_{\oE}(C(J))$ et on pose $\gm = \gm_\emptyset$.

\begin{prop}\label{CQmQJ} 
Pour tout $J\subseteq {\mathcal S}_v$, on a des isomorphismes $\Qpbar \otimes_{\oE} C(J)_{\gm} \cong \oplus \rho_\pi$ et $\Qpbar \otimes_{\oE} \T(J)_{\gm} \cong \oplus \Qpbar$ o\`u $\Qpbar \otimes_{\oE} \T(J)_{\gm}$ agit composante par composante sur $\Qpbar \otimes_{\oE} C(J)_{\gm}$ et o\`u les sommes directes sont sur les repr\'esentations automorphes $\pi$ de $(D\otimes \A)^\times$ telles que :
\begin{itemize}
\item $\pi_\infty \cong \sigma_2$
\item $\det\rho_\pi = \psi\varepsilon$
\item $\rho_\pi$ est un relev\'e de $\rhobar$
\item $\rho_\pi$ est non ramifi\'ee en dehors de $\Sigma$
\item si $w\in \Sigma\backslash \Sigma'$, alors $\rho_\pi(I_w) \stackrel{\sim}{\to} \rhobar(I_w)$
\item si $w\in \Sigma'\backslash\{v\}$, alors $\rho_\pi|_\gFw \cong  \smat{\eta_w\varepsilon & * \\ 0 & \eta_w'}$ pour un caract\`ere $\eta_w$ relevant $\overline{\xi}_w$ tel que $\eta_w(I_w) \stackrel{\sim}{\to} \overline{\xi}_w(I_w)$
\item si $w=v$, alors $\rho_\pi|_{\gFv}$ est potentiellement Barsotti-Tate de type de Weil-Deligne $[\eta'(J)\oplus\eta(J),0]$
\item si $w|p$, $\overline{\xi}_w|_{I_w} = \overline{\xi}'_w|_{I_w}$ et $\overline{\xi}_w^{-1}\rhobar|_{\gFw}$ est la fibre g\'en\'erique d'un sch\'ema en groupes fini et plat sur $\oFw$, alors $\eta_w^{-1}\rho_\pi|_{\gFw}$ est la fibre g\'en\'erique d'un groupe $p$-divisible sur $\oFw$.
\end{itemize}
De plus, l'ensemble de telles $\pi$ est non vide.
\end{prop}
\begin{proof}
La preuve est en tout point similaire \`a celle de la proposition \ref{CQmQ}, sauf pour la derni\`ere assertion (cf. (v) dans la preuve de {\it loc.cit.}) o\`u, pour pouvoir utiliser le th\'eor\`eme~\ref{lifts}, il faut v\'erifier que la repr\'esentation $\rhobar|_{\gFv}$ admet un relev\'e potentiellement Barsotti-Tate de type de Weil-Deligne $[\eta'(J)\oplus\eta(J),0]$. Cela se d\'eduit par exemple de \cite[Thm.5.1.1(i)]{BM} en se souvenant que $\tau(\emptyset)\in {\mathcal D}(\rhobar|_{\gFv})$ est un constituant de $\ind_{{\rm I}({\mathcal O}_{F_v})}^{\GL_2({\mathcal O}_{F_v})}\overline\eta'(J)\otimes\overline\eta(J)$, cf. \S\ \ref{spec} (cela se d\'eduit aussi du lemme \ref{wts_types} en utilisant que $\rhobar$ est mo\-dulaire de poids $\tau(\emptyset)$ en $v$ par la preuve du th\'eor\`eme \ref{thm:multone}(i) ci-apr\`es).
 \end{proof}
 
Pour terminer cette section, nous revenons au cadre du \S\ \ref{facteur} (o\`u il n'y a plus d'hypoth\`ese sur $v$) pour montrer le corollaire suivant qui compl\`ete le corollaire \ref{extra}.
 
\begin{cor}\label{extra+}
Supposons $p>2$, $\rhobar:\gF \to \GL_2(k_E)$ modulaire, $\rhobar|_{{\Gal}(\Qbar/F(\sqrt[p]{1}))}$ irr\'eductible et, si $p=5$, l'image de $\rhobar({\Gal}(\Qbar/F(\sqrt[p]{1})))$ dans ${\rm PGL}_2(k_E)$ non isomorphe \`a ${\rm PSL}_2(\F_5)$. Supposons \'egalement satisfaites les hypoth\`eses {\bf(H0)}, {\bf(H1)}, {\bf(H2)} (cf. \S\S\ \ref{lifts0} et \ref{facteur}). Si $D$ est ramifi\'ee en $v$ alors la repr\'esentation $\pi_{D,v}(\rhobar)$ en (\ref{piDv}) est de longueur infinie.
\end{cor}
\begin{proof}
Soient $\tau_v$, $\vartheta_{v,m}$ et $r_{v,m}$ comme dans la preuve du corollaire \ref{extra} et $\psi \= [\omega^{-1}\det\rhobar]$. On applique le th\'eor\`eme \ref{lifts} avec pour $T$ l'ensemble des places $w\neq v$ divisant $p$, pour $S$ l'ensemble $\Sigma=S\cup\{v\}$, avec $\overline{\mu}_w \= \overline{\xi}_w$ si $w \in \Sigma' \setminus \{v\}$, avec le type de Weil-Deligne $[r_{v,m},0]$ en $v$ et avec les types de Weil-Deligne de la preuve de la proposition \ref{CQmQ} aux places $w \neq v$ de $\Sigma$, c'est-\`a-dire~: 
\begin{itemize}
\item $[\sigma,0]$ comme dans le (ii) de cette preuve si $w \not\in \Sigma'$
\item $[[\overline{\xi}_w]\norm\oplus [\overline{\xi}_w],N_w]$ avec $N_w \neq 0$ si $w\in \Sigma'$ mais $w\nmid p$, ou bien si $w|p$, $\overline{\xi}_w|_{I_w} = \overline{\xi}'_w|_{I_w}$ et $\overline{\xi}_w^{-1}\otimes\rhobar|_{\gFw}$ n'est pas la fibre g\'en\'erique d'un sch\'ema en groupes fini et plat sur $\oFw$
\item $[[\overline{\xi}_w]\oplus [\overline{\xi}'_w],0]$ sinon.
\end{itemize}
(L'existence des relev\'es locaux se d\'eduit des lemmes \ref{wts_types} et \ref{lem_deformations}.) Soit $\pi_m$ une repr\'esentation automorphe de $(D\otimes_\Q\A)^\times$ donn\'ee par le th\'eor\`eme \ref{lifts} et la correspondance de Jacquet-Langlands. Sur une extension suffisamment grande $E_m$ de $E$, on obtient comme dans la preuve de la proposition \ref{CQmQ} :
$$\Hom_{\oE[U^v]}(M^v,\Qpbar\otimes_{{\mathcal{O}}_{E_m}}\pi_m^0) \cong  \Qpbar\otimes_{\Qbar}\pi_{m,v}$$
o\`u $M^v \= \otimes_{w \in S}M_w$, $\pi_m^0$ est un ${\mathcal{O}}_{E_m}$-r\'eseau de $\pi_{m,f}$ comme en (\ref{pi0}) et $T_w$ pour $w \in S'$ agit par multiplication par un scalaire dans ${\mathcal{O}}_{E_m}^\times$ qui se r\'eduit sur $\alpha_w\in k_E^\times$. On d\'eduit de la proposition \ref{app:SCoD} que le ${\mathcal{O}}_{E_m}$-rang de $\Hom_{\oE[U^v]}(M^v,\pi_m^0)$ est $Cq_v^m$ pour un entier $C$ strictement positif ind\'ependant de $m$. Soit $\varpi_{E_m}$ une uniformisante de ${\mathcal{O}}_{E_m}$, l'injection naturelle $\pi_m^0/\varpi_{E_m} \hookrightarrow k_{E_m} \otimes_{k_E}\pi_D(\rhobar)$ (cf. lemme \ref{pi0bar}) induit des injections compatibles avec l'action de $T_w$ pour $w\in S'$ :
\begin{multline*}
\Hom_{\oE[U^v]}(M^v,\pi_m^0)/\varpi_{E_m} \hookrightarrow \Hom_{k_E[U^v]}(\overline M^v,\pi_m^0/\varpi_{E_m})\\
\hookrightarrow k_{E_m} \otimes_{k_E}\Hom_{k_E[U^v]}(\overline M^v,\pi_{D}(\rhobar))
\end{multline*}
et l'image de l'injection compos\'ee tombe donc dans :
$$k_{E_m}\otimes_{k_E}\Hom_{k_E[U^v]}(\overline{M}^v,\pi_D(\rhobar))[\gm']=k_{E_m}\otimes_{k_E}\pi_{D,v}(\rhobar).$$
On en d\'eduit que $\pi_{D,v}(\rhobar)$ est de dimension (sur $k_E$) sup\'erieure ou \'egale \`a $Cq_v^m$ pour tout $m$ et on conclut comme dans la preuve du corollaire \ref{extra}.
\end{proof}

\subsection{Multiplicit\'e un II}\label{unII}

On montre que le module $C_\gm=C_{\emptyset,\gm_\emptyset}$ est libre de rang $2$ sur $\T_\gm=\T_{\emptyset,\gm_\emptyset}$. 

Beaucoup des arguments sont similaires \`a ceux que l'on trouve dans les articles g\'en\'eralisant \cite{Wi} et \cite{TW}. Une diff\'erence cependant est que le but de ces articles \'etant typiquement de montrer la modularit\'e de repr\'esentations galoisiennes, les m\'ethodes de changement de base peuvent y \^etre utilis\'ees afin de simplifier les hypoth\`eses sur le niveau. Nous donnons par cons\'equent des preuves compl\`etes aux endroits o\`u cette diff\'erence intervient. On conserve les notations des sections \ref{facteur}, \ref{deformation} et \ref{unI} et les hypoth\`eses de la section \ref{unI}.

Pour une repr\'esentation automorphe $\pi$ de $(D\otimes \A)^\times$ comme dans la proposition \ref{CQmQ}, la repr\'esentation $\rho_\pi$ est naturellement d\'efinie sur un objet de la cat\'egorie $\CNL$ (cf. \S\ \ref{deformation}) : prendre par exemple $\{a\in {\mathcal{O}}_{E'},\ \overline{a} \in k_E\}$ pour une extension $E'$ suffisamment grande de $E$. Elle d\'efinit donc un morphisme de $\oE$-alg\`ebres $R_Q \to \Qpbar$ qui envoie ${\mathrm{N}}(w)\tr(\rho_Q^\univ(\Fr_w))$ sur la valeur propre de $T_w$ agissant sur $\pi_w^{\GL_2(\oFw)}$ pour $w\not\in \Sigma\cup Q$ (o\`u $\rho_Q^\univ:\gF \to \GL_2(R_Q)$ est dans la classe d'\'equivalence de la d\'eformation universelle de $\rhobar$ correspondante). Par la proposition \ref{CQmQ}, on obtient donc un morphisme $R_Q \to \Qpbar \otimes_{\oE} \T_{Q,\gm_Q}$ qui envoie ${\mathrm{N}}(w)\tr(\rho_Q^\univ(\Fr_w))$ sur $1\otimes T_w$ pour $w\not\in \Sigma\cup Q$. Comme $R_Q$ est topologiquement engendr\'e sur $\oE$ par les traces $\{\tr(\rho_Q^\univ(\Fr_w)),\ w\not\in\Sigma\cup Q\}$, ce morphisme est donc \`a valeurs dans $\T_{Q,\gm_Q}$. On voit aussi que si $Q'\subseteq Q$, on a un diagramme commutatif :
$$\begin{array}{ccc}  
R_{Q} & \longrightarrow & R_{Q'} \\
\downarrow&       & \downarrow \\
\T_{Q,\gm_Q} & \longrightarrow & \T_{Q',\gm_{Q'}}
\end{array}$$
o\`u la fl\`eche du haut est la surjection canonique, les fl\`eches verticales sont celles que l'on vient de d\'efinir et la fl\`eche du bas peut \^etre caract\'eris\'ee comme l'unique application (surjective) qui envoie $T_w \in \T_{Q,\gm_Q}$ sur l'unique racine $\tilde{\alpha}_w \in \T_{Q',\gm_{Q'}}$ de $X^2 - T_w X + \psi(\varpi_w){\mathrm{N}}(w)$ (donn\'ee par le lemme de Hensel) qui se r\'eduit sur $\alpha_w$ si $w \in Q\backslash Q'$ et sur $T_w$ sinon. Pour voir qu'une telle application existe et rend le diagramme commutatif, notons qu'il suffit de le v\'erifier apr\`es avoir tensoris\'e la ligne du bas par $\Qpbar$ et avoir projet\'e sur chaque composante de $\Qpbar \otimes_{\oE} \T_{Q',\gm_{Q'}}$ dans la d\'ecomposition de la proposition \ref{CQmQ}. L'application de $R_Q$ vers la composante correspondant \`a $\pi$ se factorise par la composante de $\Qpbar \otimes_{\oE} \T_{Q,\gm_{Q}}$ correspondant \`a $\pi$ o\`u l'image de $T_w$ y est d\'etermin\'ee par $\rho_\pi$ comme suit :
\begin{itemize}
\item[(i)]si $w\not\in \Sigma\cup Q$ alors l'image est $\mathrm{N}(w)\tr(\rho_\pi(\Fr_w))$
\item[(ii)]si $w\in S'$ alors il existe un unique caract\`ere $\eta_w$ de $\gFw$ relevant $\overline{\xi}_w$ tel que $\rho_\pi|_{\gFw} \cong  \smat{\eta_w\varepsilon & * \\ 0 & \eta_w'}$ et $\eta_w(I_w) = \overline{\xi}_w(I_w)$, l'image est alors $\eta_w(\varpi_w)$ sauf si $w|p$ et 
$\eta_w^{-1}\rho_\pi|_{\gFw}$ est la fibre g\'en\'erique d'un groupe $p$-divisible auquel cas l'image est $\eta_w(\varpi_w) + \mathrm{N}(w)\eta_w'(\varpi_w)$
\item[(iii)]si $w \in Q$ alors $\rho_\pi|_{\gFw} \cong \eta_w \oplus \eta'_w$ o\`u $\eta_w(\varpi_w)$ rel\`eve $\alpha_w$ mais pas $\eta'_w(\varpi_w)$ (voir (iv) dans la preuve de la proposition \ref{CQmQ}) et l'image est alors $\mathrm{N}(w)\eta_w(\varpi_w)$.
\end{itemize}
Consid\'erons maintenant le morphisme canonique $R_Q^\square \rightarrow  \T_{Q,\gm_Q}^\square \= R_Q^\square\otimes_{R_Q} \T_{Q,\gm_Q}$.

\begin{lem} \label{triangle}
(i) Il y a un unique morphisme de $R_Q^\square$-alg\`ebres : $\alpha_Q:R_Q^\triangle \to \T_{Q,\gm_Q}^\square$.\\
(ii) Si $w\in S'$, alors l'application induite $R_w^\triangle \to \T_{Q,\gm_Q}^\square$ correspond \`a une paire $(\sigma_w,L_w)$ o\`u $\gFw$ agit sur $L_w$ par un caract\`ere $\eta_w\varepsilon$ tel que $\eta_w(\varpi_w) = T_w$.\\
(iii) Le morphisme $\alpha_Q$ est surjectif.
\end{lem}
\begin{proof}  
Supposons d'abord $\oE$ suffisamment grand pour contenir l'i\-mage du morphisme $\T_{Q,\gm_Q} \to \Qpbar$ pour chaque $\pi$ comme dans la proposition~\ref{CQmQ} (il n'y a qu'un nombre fini de telles $\pi$). Donc, pour chaque $\pi$, on a une d\'eformation correspondante $\sigma_\pi:\gF \to \GL_2(\oE)$ de $\rhobar$ non ramifi\'ee en dehors de $\Sigma\cup Q$. Soit $\oE^\square \= \oE\otimes_{R_Q}R_Q^\square$, on a alors pour chaque $w\in \Sigma$ un relev\'e canonique : 
$$\sigma_{\pi,w}^\square:\gFw \to \GL_2(\oE^\square)$$
de $\rhobar|_{\gFw}$ de la forme $A_w \sigma_\pi|_{\gFw} A_w^{-1}$ pour une matrice $A_w \in \GL_2(\oE^\square)$ relevant la matrice identit\'e de $\GL_2(k_E)$. Montrons que le morphisme correspondant $R_w^\square \to \oE^\square$ se factorise en un unique morphisme $R_w^\triangle \to \oE^\square$. Si $w \not\in \Sigma'$, alors $\sigma_\pi(I_w) = \rhobar(I_w)$, donc \'egalement $\sigma_{\pi,w}^\square(I_w) = \rhobar(I_w)$ et le fait que $R_w^\square \to \oE^\square$ se factorise par $R_w^\triangle$. Si $w\in \Sigma'$, alors $\sigma_\pi|_{\gFw} \cong \smat{\eta_w\varepsilon & * \\ 0 & \eta_w' }$ pour un caract\`ere $\eta_w$ relevant $\overline{\xi}_w$ et tel que $\eta_w(I_w) = \overline{\xi}_w(I_w)$. On en d\'eduit un unique facteur direct de $(\oE^\square)^2$ sur lequel $\gFw$ agit par $\eta_w\varepsilon$. De plus si $w\in S'$ alors $\eta_w(\varpi_w)$ est l'image de $T_w$ dans $\oE$. On en d\'eduit que $R_w^\square \to \oE^\square$ se factorise de mani\`ere unique en $R_w^\triangle \to \oE^\square$, et de plus que $\eta_w^\univ(\varpi_w)\in R_w^\triangle$ s'envoie vers l'image de $T_w$ dans $\oE$ pour $w\in S'$ ($\eta_w^\univ$ est le caract\`ere $\gFw \to R_w^\triangle$ relevant $\overline{\xi}_w$ dans le lemme \ref{lem_deformations}).

Il suit de la d\'efinition de $R_w^\triangle$ que pour chaque $\pi$ comme dans la proposition \ref{CQmQ} la fl\`eche $R_\loc^\square \to R_Q^\square \to \oE^\square$ se factorise de mani\`ere unique en un morphisme $R_\loc^\triangle \to \oE^\square$, et donc que la fl\`eche $R_Q^\square \to \oE^\square$ se factorise de mani\`ere unique en un morphisme $R_Q^\triangle \to \oE^\square$. On en d\'eduit que :
$$R_Q^\square \to \T_{Q,\gm_Q}^\square \hookrightarrow \oplus_{\pi}\oE^\square$$
se factorise de mani\`ere unique en un morphisme $R_Q^\triangle \to \oplus_{\pi}\oE^\square$. Comme, par le lemme~\ref{lem_deformations}, $R_Q^\triangle$ est topologiquement engendr\'e sur $R_Q^\square$ par $\{\eta_w^\univ(\varpi_w),\ w\in S'\}$, on voit que ce dernier morphisme est encore \`a valeurs dans $\T_{Q,\gm_Q}^\square$.

Maintenant ne faisons plus l'hypoth\`ese que $E$ est suffisamment grand. Si $E$ est remplac\'e par une extension $E'$, les $\oE$-alg\`ebres $\T_{Q,\gm_Q}^\square$, $R_Q^\square$ et $R_Q^\triangle$ sont remplac\'ees par leur extension des scalaires \`a ${\mathcal{O}}_{E'}$ (voir la remarque~\ref{rem_scalars}). Donc, pour $E'$ suffisamment grand, par le raisonnement pr\'ec\'edent le morphisme :
$${\mathcal{O}}_{E'}\otimes_{\oE}R_Q^\square \to {\mathcal{O}}_{E'}\otimes_{\oE}\T_{Q,\gm_Q}^\square$$
se factorise de mani\`ere unique en un morphisme ${\mathcal{O}}_{E'}\otimes_{\oE}R_Q^\triangle \to {\mathcal{O}}_{E'}\otimes_{\oE}\T_{Q,\gm_Q}^\square$ v\'erifiant le (ii) de l'\'enonc\'e. On en d\'eduit que l'image de $R_Q^\triangle$ est contenue dans $\T_{Q,\gm_Q}^\square$ ce qui d\'emontre (i) et (ii).

Enfin, pour d\'emontrer (iii), rappelons que si $w\not\in \Sigma\cup Q$ alors $T_w \in \T_{Q,\gm_Q}$ est l'image de ${\mathrm{N}}(w)\tr(\rho_Q^\univ(\Fr_w)) \in R_Q$. Si $w\in Q$, soit $g_w\in \gFw$ dont l'image dans $\gFw^\ab$ correspond \`a l'uniformisante $\varpi_w$. Le polyn\^ome caract\'eristique de $\rho_Q^\univ(g_w)$ a une unique racine dans $R_Q$ qui se r\'eduit sur ${\mathrm{N}}(w)^{-1}\alpha_w$, et cette racine s'envoie sur ${\mathrm{N}}(w)^{-1}T_w \in \T_{Q,\gm_Q}$. On en d\'eduit que $\T_{Q,\gm_Q}^\square$ est topologiquement engendr\'e sur $R_Q^\square$ par les \'el\'ements $T_w$ que l'on n'a pas consid\'er\'es, i.e. les $T_w$ pour $w\in S'$. Mais par (ii) ces \'el\'ements sont encore dans l'image de $\alpha_Q$, donc $\alpha_Q$ est surjectif.
\end{proof}

Notons que si $Q' \subseteq Q$, alors $R_{Q'}^\square$ s'identifie \`a $R_Q^\square \otimes_{R_{Q'}^\square} R_{Q'}$, et donc $\T_{Q',\gm_{Q'}}^\square$ s'identifie \`a $R_Q^\square \otimes_{R_Q} \T_{Q',\gm_{Q'}}$. De plus la surjection induite $\T_{Q,\gm_Q}^\square \to \T_{Q',\gm_{Q'}}^\square$ est compatible avec les applications de source $R_w^\triangle$ pour $w\in \Sigma$ de sorte que l'on en d\'eduit un diagramme commutatif o\`u toutes les fl\`eches sont surjectives :
$$\begin{array}{ccc}  
R_{Q}^\triangle & \longrightarrow & R_{Q'}^\triangle \\
\downarrow&      & \downarrow \\
\T_{Q,\gm_Q}^\square & \longrightarrow & \T_{Q',\gm_{Q'}}^\square.
\end{array}$$

On pose $C_{Q,\gm_Q}^\square \= R_Q^\square\otimes_{R_Q}C_{Q,\gm_Q}$ que l'on voit comme $R_Q^\triangle$-module via le morphisme $\alpha_Q$ du lemme \ref{triangle}.

Pour $w\in Q$, soit $\Delta_w$ le sous-groupe de $p$-Sylow de $k_w^\times$ et $\Delta_Q \= \prod_{w\in Q}\Delta_w$. L'argument de \cite[Lem.2.44]{DDT} montre que $\rho_Q^\univ|_{\gFw}$ est la somme de deux caract\`eres et l'on note $\eta_w^\univ$ celui qui se r\'eduit sur $\overline{\xi}_w$ o\`u $\overline{\xi}_w(\Fr_w) = \alpha_w' = {\mathrm N}(w)^{-1}$. Comme $\eta_w^\univ|_{I_w}$ a une r\'eduction triviale, donc est d'ordre une puissance de $p$, il se factorise par $I_w \to \oFw^\times \to k_w^\times$ (rappelons que $w\!\nmid \!p$). Si l'on restreint l'application induite $k_w^\times \to R_Q^\times$ \`a $\Delta_w$, on a un homomorphisme $\Delta_w \to R_Q^\times$ pour chaque $w\in Q$. En prenant le produit sur $w\in Q$ on obtient un homomorphisme $\Delta_Q \to R_Q^\times$, donc un homomorphisme de $\oE$-alg\`ebres $\oE[\Delta_Q] \to R_Q$ par lequel on voit $C_{Q,\gm_Q}$ comme un $\oE[\Delta_Q]$-module.

On peut aussi d\'efinir une action naturelle de $\Delta_w$ sur $C_{Q,\gm_Q}$ comme suit. On identifie $\Delta_w$ \`a un sous-groupe de $U_w^0/U_w$ via l'isomorphisme $U_w^0/U_w \buildrel\sim\over\to k_w^\times$ d\'efini par :
$$\begin{pmatrix} a & b \\ c & d \end{pmatrix}  \mapsto ad^{-1} \bmod \varpi_w\oFw.$$
Alors l'action naturelle de $U_w^0/U_w$ sur $H^1_\et(X_{V_Q,\Qbar},\oE)(1)$ commute avec celles de $G$ et $\tilde{\T}$ et donc induit une action sur $C_{Q,\gm_Q}$. La compatibilit\'e avec la correspondance de Langlands locale nous dit que $\Delta_w$ agit via $\eta_w$ sur $(\pi_w^{U_w})_{\alpha_w}$ pour chaque $\pi$ comme dans la proposition~\ref{CQmQ} (en utilisant les notations de sa preuve) ce qui entra\^\i ne que les deux d\'efinitions de l'action de $\Delta_w$ co\"\i ncident. Notons en particulier que, pour chaque repr\'esentation automorphe $\pi$ qui contribue \`a la d\'ecomposition de $\Qpbar\otimes_{\oE}C_{Q,\gm_Q}$, le facteur local $\pi_w$ est une s\'erie principale qui est non ramifi\'ee si et seulement si $U_w^0$ agit trivialement sur la composante correspondante de $\Qpbar\otimes_{\oE}C_{Q,\gm_Q}$.

Supposons maintenant $Q'\subseteq Q$. Soit $\tilde{\T}^0$ la sous-$\oE$-alg\`ebre de $\tilde{\T}$ engendr\'ee par les $T_w$ pour $w\not\in (\Sigma\backslash S')\cup (Q\backslash Q')$, $\T^0_{Q}$ (resp.~$\T^0_{Q'}$) l'image de $\tilde{\T}^0$ dans $\End_{\oE}(C_Q)$ (resp.~$\End_{\oE}(C_{Q'})$) et soit $\gm^0 \= \gm_Q \cap \tilde{\T}^0 = \gm_{Q'}\cap\tilde{\T}^0$. La preuve de la proposition \ref{CQmQ} est encore valable lorsque $\T_{Q',\gm_{Q'}}$ est remplac\'e par $\T^0_{Q',\gm^0}$ et montre que $\T^0_{Q',\gm^0}$ et $\T_{Q',\gm_{Q'}}$ ont m\^eme rang sur $\oE$, de m\^eme que $C_{Q',\gm^0}$ et $C_{Q',\gm_{Q'}}$. La construction de la fl\`eche $R_{Q'} \to \T_{Q',\gm_{Q'}}$ est aussi valable, et montre qu'elle se factorise par $\T^0_{Q',\gm^0}$. Comme $T_w$ est dans l'image de $R_{Q'}$ pour chaque $w\in Q\setminus Q'$, on en d\'eduit que l'application naturelle $\T^0_{Q',\gm^0} \to \T_{Q',\gm_{Q'}}$ est surjective, donc est un isomorphisme. Comme l'application naturelle $C_{Q',\gm^0} \to C_{Q',\gm_{Q'}}$ est automatiquement surjective, c'est aussi un isomorphisme. L'injection naturelle $C_{Q'} \to C_Q$ est $\tilde{\T}^0$-lin\'eaire, donc induit un morphisme $C_{Q',\gm^0} \to C_{Q,\gm^0}$, et on consid\`ere l'application compos\'ee :
$$\iota_Q^{Q'}:C_{Q',\gm_{Q'}} \cong C_{Q',\gm^0} \to C_{Q,\gm^0} \to C_{Q,\gm_Q}.$$
En tensorisant par $\Qpbar$ et en appliquant la d\'ecomposition de la proposition~\ref{CQmQ}, on voit que $\iota_Q^{Q'}$ est $\T_{Q,\gm_Q}$-lin\'eaire o\`u l'action de $\T_{Q,\gm_Q}$ sur $C_{Q',\gm_{Q'}}$ est d\'efinie via la surjection $\T_{Q,\gm_Q}\to \T_{Q',\gm_{Q'}}$ pr\'ec\'edente (envoyant $T_w$ sur $\tilde{\alpha}_w$ pour $w \in Q\backslash Q'$). On pose $\Delta_{Q\backslash Q'} \= \prod_{w\in Q\backslash Q'}\Delta_w$.

\begin{lem} \label{lem:DeltaQfree} 
(i) L'application $\iota_Q^{Q'}$ induit un isomorphisme $C_{Q',\gm_{Q'}} \!\!\buildrel\sim\over\to \!C_{Q,\gm_Q}^{\Delta_{Q\backslash Q'}}$\!\!.\\
(ii) Le module $C_{Q,\gm_Q}$ est libre sur $\oE[\Delta_Q]$. 
\end{lem}
\begin{proof} 
(i) Pour $w\in Q\backslash Q'$, soit $U_w'$ la pr\'eimage de $\Delta_w$ dans $U_w^0$ (on a donc $U_w\subseteq U'_w\subseteq U^0_w$). On pose :
$$U_Q\subseteq U_{Q}^{Q'} \= U_Q \!\!\!\prod_{w\in Q\backslash Q'} \!\!\!U_w'\subseteq U_{Q,0}^{Q'} \= U_Q \!\!\!\prod_{w\in Q\backslash Q'} \!\!\!U_w^0\subseteq U_{Q'}=U_Q \!\!\!\prod_{w\in Q\backslash Q'} \!\!\!\oDw^\times$$
$$V_Q\subseteq V_{Q}^{Q'} \= V_Q \!\!\!\prod_{w\in Q\backslash Q'} \!\!\!U_w'\subseteq V_{Q,0}^{Q'} \= V_Q \!\!\!\prod_{w\in Q\backslash Q'} \!\!\!U_w^0\subseteq V_{Q'}=V_Q \!\!\!\prod_{w\in Q\backslash Q'} \!\!\!\oDw^\times$$
ainsi que :
\begin{eqnarray*}
C_{Q}^{Q'} &\=& \Hom_{\oE[G]}\big(M,H^1_\et(X_{V_Q^{Q'},\Qbar},\oE)(1)\big)\\
\quad C_{Q,0}^{Q'} &\=& \Hom_{\oE[G]}\big(M,H^1_\et(X_{V_{Q,0}^{Q'},\Qbar},\oE)(1)\big).
\end{eqnarray*}
L'application $\iota_Q^{Q'}$ est alors la compos\'ee des applications :
$$C_{Q',\gm_{Q'}} \to C_{Q,0,\gm_Q}^{Q'} \to C_{Q,\gm_Q}^{Q'} \to C_{Q,\gm_Q}.$$
D'apr\`es la preuve de la proposition~\ref{CQmQ} (et la compatibilit\'e entre les deux descriptions de l'action de $\Delta_w$), on voit que les trois premiers $\oE$-modules ont m\^eme rang que le $\oE$-module $C_{Q,\gm_{Q}}^{\Delta_{Q\backslash Q'}}$.

(ii) Montrons que la fl\`eche $C_{Q',\gm_{Q'}} \to C_{Q,0,\gm_Q}^{Q'}$ est un isomorphisme. Par induction, il suffit de montrer que l'application :
$$ \delta: C_{Q'',0,\gm_{Q''}}^{Q'} \to C_{Q,0,\gm_{Q}}^{Q'},$$
avec $Q' \subset Q''$ et $Q = Q''\cup \{w\}$ pour une place $w\not\in Q'$ est un isomorphisme. Pour cela, comme les deux $\oE$-modules ont m\^eme rang, on voit qu'il suffit de d\'efinir une application $\oE$-lin\'eaire :
$$\epsilon:C_{Q,0,\gm_Q}^{Q'} \to C_{Q'',0,\gm_{Q''}}^{Q'}$$
telle que $\epsilon\circ\delta$ est bijectif. Soit $\pi$ le morphisme naturel $X_{V_{Q,0}^{Q'}} \to X_{V_{Q'',0}^{Q'}}$ ($\pi$ n'est pas une repr\'esentation automorphe ici !), $\delta$ provient donc de l'application : 
$$\pi^* : H^1_\et(X_{V_{Q'',0}^{Q'},\Qbar},\oE)(1) \to H^1_\et(X_{V_{Q,0}^{Q'},\Qbar},\oE)(1).$$
Soit $g_w \= \smat{ 1 & 0 \\ 0 & \varpi_w}$, on a $g_w^{-1} U_w^0 g_w \subset \oDw^\times$ et donc $g_w^{-1} V_{Q,0}^{Q'} g_w \subset V_{Q'',0}^{Q'}$. Soit $\pi'$ le morphisme associ\'e $X_{V_{Q,0}^{Q'}} \to X_{V_{Q'',0}^{Q'}}$. On ne va pas utiliser l'application $(\pi')^*$ sur la cohomologie, mais plut\^ot les deux applications trace :
$$\pi_*,\pi'_* : H^1_\et(X_{V_{Q,0}^{Q'},\Qbar},\oE)(1) \to H^1_\et(X_{V_{Q'',0}^{Q'},\Qbar},\oE)(1).$$
(L'application $\pi'_*$ peut \^etre vue comme associ\'ee \`a la double classe $\oDw^\times g_w^{-1} U_w$.) Les applications $\pi_*$ et $\pi'_*$ commutent avec l'action de $G$ et celle des op\'erateurs $T_{w'}$ pour $w'\not\in (\Sigma\cup\{w\})\backslash S'$. Un calcul standard de doubles classes donne la relation~:
$$\begin{pmatrix} \pi_* \\ \pi'_* \end{pmatrix} \circ T_w  = \begin{pmatrix} 0 & S_w{\mathrm{N}}(w) \\ -1 & T_w \end{pmatrix} \circ \begin{pmatrix} \pi_* \\ \pi'_* \end{pmatrix}$$
entre applications $H^1_\et(X_{V_{Q,0}^{Q'},\Qbar},\oE)(1) \to H^1_\et(X_{V_{Q'',0}^{Q'},\Qbar},\oE)(1)^2$, o\`u le $T_w$ \`a gauche est un endomorphisme de $H^1_\et(X_{V_{Q,0}^{Q'},\Qbar},\oE)(1)$, le $T_w$ \`a droite un endomorphisme de $H^1_\et(X_{V_{Q'',0}^{Q'},\Qbar},\oE)(1)$, et o\`u $S_w$ est d\'efini par l'action de $\varpi_w \in D_w^\times$. Notant encore $\pi_*$ et $\pi'_*$ les applications induites $C_{Q,0,\gm^0}^{Q'} \to C_{Q'',0,\gm^0}^{Q'} = C_{Q'',0,\gm_{Q''}}^{Q'}$, on en d\'eduit que $\pi_* - \tilde{\alpha}_w\circ \pi'_*$ se factorise par la localisation $C_{Q,0,\gm^0}^{Q'} \to C_{Q,0,\gm_Q}^{Q'}$ et on d\'efinit $\epsilon$ comme l'application induite $C_{Q,0,\gm_Q}^{Q'} \to C_{Q'',0,\gm_{Q''}}^{Q'}$. Comme $\pi_*\pi^* = {\mathrm{N}}(w) + 1$ et $\pi'_*\pi^* = S_w^{-1}T_w$, on en d\'eduit :
$$\epsilon\circ\delta = {\mathrm{N}}(w) + 1 - \psi(\varpi_w)^{-1}\tilde{\alpha}_wT_w = 1 - \psi(\varpi_w)^{-1}\tilde{\alpha}_w^2 = 1 - {\mathrm{N}}(w)\tilde{\alpha}_w\tilde{\beta}_w^{-1} \in \T_{Q'',\gm_{Q''}}$$
o\`u $\tilde{\beta}_w$ est l'autre racine de $X^2 - T_w X + \psi(\varpi_w){\mathrm{N}}(w)$. Par l'hypoth\`ese {\bf{(Q)}}, on voit que $1 - {\mathrm{N}}(w)\tilde{\alpha}_w\tilde{\beta}_w^{-1}$ a une r\'eduction non nulle, et donc est une unit\'e dans $\T_{Q'',\gm_{Q''}}$. Donc $\epsilon\circ\delta$ est un isomorphisme.

(iii) Montrons que $C_{Q,0,\gm_Q}^{Q'} \to C_{Q,\gm_Q}^{Q'}$ est un isomorphisme. L'application :
$$X_{V_Q^{Q'}} \to X_{V_{Q,0}^{Q'}}$$
\'etant de degr\'e premier \`a $p$, la composition avec l'application $C_{Q,\gm_Q}^{Q'} \to C_{Q,0,\gm_Q}^{Q'}$ induite par la trace est un automorphisme de $C_{Q,0,\gm_Q}^{Q'}$. Comme les deux $\oE$-modules $C_{Q,0,\gm_Q}^{Q'}$ et $C_{Q,\gm_Q}^{Q'}$ ont m\^eme rang, on en d\'eduit que $C_{Q,0,\gm_Q}^{Q'} \to C_{Q,\gm_Q}^{Q'}$ est un isomorphisme.

(iv) On introduit maintenant une place auxiliaire pour simplifier le reste de l'argument. Par \cite[Lem.3]{DT}, il y a un nombre infini de places finies $w$ v\'erifiant {\bf{(Q)}} telles que ${\mathrm{N}}(w) \not\equiv 1 \bmod p$ (le lemme de {\it loc.cit.} ne dit pas que les racines du polyn\^ome caract\'eristique sont distinctes, mais sa preuve montre que l'on peut toujours le supposer lorsque $p > 3$). On peut donc choisir une telle place $w_0 \not\in Q$ telle que $w_0$ ne divise aucun nombre premier $q$ v\'erifiant $[F(\sqrt[q]{1}):F] \le 2$. Comme $\Delta_{w_0}$ est trivial, on a :
$$C_{Q\cup\{w_0\},\gm_{Q\cup\{w_0\}}}^{Q'}= C_{Q\cup\{w_0\},\gm_{Q\cup\{w_0\}}}^{Q'\cup\{w_0\}}$$
ce qui montre d\'ej\`a (par ce qui pr\'ec\`ede) que les applications horizontales dans le diagramme :
$$\begin{array}{ccccc}
C_{Q',\gm_{Q'}} & \to & C_{Q,\gm_Q}^{Q'} & \to & C_{Q\cup\{w_0\},\gm_{Q\cup\{w_0\}}}^{Q'\cup\{w_0\}}\\
&     & \downarrow      &     & \downarrow \\
&     & C_{Q,\gm_Q}      & \to & C_{Q\cup\{w_0\},\gm_{Q\cup\{w_0\}}}
\end{array}$$
sont toutes des isomorphismes. On peut donc remplacer $Q'$ par $Q'\cup \{w_0\}$ et $Q$ par $Q\cup\{w_0\}$, i.e. supposer que $Q'$ contient $w_0$.

L'hypoth\`ese sur $w_0$ assure que, pour $x\in D_f^\times$, le groupe :
$$\Gamma_x  \= (D^\times \cap xU_Q^{Q'} x^{-1}\A_F^\times D_{\infty}^\times)/F^\times$$
n'a pas d'\'el\'ement non trivial d'ordre fini. (Pour voir cela, notons que si $\gamma F^\times \in \Gamma_x$ est d'ordre premier $q$, alors le quotient des racines du polyn\^ome caract\'eristique de $\gamma$ est une racine $q$-i\`eme de $1$ et $[F(\sqrt[q]{1}):F] \le 2$. Comme d'autre part $\gamma_{w_0} \in x_{w_0}U_{w_0}x_{w_0}^{-1}F_{w_0}^\times$, la r\'eduction de ce quotient est congrue \`a $1$ modulo $w_0$, ce qui implique que $w_0|q$ et est impossible par choix de $w_0$.) Cela entra\^\i ne que, dans l'action \`a droite du groupe $G \times \Delta_{Q\backslash Q'} \cong U_Q^{Q'}\A_{F,f}^\times / V_Q F^\times$ sur la courbe $X_{V_Q}$, les \'el\'ements non triviaux de $G \times \Delta_{Q\backslash Q'}$ ne fixent aucun point g\'eom\'etrique. En effet, il suffit de le v\'erifier pour l'action sur les points complexes de $D^\times\backslash ((\C\backslash\R) \times D_f^\times/V_Q)$. Si $u\in U_Q^{Q'}\A_{F,f}^\times$ fixe le point $D^\times(z,xV_Q)$ (avec $z\in\C\backslash\R,x\in D_f^\times$) alors $(z,xu)=(\gamma_\tau(z),\gamma_f x v)$ pour $\gamma\in D^\times$ et $v\in V_Q$ o\`u $\gamma_f$ est l'image de $\gamma$ dans $D_f^\times$ et $\gamma_\tau$ son image dans $D_\tau^\times \cong \GL_2(\R)$ (rappelons que $\tau$ est l'unique place archim\'edienne de $F$ o\`u $D$ est deploy\'ee). Comme $\gamma_f= xuv^{-1}x^{-1}$, on en d\'eduit que $\gamma F^\times \in \Gamma_x$, mais l'image de $\Gamma_x$ dans l'injection $D^\times/F^\times \hookrightarrow D_\tau^\times/F_\tau^\times \cong \PGL_2(\R)$ est discr\`ete et le stabilisateur de $z$ est compact, donc $\gamma F^\times$ est d'ordre fini, ce qui implique $\gamma \in F^\times$ et donc $u \in V_QF^\times$.

On a donc un diagramme commutatif de rev\^etements galoisiens \'etales de courbes connexes sur $F$:
$$\begin{array}{ccc}   
X_{V_Q}   & \to  & X_{V_Q^{Q'}} \\
\downarrow &      & \downarrow   \\
X_Q      & \to  & X_Q^{Q'},    
\end{array}$$
o\`u $X_Q^{Q'} \= X_{V_Q^{Q'}}/G$, $X_Q \= X_{V_Q}/G$, les applications horizontales ont comme groupe de Galois $\Delta_{Q\backslash Q'}$, les applications verticales $G$, et o\`u on fait agir \`a gauche les \'el\'ements de ces groupes via l'action \`a droite de leur inverse.

Soit ${\mathcal{F}}$ le $\oE$-faisceau lisse sur la courbe $X_Q^{Q'}$ associ\'e \`a l'action de son groupe fondamental $\pi_1(X_Q^{Q'},\overline{s})$ sur $\Hom_{\oE}(M,\oE(1))$ o\`u $\pi_1(X_Q^{Q'},\overline{s})$ agit sur $M$ via son quotient $G$ ($\overline s$ est un point g\'eom\'etrique). On note encore ${\mathcal{F}}$ son image inverse sur $X_{\Qbar}$ pour toutes les courbes $X$ du diagramme ci-dessus. Notons que $C_Q$ s'identifie alors \`a $H^1_{\et}(X_{V_Q,\Qbar},{\mathcal{F}})^G$, et la suite spectrale de Hochschild-Serre donne donc une suite exacte de $\oE[\gF]$-modules :
\begin{multline*}
0 \to H^1(G\times\Delta_{Q\backslash Q'},H^0_{\et}(X_{V_Q,\Qbar},{\mathcal{F}})) \to H^1_{\et}(X^{Q'}_{Q,\Qbar},{\mathcal{F}})  \to  \\
 C_Q^{\Delta_{Q\backslash Q'}} \to H^2(G\times \Delta_{Q\backslash Q'},H^0_{\et}(X_{V_Q,\Qbar},{\mathcal{F}})).
 \end{multline*}
Comme l'action de $\gF$ sur $H^0_{\et}(X_{V_Q,\Qbar},{\mathcal{F}})$ se factorise par $\gF^\ab$ et que $C_{Q,\gm_Q}$ est un r\'eseau dans une somme directe de repr\'esentations $\rho_\pi$ dont la r\'eduction $\rhobar$ est irr\'eductible, on en d\'eduit que l'application compos\'ee :
$$H^1_{\et}(X^{Q'}_{Q,\Qbar},{\mathcal{F}}) \to C_{Q}^{\Delta_{Q\backslash Q'}} \to C_{Q,\gm_Q}^{\Delta_{Q\backslash Q'}}$$
est surjective. Elle se factorise par $C_{Q,\gm_Q}^{Q'} \to C_{Q,\gm_Q}^{\Delta_{Q\backslash Q'}}$ qui est donc aussi surjective. Cela ach\`eve la preuve de (i) puisque ces $\oE$-modules ont le m\^eme rang.

(v) Comme $H^j_{\et}(X_{Q,\Qbar},{\mathcal{F}})$ et $H^j_{\et}(X_{Q,\Qbar}^{Q'},{\mathcal{F}})$ sont nuls pour $j > 2$ et comme l'action de $\gF$ se factorise par $\gF^\ab$ pour $j=0,2$, la suite spectrale de Hochschild-Serre appliqu\'ee au rev\^etement $X_{Q,\Qbar} \to X_{Q,\Qbar}^{Q'}$ montre que, pour $i > 0$, le $\oE[\gF]$-module $H^i(\Delta_{Q\backslash Q'},H^1_{\et}(X_{Q,\Qbar},{\mathcal{F}}))$ a une filtration finie pour laquelle $\gF$ agit via $\gF^\ab$ sur chaque gradu\'e. De plus la suite spectrale de Hochschild-Serre pour le rev\^etement $X_{V_Q,\Qbar} \to X_{Q,\Qbar}$ montre que l'action de $\gF$ sur les noyau et conoyau de $H^1_{\et}(X_{Q,\Qbar},{\mathcal{F}}) \to C_Q$ se factorise par $\gF^\ab$. On en d\'eduit que $H^i(\Delta_{Q\backslash Q'},C_Q)$ a une filtration finie pour laquelle $\gF$ agit via $\gF^\ab$ sur chaque gradu\'e. Comme $C_{Q,\gm_Q}$ est un facteur direct de $C_Q$ en tant que $\oE[\Delta_{Q\backslash Q'}\times \gF]$-module, ceci reste vrai avec $C_{Q,\gm_Q}$ au lieu de $C_Q$. Par ailleurs, par d\'evissage on voit que $H^i(\Delta_{Q\backslash Q'},C_{Q,\gm_Q})$ a une filtration finie par des sous-$\oE[\gF]$-modules tels que chaque gradu\'e est isomorphe \`a $\rhobar$. On en d\'eduit finalement que $H^i(\Delta_{Q\backslash Q'},C_{Q,\gm_Q}) = 0$ pour tout $i > 0$, ce qui implique que $C_{Q,\gm_Q}$ est libre sur $\oE[\Delta_{Q\backslash Q'}]$ par \cite[VI(8.7)]{Bro}.
\end{proof}

Tous les ingr\'edients sont maintenant en place pour appliquer le ``patching argument'' de Taylor-Wiles et d\'emontrer le r\'esultat principal de cette section. Lorsque $Q=\emptyset$, on omet l'indice $Q$.

\begin{thm}\label{thm:free} 
(i) Le $\T_\gm$-module $C_\gm$ est libre de rang $2$.\\
(ii) L'anneau local $\T_\gm$ est un anneau d'intersection compl\`ete sur $\oE$.
\end{thm}
\begin{proof}  
On le d\'emontre exactement comme dans \cite[Thm.3.4.11]{Ki}. Aux notations pr\`es, la seule diff\'erence est que l'anneau jouant le r\^ole de $B$ dans {\it loc.cit.}, c'est-\`a-dire $R_\loc^\triangle$, est ici formellement lisse sur $\oE$ et il n'y a donc pas besoin d'inverser $p$ dans la conclusion de \cite[Prop.3.3.1]{Ki} (et en fait la preuve devient plus simple, voir \cite[Thm.2.1]{Di1}). On conclut par cons\'equent que $R^\triangle \buildrel\sim\over \to \T_\gm^\square = R^\square\otimes_R\T_\gm$ est un isomorphisme d'anneaux locaux d'intersection compl\`ete sur $\oE$ et que $R^\square\otimes_R C_\gm$ est libre de rang $2$ sur $\T_\gm^\square$. Comme $R^\square$ est formellement lisse sur $\oE$, on en d\'eduit que $\T_\gm$ est un anneau local d'intersection compl\`ete sur $\oE$ et que $C_\gm$ est libre de rang $2$ sur $\T_\gm$.
\end{proof}

\subsection{R\'esultats principaux}\label{resprinc}

On montre les deux th\'eor\`emes \'enonc\'es dans l'introduction. 

On conserve les notations et hypoth\`eses des sections pr\'ec\'edentes. Rappelons ces hypoth\`eses. On fixe un corps totalement r\'eel $F$ et une repr\'esentation :
$$\rhobar:\gF \to \GL_2(k_E)$$
continue, modulaire, irr\'eductible en restriction \`a ${\Gal}(\Qbar/F(\sqrt[p]{1}))$ (o\`u $k_E$ est une extension finie de $\Fq$ suffisamment grande comme au \S\ \ref{globalprelim}), telle que l'image de $\rhobar({\Gal}(\Qbar/F(\sqrt[p]{1})))$ dans ${\rm PGL}_2(k_E)$ est non isomorphe \`a ${\rm PSL}_2(\F_5)$ si $p=5$, et v\'erifiant les hypoth\`eses :
\begin{enumerate}
\item[(i)]$\rhobar|_{\gFw}$ est r\'eductible non scalaire aux places $w$ de $F$ divisant $p$\\
\item[(ii)]il existe une place $v$ (fix\'ee) de $F$ divisant $p$ telle que $F_v$ est une extension non ramifi\'ee de $\Qp$ et $\rhobar|_{\gFv}$ est g\'en\'erique (r\'eductible).
\end{enumerate}
Rappelons que (ii) implique $p>3$. 

On fixe une alg\`ebre de quaternions $D$ sur $F$ v\'erifiant les hypoth\`eses :
\begin{enumerate}
\item[(iii)]$D$ est d\'eploy\'ee en une seule des places infinies et aux places divisant $p$\\
\item[(iv)]si $w$ est une place finie de $F$ o\`u $D$ est ramifi\'ee, alors $\rhobar|_{\gFw}$ est soit irr\'e\-ductible soit non scalaire de la forme $\overline\mu_w\smat{\omega & *\\ 0& 1}$ pour un caract\`ere $\overline\mu_w: \gFw \to k_E^\times$.
\end{enumerate}
Rappelons que (iv) implique $\pi_D(\rhobar)\neq 0$ par le corollaire \ref{nonzero}. 

On \'ecrit $\rhobar|_{\gFv}$ comme en (\ref{ecrit}) et on rappelle que $\tau(\emptyset)\in {\mathcal D}(\rhobar|_{\gFv})$ est l'unique poids de Serre tel que l'action de $U_v={\rm I}({\mathcal O}_{F_v})$ sur $\tau(\emptyset)^{{\rm I}_1({\mathcal O}_{F_v})}$ est donn\'ee par $\overline{M}_v=\overline\eta'(\emptyset)\otimes\overline\eta(\emptyset)=\overline{\xi}_v\vert_{[k_v^{\times}]}\otimes \overline{\xi}'_v\vert_{[k_v^{\times}]}$ (cf. \S~\ref{spec}). Rappelons aussi que l'on a d\'efini $Z(\rhobar|_{\gFv})$ en (\ref{zrho}) et $F(J)$ en (\ref{fj}) pour $J\subseteq {\mathcal S}_v$.

\begin{thm}\label{thm:multone} 
(i) On a $\dim_{k_E}\Hom_{k_E[\GL_2({\mathcal O}_{F_v})]}\big(\tau(\emptyset),\pi_{D,v}(\rhobar)\big)=1$.\\
(ii) Pour tout $J\subseteq {\mathcal S}_v$, on a :
$$\Hom_{k_E[\GL_2({\mathcal O}_{F_v})]}\big(\ind_{{\rm I}({\mathcal O}_{F_v})}^{\GL_2({\mathcal O}_{F_v})}\overline\eta'(J)\otimes\overline\eta(J),\pi_{D,v}(\rhobar)\big)\ne 0.$$
et si de plus $Z(\rhobar|_{\gFv})\cap F(J)= \emptyset$, alors cet espace d'homomorphismes est de dimension $1$ sur $k_E$.
\end{thm}
\begin{proof} 
(i) Par le th\'eor\`eme \ref{thm:free}(i), on a $C_\gm/\pE$ libre de rang $2$ sur $\T_\gm/\pE$. Par le th\'eor\`eme \ref{thm:free}(ii), l'anneau artinien $\T_\gm/\pE$ est d'intersection compl\`ete, donc Gorenstein. Par cons\'equent $(\T_\gm/\pE)[\gm]$ est de dimension $1$ sur $k_E$, et donc $(C_\gm/\pE)[\gm]$ de dimension $2$ sur $k_E$.

Montrons que l'application naturelle injective :
\begin{multline*}
C/\pE = k_E \otimes_{\oE} \Hom_{\oE[G]}\big(M,H^1_{\et}(X_{V,\Qbar},\oE)(1)\big)\hookrightarrow \\
\Hom_{k_E[G]}\big(\overline M, H^1_{\et}(X_{V,\Qbar},k_E)(1)\big)
\end{multline*}
est un isomorphisme apr\`es localisation en $\gm$ (o\`u $\overline M\=M/\pE$ et o\`u on omet $Q=\emptyset$ en indice). Posons $Q = \{w_0\}$ o\`u la place $w_0$ est comme dans la preuve du lemme~\ref{lem:DeltaQfree}, on a un isomorphisme $C_\gm \buildrel\sim\over\rightarrow C_{Q,\gm_Q}$ (lemme~\ref{lem:DeltaQfree}) et un isomorphisme d\'efini de mani\`ere analogue : 
$$\Hom_{k_E[G]}\big(\overline M, H^1_{\et}(X_{V,\Qbar},k_E)(1)\big)_\gm \buildrel\sim\over\longrightarrow \Hom_{k_E[G]}\big(\overline M, H^1_{\et}(X_{V_Q,\Qbar},k_E)(1)\big)_{\gm_Q}.$$
Consid\'erons le diagramme commutatif :
$$\begin{array}{rcl}
H^1_{\et}(X_{Q,\Qbar},{\mathcal{F}})  & \!\!\!\to  \!\!\!& H^1_{\et}(X_{Q,\Qbar},{\mathcal{F}}/\pE) \to   H^2_{\et}(X_{Q,\Qbar},{\mathcal{F}}) \\
\downarrow\qquad & & \qquad\downarrow   \\
\Hom_{\oE[G]}\big(M,H^1_{\et}(X_{V_Q,\Qbar},\oE)(1)\big) &\!\!\!\to \!\!\!& \Hom_{k_E[G]}\big(\overline M, H^1_{\et}(X_{V_Q,\Qbar},k_E)(1)\big)  \\
& &\qquad\downarrow  \\
& & \!\! H^2\big(G,\Hom_{k_E}\!\big(\overline M, H^0_{\et}(X_{V_Q,\Qbar},k_E)(1)\big)\big)
\end{array}$$
avec $X_Q$ et ${\mathcal{F}}$ comme dans la preuve du lemme~\ref{lem:DeltaQfree}. Comme la ligne du haut et la colonne de droite sont exactes et que l'action de $\gF$ sur $H^2_{\et}(X_{Q,\Qbar},{\mathcal{F}})$ et $H^0_{\et}(X_{V_{Q},\Qbar},k_E(1))$ se factorise par $\gF^\ab$, on a (comme dans la partie (v) de la preuve du lemme~\ref{lem:DeltaQfree}) que l'application compos\'ee :
$$H^1_{\et}(X_{Q,\Qbar},{\mathcal{F}}) \longrightarrow  H^1_{\et}(X_{Q,\Qbar},{\mathcal{F}}/\pE) \longrightarrow \Hom_{k_E[G]}\big(\overline M, H^1_{\et}(X_{V_Q,\Qbar},k_E)(1)\big)_{\gm_Q}$$
est surjective, et donc qu'il en est de m\^eme de :
\begin{multline*}
C_{Q,\gm_Q}=\Hom_{\oE[G]}\big(M,H^1_{\et}(X_{V_Q,\Qbar},\oE)(1)\big)_{\gm_Q}\longrightarrow \\
\Hom_{k_E[G]}\big(\overline M, H^1_{\et}(X_{V_Q,\Qbar},k_E)(1)\big)_{\gm_Q}.
\end{multline*}
On en d\'eduit avec ce qui pr\'ec\`ede que :
$$C_\gm/\pE \longrightarrow \Hom_{k_E[G]}\big(\overline M, H^1_{\et}(X_{V,\Qbar},k_E)(1)\big)_\gm$$
est un isomorphisme et donc que :
$$\dim_{k_E}\Hom_{k_E[G]}\big(\overline M, H^1_{\et}(X_{V,\Qbar},k_E)(1)\big)[\gm]=\dim_{k_E}(C_\gm/\pE)[\gm]=2.$$

Montrons maintenant que $\Hom_{k_E[U_v]}(\overline{M}_v,\pi_{D,v}(\rhobar))$ a dimension $1$. Par \cite[Lem.4.6]{BDJ} et \cite[Lem.4.10]{BDJ} (aux changements de convention pr\`es, cf. \S\ \ref{globalprelim}), on a l'isomorphisme d'\'evaluation (\cite[Lem.4.11]{BDJ}) :
$$\rhobar \otimes_{k_E} \pi_D(\rhobar)^{V}\buildrel\sim\over\to H^1_{\et}(X_{V,\Qbar},k_E)(1)[\gm_{\rhobar}^\Sigma]$$
o\`u $\gm_{\rhobar}^\Sigma$ est l'id\'eal de $k_E[T_w,S_w]_{w\not\in\Sigma}$ engendr\'e par les \'el\'ements $T_w - {\mathrm{N}}(w)\tr(\rhobar(\Fr_w))$ et $S_w - {\mathrm{N}}(w)\det(\rhobar(\Fr_w))$ pour $w\not\in\Sigma$. On a donc les identifications par (\ref{piDv}) :
$$\begin{array}{rcl} 
\rhobar\!\otimes_{k_E}\!\Hom_{k_E[U_v]}(\overline{M}_v,\pi_{D,v}(\rhobar)\!)\!\!&\!\!=\!\!&\!\!\rhobar\otimes_{k_E}\Hom_{k_E[U]}(\overline{M}_v \otimes_{k_E}\overline{M}^v,\pi_{D}(\rhobar))[\gm']\\
\!\!\!&\!\!=\!\!&\!\!\Hom_{k_E[U/V]}(\overline{M}_v \otimes_{k_E}\overline{M}^v,\rhobar\otimes_{k_E}\pi_{D}(\rhobar)^{V})[\gm']\\
\!\!\!&\!\!=\!\!&\!\!\Hom_{k_E[U/V]}\big(\overline M,H^1_{\et}(X_{V,\Qbar},k_E)(1)[\gm_{\rhobar}^\Sigma]\big)[\gm']\\
\!\!\!&\!\!=\!\!&\!\!\Hom_{k_E[G]}\big(\overline M,H^1_{\et}(X_{V,\Qbar},k_E)(1)\big)[\gm]
\end{array}$$
o\`u, dans la derni\`ere \'egalit\'e, on a identifi\'e l'action des op\'erateurs $S_w$ avec celle des \'el\'ements correspondants de $G = U\A_{F,f}^\times/VF^\times$. Comme on vient de montrer que ce dernier espace a dimension $2$, on en d\'eduit que $\Hom_{k_E[U_v]}(\overline{M}_v,\pi_{D,v}(\rhobar))$ a dimension $1$, ou encore par r\'eciprocit\'e de Frobenius :
$$\dim_{k_E}\Hom_{k_E[\GL_2({\mathcal O}_{F_v})]}\big(\ind_{{\rm I}({\mathcal O}_{F_v})}^{\GL_2({\mathcal O}_{F_v})}\overline\eta'(\emptyset) \otimes \overline\eta(\emptyset),\pi_{D,v}(\rhobar)\big)=1.$$
Par le lemme~\ref{unik}, le seul constituant de $\ind_{{\rm I}({\mathcal O}_{F_v})}^{\GL_2({\mathcal O}_{F_v})}\overline\eta'(\emptyset) \otimes \overline\eta(\emptyset)$ qui est dans ${\mathcal D}(\rhobar|_{\gFv})$ est son co-socle $\tau(\emptyset)$. Le th\'eor\`eme~\ref{cor:Geeplus} implique alors que les homomorphis\-mes :
$$\ind_{{\rm I}({\mathcal O}_{F_v})}^{\GL_2({\mathcal O}_{F_v})}\overline\eta'(\emptyset) \otimes \overline\eta(\emptyset)\longrightarrow \pi_{D,v}(\rhobar)$$
sont exactement ceux qui se factorisent par le co-socle $\tau(\emptyset)$. On en d\'eduit que $\Hom_{k_E[\GL_2({\mathcal O}_{F_v})]}(\tau(\emptyset),\pi_{D,v}(\rhobar))$ a aussi dimension $1$.

(ii) Par la proposition \ref{CQmQJ}, on a $C(J)_\gm \neq 0$ et l'injection :
 $$C(J)_\gm/\pE \hookrightarrow \Hom_{k_E[G]}\big(\overline{M(J)}, H^1_{\et}(X_{V,\Qbar},k_E)(1)\big)_\gm$$
implique donc en particulier $\Hom_{k_E[G]}\big(\overline{M(J)}, H^1_{\et}(X_{V,\Qbar},k_E)(1)\big)[\gm] \neq 0$. On montre comme dans le (i) ci-dessus l'identification :
$$\rhobar\otimes_{k_E}\Hom_{k_E[U_v]}\big(\overline{M_v(J)},\pi_{D,v}(\rhobar)\big)\cong \Hom_{k_E[G]}\big(\overline{M(J)},H^1_{\et}(X_{V,\Qbar},k_E)(1)\big)[\gm].$$
On voit donc que :
$$\Hom_{k_E[\GL_2({\mathcal O}_{F_v})]}\big(\ind_{{\rm I}({\mathcal O}_{F_v})}^{\GL_2({\mathcal O}_{F_v})}\overline{\eta}'(J)\otimes\overline{\eta}(J),\pi_{D,v}(\rhobar)\big)=\Hom_{k_E[U_v]}\big(\overline{M_v(J)},\pi_{D,v}(\rhobar)\big)$$
est non nul. La derni\`ere assertion d\'ecoule du lemme \ref{unik} et du th\'eor\`eme \ref{cor:Geeplus}.
\end{proof}

\begin{rem}\label{cheng}
{\rm Cheng dans \cite[Thm.4.1, Thm.5.3]{Ch} d\'emontre des r\'esultats similaires aux th\'eor\`emes ~\ref{thm:free} et~\ref{thm:multone}(i) ci-dessus, mais sous des hypoth\`eses techniques diff\'erentes. Les principales diff\'erences sont les suivantes~: d'une part l'hypoth\`ese que les repr\'esenta\-tions $\rhobar|_{\gFw}$ sont r\'eductibles non scalaires pour $w|p$ est remplac\'ee par $\End_{k_E[\gFw]}(\rhobar|_{\gFw}) = k_E$ et $F_w$ non ramifi\'ee, d'autre part tous les poids de Serre r\'eguliers de $\rhobar|_{\gFw}$ sont consid\'er\'es (\`a toutes les places $w|p$) au lieu du poids de Serre particulier $\tau(\emptyset)$ consid\'er\'e ici. Notons que la condition que $\rhobar|_{\gFw}$ a un poids de Serre r\'egulier au sens de \cite{Ch} est plus forte que la g\'en\'ericit\'e au sens de \cite[Def.11.7]{BP}. Les hypoth\`eses impos\'ees dans \cite{Ch} aux places $w$ ne divisant pas $p$ sont par ailleurs plus restrictives que celles de cet article.}
\end{rem}

\begin{rem}\label{egs}
{\rm Dans le tr\`es r\'ecent travail \cite{EGS}, Emerton, Gee and Savitt montrent une variante de \cite[Conj.B.1]{De} qui entra\^\i ne en particulier que le th\'eor\`eme~\ref{thm:multone}(i) est vrai en rempla\c cant $\tau(\emptyset)$ par un quelconque poids de Serre $\tau \in {\mathcal D}(\rhobar|_{\gFv})$, et aussi que l'espace des homomorphismes du th\'eor\`eme~\ref{thm:multone}(ii) est de dimension $1$ pour tout $J \subseteq {\mathcal S}_v$ (sous l'hypoth\`ese que $F$ est non ramifi\'ee en $p$ et que $\rhobar|_{\gFw}$ est g\'en\'erique pour tout $w|p$).}
\end{rem}

On en d\'eduit maintenant le corollaire satisfaisant suivant. 

\begin{cor}\label{ok}
Si la repr\'esentation $\pi_D(\rhobar)$ se factorise comme en (\ref{factorization}) (cf. \S\ \ref{globalprelim}), alors la repr\'esentation $\pi_{D,v}(\rhobar) \= \Hom_{k_E[U^v]}\big(\overline{M}^v,\pi_D(\rhobar)\big)[\gm']$ (c'est-\`a-dire d\'efinie comme en (\ref{piDv}), cf. \S\ \ref{facteur}) est n\'ecessairement le facteur local en $v$ de $\pi_D(\rhobar)$.
\end{cor}
\begin{proof}
Si  l'on a $\pi_D(\rhobar) \cong \otimes_w'\pi_{D,w}'(\rhobar)$ pour $\pi_{D,w}'(\rhobar)$ repr\'esentation lisse admissible de $D_w^\times$ sur $k_E$, alors $\pi_{D,v}(\rhobar) \cong  X \otimes_{k_E} \pi_{D,v}'(\rhobar)$ o\`u :
$$X \= \Hom_{k_E[U^v]}\big(\overline{M}^v,\otimes'_{w\neq v}\pi_{D,w}'(\rhobar)\big)[\gm']$$
(avec action triviale de $\GL_2(F_v)$). On en d\'eduit par le th\'eor\`eme \ref{thm:multone}(i) que le $k_E$-espace vectoriel~:
$$X \otimes_{k_E} \Hom_{k_E[\GL_2(\oFv)]}\big(\tau(\emptyset),\pi_{D,v}'(\rhobar)\big)=\Hom_{k_E[\GL_2({\mathcal O}_{F_v})]}\big(\tau(\emptyset),\pi_{D,v}(\rhobar)\big)$$
est de dimension $1$. Il en est donc de m\^eme du $k_E$-espace vectoriel $X$.
\end{proof}

\begin{cor}\label{xjok}
La repr\'esentation lisse admissible $\pi_{D,v}(\rhobar)$ v\'erifie les hypoth\`eses de la proposition \ref{xj} (pour la repr\'esentation locale $\rhobar|_{\gFv}$) pour tout $J\subseteq {\mathcal S}_v$ tel que $Z(\rhobar|_{\gFv})\cap F(J)= \emptyset$.
\end{cor}
\begin{proof}
L'hypoth\`ese (i) r\'esulte du th\'eor\`eme \ref{thm:multone}(i), l'hypoth\`ese (ii) du th\'eor\`eme \ref{cor:Geeplus} et du lemme \ref{unik}, et l'hypoth\`ese (iii) du th\'eor\`eme \ref{thm:multone}(ii) et de l'hypoth\`ese (ii).
\end{proof}

\begin{rem}\label{folk}
{\rm Il suit de \cite{EGS} que la repr\'esentation $\pi_{D,v}(\rhobar)$ v\'erifie les hypoth\`eses de la proposition \ref{xjbis} (pour $\rhobar|_{\gFv}$), ce qui contient strictement le corollaire \ref{xjok} (voir la remarque \ref{egs}).}
\end{rem}

Par le corollaire \ref{xjok} et la proposition \ref{xj}, pour tout $J\subseteq {\mathcal S}_v$ tel que $Z(\rhobar|_{\gFv})\cap F(J)= \emptyset$ la repr\'esentation $\pi_{D,v}(\rhobar)$ fournit donc un invariant $x(J)\in k_E^\times$. Pour une repr\'esentation (lisse avec caract\`ere central) arbitraire $\pi$ de $\G$ sur $k_E$ v\'erifiant les hypoth\`eses de la proposition \ref{xj}, ces $x(J)$ peuvent \^etre essentiellement quelconques (cf. \S\ \ref{spec}). 

\begin{cor}\label{mainglobal}
Pour $J\subseteq {\mathcal S}_v$ tel que $Z(\rhobar|_{\gFv})\cap F(J)= \emptyset$, les invariants $x(J)\in k_E^{\times}$ de la repr\'esentation $\pi_{D,v}(\rhobar)$ sont donn\'es par :
$$x(J)=-\overline{\xi}'_v(-1)\Big(\prod_{\sigma\in J}\alpha_{v,\sigma}\prod_{\sigma\notin J}\beta_{v,\sigma}\Big)^{-1}\frac{\displaystyle{\prod_{\stackrel{\sigma\in J}{ \sigma\circ\varphi^{-1}\notin J}}}\!\!x_{v,\sigma}(r_{v,\sigma}+1)}{\displaystyle{\prod_{\stackrel{\sigma\notin J}{ \sigma\circ\varphi^{-1}\in J}}}\!\!x_{v,\sigma}(r_{v,\sigma}+1)}$$
avec $(\alpha_{v,\sigma})_{\sigma\in {\mathcal S}_v}$, $(\beta_{v,\sigma})_{\sigma\in {\mathcal S}_v}$ et $(x_{v,\sigma})_{\sigma\in {\mathcal S}_v}$ comme en (\ref{norm}) pour le module de Fontaine-Laffaille contravariant de $\rhobar|_\gFv\otimes \theta_v^{-1}$ (cf. (\ref{ecrit})). 
\end{cor}
\begin{proof}
On va appliquer le corollaire \ref{mainlocal} \`a $J$ et $\pi=\pi_{D,v}(\rhobar)$, il faut donc v\'erifier toutes les conditions de ce corollaire. Notons qu'il suffit de d\'emontrer l'\'egalit\'e du corollaire dans n'importe quelle extension finie de $k_E$, et on peut donc agrandir arbitrairement $E$ dans l'application du corollaire \ref{mainlocal}, ce que l'on fait de mani\`ere tacite dans la suite de la preuve. 

Soit $\pi$ une repr\'esentation automorphe de $(D\otimes \A)^\times$ telle que $\rho_{\pi}$ contribue \`a $\Qpbar\otimes_{\oE}C(J)_\gm$ comme dans la proposition \ref{CQmQJ}. Comme dans la preuve de la proposition \ref{CQmQ} (pour $Q=\emptyset$), on voit que l'on a $\Hom_{\oE[U^v]}(M^v,\Qpbar\otimes_{\Qbar}\pi_f)\cong \Qpbar\otimes_{\Qbar}\pi_v$ o\`u $\pi_v$ est le facteur local de $\pi$ en $v$, o\`u $S$ et $U^v$ sont comme au \S\ \ref{facteur} et $M^v=\otimes_{w\in S}M_w$ avec $M_w$ comme au \S\ \ref{unI}. Soit $\pi^0$ un $\oE$-r\'eseau de $\pi_f$ comme en (\ref{pi0}), alors $\Hom_{\oE[U^v]}(M^v,\pi^0)$ est un $\oE$-r\'eseau $\pi_v^0$ de $\pi_v$ (stable par $\GFv$) et le $\oE$-module :
$$\Hom_{\oE[U]}(M(J),\pi^0)=\Hom_{\oE[U_v]}(M_v(J),\pi_v^0)\ne 0$$
est non nul, donc libre de rang $1$. On note $\widehat v\in \pi_v^0$ un g\'en\'erateur du sous-$\oE$-module (de rang $1$) des \'el\'ements de $\pi_v^0$ sur lesquels $U_v$ agit par le caract\`ere $M_v(J)=\eta'(J)\otimes\eta(J)$. 

L'injection naturelle $\overline\pi^0 \=k_E \otimes_{\oE} \pi^0 \hookrightarrow \pi_D(\rhobar)$ (lemme \ref{pi0bar}) induit une injection $\Hom_{k_E[U^v]}(\overline M^v,\overline\pi^0)\hookrightarrow \Hom_{k_E[U^v]}(\overline M^v,\pi_{D}(\rhobar))$ qui, par construction, tombe dans $\Hom_{k_E[U^v]}(\overline M^v,\pi_{D}(\rhobar))[\gm']=\pi_{D,v}(\rhobar)$. On a ainsi des injections :
$$k_E \otimes_{\oE}\pi_v^0=\Hom_{\oE[U^v]}(M^v,\pi^0)/\pE\hookrightarrow \Hom_{k_E[U^v]}(\overline M^v,\overline\pi^0)\hookrightarrow \pi_{D,v}(\rhobar)$$
de sorte que l'image de $\widehat v$ dans $\pi_{D,v}(\rhobar)$ est non nulle. 

Le fait que $\rho_{\pi}^0|_{\gFv}\cong {\rm Hom}_{{\rm Fil}^1,\varphi_1}(\M,\widehat A_{\rm cris})^{\vee}(1)$ (o\`u $\rho_\pi^0$ est un $\oE$-r\'eseau stable par $\gF$ dans $\rho_\pi$) pour un (unique) $\oE$-module fortement divisible $\M$ de type $\eta(J)\otimes\eta'(J)$ provient de $\Hom_{\oE[U_v]}(M_v(J),\pi_v)\ne 0$ et de la compatibilit\'e avec la correspondance de Langlands locale (\cite{Sai}). 

On peut donc appliquer le corollaire \ref{mainlocal} avec $J$, $\pi_{D,v}(\rhobar)$, $\M$, $\pi_v^0$ et $\widehat v$, ce qui donne le r\'esultat (le lecteur pourra v\'erifier sur toutes nos normalisations que la repr\'esentation $\pi_v$ est bien alors la repr\'esentation $\pi_p$ du th\'eor\`eme \ref{main}).
\end{proof}

\begin{rem}
{\rm (i) Notons que les valeurs des $x(J)$ du corollaire \ref{mainglobal} ne d\'ependent que de $\rhobar|_{\gFv}$. Lorsque la repr\'esentation $\pi_{D,v}(\rhobar)$ v\'erifie les hypoth\`eses de la proposition \ref{xjbis} pour $\rhobar|_{\gFv}$ (ce qui est maintenant connu par \cite{EGS}, cf. la remarque \ref{folk}) et que $Z(\rhobar|_{\gFv})\ne \emptyset$, il e\-xiste par les r\'esultats de \cite{BP} de nombreux autres invariants que l'on peut associer \`a $\pi_{D,v}(\rhobar)$, par exemple les invariants ``cycliques'' de \cite[Ques.9.5]{Br} lorsque $\rhobar|_{\gFv}$ est (g\'en\'erique) semi-simple. Nous ignorons comment calculer ces invariants cycliques en g\'en\'eral, mais on peut se demander si leurs valeurs conjecturales explicit\'ees dans \cite[Ques.9.5]{Br} ne seraient pas aussi celles donn\'ees par $\pi_{D,v}(\rhobar)$. Voir \`a ce propos les r\'esultats de \cite{Hu}.\\
(ii) Il est vraisemblable que le corollaire \ref{mainglobal} reste valable dans le cas o\`u les $r_{v,\sigma}$ sont tous nuls mais o\`u $\rhobar|_\gFv\otimes \theta_v^{-1}$ est suppos\'ee de plus dans la cat\'egorie de Fontaine-Laffaille.}
\end{rem}

\section{Appendice : R\'eductions de $K$-types}

On donne des formules explicites pour la semi-simplification modulo $p$ de tous les $K$-types pour $\GL_2(L)$ o\`u $L/\Qp$ est une extension finie {\it quelconque}. 

On note $M$ l'extension quadratique non ramifi\'ee de $L$, $\sigma$ l'\'el\'ement non trivial de $\Gal(M/L)$ et $\F_q$, $\F_{q^2}$ les corps r\'esiduels de $L$ et $M$. On note $I_L$ le sous-groupe d'inertie de $\g$, $\varpi_L$ une uniformisante de $L$ et $v_L$ la valuation sur $L$ telle que $v_L(\varpi_L)=1$.

On rappelle que $\II\subset \K$ est le pro-$p$-sous-groupe de Sylow du sous-groupe d'Iwahori $\I$ des matrices triangulaires sup\'erieures modulo une uniformisante $\varpi_L$ de $\oL$. On note ${\mathrm I}(\F_q)\subset \GL_2(\F_q)$ le sous-groupe des matrices triangulaires sup\'erieures et ${\mathrm I}_1(\F_q)$ son $p$-sous-groupe de Sylow. Pour $n\ge 1$, on note $I(n)$ le sous-groupe de $\I$ des matrices triangulaires sup\'erieures modulo $\varpi_L^n$ et $I_1(n)$ son pro-$p$-sous-groupe de Sylow (donc $I(1)=\I$ et $I_1(1)=\II$).

Pour un groupe profini $G$ contenant un pro-$p$-sous-groupe ouvert et pour $\Lambda \in \{E,k_E\}$, on note $R_\Lambda(G)$ le groupe de Grothendieck des repr\'esentations (lisses de dimension finie) de $G$ sur $\Lambda$. Rappelons que $R_{k_E}(G) \cong R_{k_E}(G/N)$ pour tout pro-$p$-sous-groupe normale $N$ de $G$, et qu'un \'el\'ement de $R_{k_E}(G)$ est d\'etermin\'e par son caract\`ere de Brauer, qui est une fonction \`a valeurs dans $\oE$ (pour $E$ suffisamment grand) sur l'ensemble des classes de conjugaison $p$-r\'eguli\`eres de $G$ (ou de $G/N$). Si $\vartheta \in R_E(G)$, on note $\overline\vartheta$ l'\'el\'ement dans $R_{k_E}(G)$ donn\'e par la r\'eduction de $\vartheta$. Rappelons que le caract\`ere de Brauer de $\overline{\vartheta}$ est juste la restriction du caract\`ere de $\vartheta$ \`a l'ensemble des \'el\'ements $p$-r\'eguliers de $G$. On suppose toujours que $E$ est suffisamment grand pour que tous les caract\`eres de Brauer prennent leurs valeurs dans $\oE$ et que toutes les repr\'esentations irr\'eductibles de $G$ en caract\'eristique $p$ soient d\'efinies sur $k_E$.

Pour un caract\`ere $\xi:\F_{q^2}^\times \to k_E^\times$, on note $\vartheta(\xi)$ l'image dans $R_{k_E}(\GL_2(\F_q))$ (ou dans $R_{k_E}(\K)$) de la r\'eduction de la repr\'esentation irr\'eductible de $\GL_2(\F_q)$ de degr\'e $q-1$ sur $E$ associ\'ee \`a $[\xi]$ si $\xi^q\ne\xi$ ou bien de $([\chi]\circ\det)\otimes({\rm Sp}-1) \in R_E(\GL_2(\F_q))$ si $\xi = \chi\circ{\mathrm N}_{\F_{q^2}/\F_q}$ (o\`u Sp d\'esigne la repr\'esentation sp\'eciale de $\GL_2(\F_q)$ sur $E$). Le caract\`ere de Brauer de $\vartheta(\xi)$ (sur les classes de conjugaison $p$-r\'eguli\`eres de $\GL_2(\F_q)$) envoie donc la classe de $g$ vers $(q-1)[\xi(c)]$ si $g = c\in \F_q^\times$, vers $- [\xi(c)]  - [\xi(c)]^q$ si $g$ a pour valeurs propres $c,c^q\in \F_{q^2}^\times \setminus \F_q^\times$, et vers $0$ sinon.

Si $[r,N]$ est un type de Weil-Deligne (cf. \S\ \ref{lifts0}), on d\'efinit son {\em conducteur essentiel} comme le conducteur minimal parmi tous ses tordus par des caract\`eres.

On note d'abord la formule suivante :

\begin{lem0}\label{app:sums}
Pour tout caract\`ere $\psi:\F_{q}^\times \to k_E^\times$, on a :
$$\sum_{\xi}\Theta(\xi) = \ind_{\F_q^\times {\mathrm I}_1(\F_q)}^{\GL_2(\F_q)} \psi = \sum_{\tau}m_\tau\tau$$
dans $R_{k_E}(\GL_2(\F_q))$, o\`u la premi\`ere somme est sur tous les caract\`eres $\xi:\F_{q^2}^\times \!\to k_E^\times$ tels que $\psi = \xi|_{\F_q^\times}$, o\`u la deuxi\`eme somme est sur toutes les repr\'esentations irr\'eductibles $\tau$ de $\GL_2(\F_q)$ sur $k_E$ de caract\`ere central $\psi$ et o\`u :
$$m_\tau \= \left\{\begin{array}{ll}
2^{\val(q) -\val(\dim\tau)}& \mbox{si $\dim\tau > 1$}\\
2^{\val(q)} - 1&\mbox{si $\dim\tau = 1$.}\end{array}\right.$$
\end{lem0}
\begin{proof}
Pour d\'emontrer la premi\`ere \'egalit\'e, on v\'erifie facilement que les deux sommes ont m\^eme caract\`ere de Brauer, qui envoie $g$ vers $(q^2-1)[\psi(g)]$ pour $g\in \F_q^\times$ et vers $0$ sinon. Il suffit de d\'emontrer la seconde \'egalit\'e en sommant sur tous les caract\`eres centraux possibles $\psi:\F_q^\times \to k_E^\times$, i.e. de d\'emontrer :
$$\sum_{\chi,\chi'} \ind_{{\mathrm I}(\F_q)}^{\GL_2(\F_q)}  \chi\otimes\chi'  = \sum_{\tau} m_\tau\tau$$
o\`u la premi\`ere somme est sur tous les caract\`eres $\chi,\chi': \F_q^\times \to k_E^\times$ et la deuxi\`eme somme est sur toutes les repr\'esentations irr\'eductibles $\tau$ de $\GL_2(\F_q)$ sur $k_E$. Comme les constituants irr\'eductibles de chaque $\ind_{{\mathrm I}(\F_q)}^{\GL_2(\F_q)}  \chi\otimes\chi'$ sont distincts, on doit montrer que $m_\tau = m_\tau'$ o\`u $m_\tau'$ est le nombre de paires ordonn\'ees $(\chi,\chi')$ telles que $\tau$ est un constituant de $\ind_{{\mathrm I}(\F_q)}^{\GL_2(\F_q)}  \chi\otimes\chi'$. Quitte \`a tordre par une puissance de $\det$, on peut supposer $\tau =  \otimes_{\sigma\in {\mathcal S}}(\Sym^{r_{\sigma}}k_E^2)^{\sigma}$ o\`u $\mathcal S$ est l'ensemble des plongements $\F_q \hookrightarrow k_E$. Par \cite[Prop.1.1]{Di2}, $m_\tau'$ est le nombre de sous-ensembles $J \subset {\mathcal S}$
tels que les \'equations :
$$\begin{array}{ccccccc}
r_{\sigma}&=&c_{\sigma}&\ {\rm si}&\ \sigma\in J&{\rm et}&\sigma\circ\varphi^{-1}\in J\\
r_{\sigma}&=&c_{\sigma}-1&\ {\rm si}&\ \sigma\in J&{\rm et}&\sigma\circ\varphi^{-1}\notin J\\
r_{\sigma}&=&p-2-c_{\sigma}&\ {\rm si}&\ \sigma\notin J&{\rm et}&\sigma\circ\varphi^{-1}\in J\\
r_{\sigma}&=&p-1-c_{\sigma}&\ {\rm si}&\ \sigma\notin J&{\rm et}&\sigma\circ\varphi^{-1}\notin J
\end{array}$$
aient une solution avec $0 \le c_\sigma \le p-1$ pour tout $\sigma \in {\mathcal S}$ et au moins un $c_\sigma < p - 1$. On trouve qu'il y a une solution si et seulement si $J$ est tel que :
\begin{enumerate}
\item[(i)]$\{\sigma,r_\sigma = p-1\} \cap F(J) = \emptyset$ (o\`u $F(J)$ est comme en (\ref{fj}))
\item[(ii)]si $r_\sigma = p - 1$ pour tout $\sigma \in {\mathcal S}$ alors $J = \emptyset$
\item[(iii)]si $r_\sigma = 0$ pour tout $\sigma \in {\mathcal S}$ alors $J \neq \emptyset$.
\end{enumerate}
En se souvenant que $\dim \tau = \prod_{\sigma \in {\mathcal S}} (r_\sigma + 1)$, on d\'eduit facilement $m_\tau = m_\tau'$.
\end{proof}

On fixe dans la suite un type de Weil-Deligne $[r,N]$.

\begin{prop0}\label{app:PS}
Supposons que $r$ est d\'ecomposable, i.e. $r|_{I_L} = \chi_1 \oplus \chi_2$ pour des caract\`eres lisses $\chi_i: \oL^\times \to E^\times$. Soit $n$ le conducteur de $\chi_1/\chi_2$ (i.e. le conducteur essentiel de $[r,0]$) et $\vartheta$ un $K$-type correspondant \`a $[r,0]$ (cf. \S\ \ref{lifts0}).\\
(i) Si $n=0$, alors $\overline{\vartheta} = \overline\chi\circ\det$ o\`u $\chi = \chi_1 = \chi_2$.\\
(ii) Si $n\ge 1$, alors :
$$\overline{\vartheta} = \left(\ind_{\I}^{\K} (\overline{\chi}_1\otimes\overline{\chi}_2)\right) + \frac{q^{n-1}-1}{q-1}\ind_{\oL^\times\II}^{\K} \psi$$
dans $R_{k_E}(\K)$, o\`u $\psi = \overline{\chi}_1\overline{\chi}_2$ est la r\'eduction du caract\`ere central de $\vartheta$.
\end{prop0}
\begin{proof}
Si $n = 0$ alors $\chi_1 = \chi_2$ et $\vartheta = \chi\circ\det$, l'assertion est alors claire. Si $n > 1$ et $q > 2$, alors il y a un unique $K$-type correspondant \`a $[r,0]$, \`a savoir $\vartheta = (\det\circ\chi_1)\otimes \ind_{I(n)}^{\K}\chi$ o\`u $\chi\left(\smat{a&b\\c&d}\right) \= (\chi_2/\chi_1)(d)$ (cf. \cite[\S~A.2.5]{He}). On a donc $\overline{\vartheta} = \ind_{I(n)}^{\K}\res_{\I}^{I(n)}\overline{\chi}_1\otimes\overline{\chi}_2$. Nous allons montrer la formule de l'\'enonc\'e :
\begin{equation}\label{PSind1}
\ind_{I(n)}^{\K}\res_{\I}^{I(n)}\overline{\chi}_1\otimes\overline{\chi}_2
  = \big(\ind_{\I}^{\K} (\overline{\chi}_1\otimes\overline{\chi}_2)\big) + \frac{q^{n-1}-1}{q-1}\ind_{\oL^\times\II}^{\K} \psi\end{equation}
par r\'ecurrence sur $n$. Le cas $n=1$ \'etant clair, on suppose $n > 1$ et (\ref{PSind1}) vraie avec $n$ remplac\'e par $n-1$. On montre d'abord que :
\begin{equation}\label{PSind2}
\ind_{I(n)}^{I(n-1)}\res_{\I}^{I(n)}\overline{\chi}_1\otimes\overline{\chi}_2 = \res_{\I}^{I(n-1)}\overline{\chi}_1\otimes\overline{\chi}_2 + \ind_{\oL^\times I_1(n-1)}^{I(n-1)}\res_{\I}^{\oL^\times I_1(n-1)}\overline{\chi}_1\otimes\overline{\chi}_2
\end{equation}
dans $R_{k_E}(I(n-1))$. Les classes de conjugaison $p$-r\'eguli\`eres de $I(n-1)$ sont celles des matrices $\smat{[a]&0\\0&[d]}$ pour $a,d\in\F_q^\times$. La valeur du caract\`ere de Brauer sur une telle matrice des deux c\^ot\'es de (\ref{PSind2}) est donn\'ee par $q\chi_1([a])\chi_2([d])$ si $a=d$ et par $\chi_1([a])\chi_2([d])$ si $a\neq d$. L'\'egalit\'e (\ref{PSind2}) et l'hypoth\`ese de r\'ecurrence donnent :
\begin{multline*}
\ind_{I(n)}^{\K}\res_{\I}^{I(n)}\overline{\chi}_1\otimes\overline{\chi}_2 = \\
\ind_{I(n-1)}^{\K}\res_{\I}^{I(n-1)}\overline{\chi}_1\otimes\overline{\chi}_2 + \ind_{\oL^\times I_1(n-1)}^{\K}\res_{\I}^{\oL^\times I_1(n-1)}\overline{\chi}_1\otimes\overline{\chi}_2=\\
 \big(\ind_{\I}^{\K} (\overline{\chi}_1\otimes\overline{\chi}_2)\big) + \frac{q^{n-2}-1}{q-1}\ind_{\oL^\times\II}^{\K} \psi\\
+ \ind_{\oL^\times\II}^{\K}\ind_{\oL^\times I_1(n-1)}^{\oL^\times\II}\res_{\oL^\times\II}^{\oL^\times I_1(n-1)} \psi.
\end{multline*}
Comme les \'el\'ements $p$-r\'eguliers de $\oL^\times\II$ sont centraux et comme $[\oL^\times\II:\oL^\times I_1(n-1)] = q^{n-2}$, on a dans $R_{k_E}(\oL^\times\II)$ :
$$\ind_{\oL^\times I_1(n-1)}^{\oL^\times\II}\res_{\oL^\times\II}^{\oL^\times I_1(n-1)} \psi= q^{n-2}\psi$$
ce qui donne (\ref{PSind1}). Si $n>1$ et $q=2$, il y a deux $K$-types associ\'es \`a $[r,0]$, l'un donn\'e par la m\^eme formule que dans le cas $q>2$, l'autre par son compl\'ement dans $\det\circ\chi_1\otimes \ind_{I(n+1)}^{\K}\res_{I(n)}^{I(n+1)}\chi$ (cf. \cite[\S\ A.2.7]{He}). Notons que si $q = 2$ alors $I(n) = I_1(n)$ et $\overline{\chi}_1$, $\overline{\chi}_2$, ${\psi}$ sont tous triviaux. Le m\^eme calcul que dans le cas $q > 2$ montre que la r\'eduction du premier $K$-type est $2^{n-1}\ind_{\I}^{\K}{\psi}$ (comme demand\'e). De plus la r\'eduction de $\det\circ\chi_1\otimes \ind_{I(n+1)}^{\K}\res_{I(n)}^{I(n+1)}\chi$ est $2^n\ind_{\I}^{\K}{\psi}$. On en d\'eduit que les deux $K$-types ont m\^eme r\'eduction.
\end{proof}

On rappelle que si $r$ est irr\'eductible, alors son conducteur essentiel est pair si et seulement si $r$ est l'induite d'un caract\`ere du groupe de Weil de $M$ (voir par exemple \cite[Cor.4.1.9]{Ku}). Dans ce cas, les r\'eductions des $K$-types sont aussi consid\'er\'ees dans \cite[Thm.2.3]{Sc}, qui donne aussi la formule r\'ecursive (\ref{SCind1}) de la d\'emonstration ci-dessous.

\begin{prop0}\label{app:SCeGL2}
Supposons que $r$ est irr\'eductible de conducteur essentiel pair $n=2m$ (avec $m\ge 1$), de sorte que $r|_{I_L} = \xi\oplus\xi^\sigma$ pour un caract\`ere lisse $\xi:\oM^\times \to E^\times$ tel que $\xi/\xi^\sigma$ est de conducteur $m$. Soit $\vartheta$ le $K$-type correspondant \`a $[r,N]=[r,0]$. On a :
$$\overline{\vartheta} = (-1)^{m-1}\Theta(\overline{\xi}) + \frac{q^{m-1} - (-1)^{m-1}
}{q+1}\ind_{\oL^\times\II}^{\K}\psi$$
dans $R_{k_E}(\K)$, o\`u $\psi = \overline{\xi}|_{\oL^\times}$ est la r\'eduction du caract\`ere central de $\vartheta$.
\end{prop0}
\begin{proof}
Il y a dans ce cas un unique type $\vartheta$ associ\'e \`a $[r,0]$ dont nous rappelons la d\'efinition suivant \cite[\S~3]{Ger} (voir \cite[\S\ A.3]{He} pour son unicit\'e). On identifie ${\mathrm M}_2(L)$ avec $\End_L(M) = M \oplus M\sigma$ et ${\mathrm M}_2(\oL)$ avec $\oM \oplus \oM\sigma$. On pose $U_0 \= \K$ et on d\'efinit des sous-groupes ouverts compacts de $U_0$ par :
$$U_t \= \{ a + b\sigma,  a \in \oM^\times, b \in \varpi_L^t\oM\}$$
pour $t > 0$. Quitte \`a tordre, on peut supposer que $r$ est lui-m\^eme de conducteur $n$. Donc pour $t \ge m/2$, la formule $\beta_\xi(a+b\sigma) = \xi(a)$ d\'efinit un caract\`ere de $U_t$. Alors $\vartheta = \vartheta_\xi = \ind_{U_{[m/2]}}^{\K} \alpha_\xi$ o\`u $\alpha_\xi$ est la restriction \`a $U_{[m/2]}$ de la repr\'esentation $\kappa_\xi$ d\'efinie en \cite[\S\ 3 3]{Ger}.  En particulier, les repr\'esentations $\alpha_\mu$ pour $\mu\neq \mu^\sigma$ sont d\'etermin\'ees par les formules :
\begin{enumerate}
\item[(i)]$\alpha_\mu \otimes {\rm Sp}  = \ind_{U_1}^{\K} \beta_\mu$ si $m=1$
\item[(ii)]$\alpha_\mu = \beta_\mu$ si $m$ est pair
\item[(iii)]$\ind_{U_{(m+1)/2}}^{U_{(m-1)/2}} \beta_\mu = \oplus_{\omega\neq 1} \alpha_{\mu\omega}$ si $m>1$ est impair, o\`u la somme est sur les caract\`eres non triviaux $\omega:\F_{q^2}^\times/\F_q^\times \to E^\times$.
\end{enumerate}
Pour voir l'unicit\'e des $\alpha_\mu$ dans le cas (iii), on note que
$$q\alpha_\mu = (1-q)\ind_{U_{(m+1)/2}}^{U_{(m-1)/2}} \beta_\mu + \sum_{\omega\neq 1} \ind_{U_{(m+1)/2}}^{U_{(m-1)/2}} \beta_\mu\omega$$
dans $R_E(U_{(m-1)/2})$.

On d\'emontre maintenant la proposition par r\'ecurrence sur $m$. Si $m=1$, alors $\alpha_\xi$ est pr\'ecis\'ement la repr\'esentation irr\'eductible de $\GL_2(\F_q)$ associ\'ee \`a $\xi$ et la proposition est imm\'ediate dans ce cas. Si $m=2$, alors le caract\`ere de Brauer de $\overline{\vartheta} = \ind_{U_1}^{U_0}\overline{\beta}_\xi$ vu comme repr\'esentation de $\GL_2(\F_q)$ envoie la classe de conjugaison d'un \'el\'ement $p$-r\'egulier $g$ vers $(q-1)q[\xi(c)]$ si $g = c\in \F_q^\times$, vers $[\xi(c)]  + [\xi^\sigma(c)]$ si $g$ a pour valeurs propres
$c,\sigma(c)\in \F_{q^2}^\times \setminus \F_q^\times$, et vers $0$ sinon. On en d\'eduit que $\overline{\vartheta} + \Theta(\overline{\xi})  = \ind_{\oL^\times\II}^{\K}\psi$ (voir la preuve du lemme \ref{app:sums}). Supposons maintenant $m > 2$ et soit $\xi':\oM^\times \to E^\times$ tel que $\xi'$ a m\^eme r\'eduction que $\xi$ mais avec $\xi'/(\xi')^\sigma$ de conducteur $m-1$. Nous allons montrer que :
\begin{equation} \label{SCind1}
\overline{\vartheta}_\xi = \sum_{\omega\neq 1}\overline{\vartheta}_{\xi'\omega}
\end{equation}
o\`u la somme est sur tous les caract\`eres non triviaux $\omega:\F_{q^2}^\times/\F_q^\times \to E^\times$. Si $m > 2$ est pair, alors $\overline{\beta}_\xi = \overline{\beta}_{\xi'}$ et les formules d\'efinissant les repr\'esentations ci-dessus donnent :
\begin{eqnarray*} 
\overline{\vartheta}_\xi &=&
\ind_{U_{(m-2)/2}}^{\K}\ind_{U_{m/2}}^{U_{(m-2)/2}}\overline{\beta}_\xi\\
&=& \ind_{U_{(m-2)/2}}^{\K}\ind_{U_{m/2}}^{U_{(m-2)/2}}\overline{\beta}_{\xi'}
= \sum_{\omega\neq 1} \ind_{U_{(m-2)/2}}^{\K}\overline{\alpha}_{\xi'\omega}
= \sum_{\omega\neq 1} \overline{\vartheta}_{\xi'\omega}
\end{eqnarray*}
comme demand\'e. Si $m > 2$ est impair, alors (\ref{SCind1}) d\'ecoulera de la formule $\overline{\alpha}_\xi = \sum_{\omega\neq 1}
\overline{\beta}_{\xi'\omega}$ dans $R_{k_E}(U_{(m-1)/2})$. Les caract\`eres de Brauer des $\overline{\alpha}_\mu$ pour $\mu \in \{\xi\omega, \omega:\F_{q^2}^\times/\F_q^\times \to E^\times\}$ sont d\'etermin\'es comme ci-dessus par :
$$\ind_{U_{(m+1)/2}}^{U_{(m-1)/2}} \overline{\beta}_\mu = \sum_{\omega\neq 1} \overline{\alpha}_{\mu\omega}.$$
Comme $\overline{\beta}_\mu = \res_{U_{(m-1)/2}}^{U_{(m+1)/2}} \overline{\beta}_{\mu'}$ (o\`u $\mu' = \xi'\omega$ si $\mu = \xi\omega$), il suffit de montrer :
\begin{equation}\label{SCind2}
\ind_{U_{(m+1)/2}}^{U_{(m-1)/2}} \res_{U_{(m-1)/2}}^{U_{(m+1)/2}} \overline{\beta}_{\mu'}
 = \sum_{\omega,\eta \neq 1} \overline{\beta}_{\mu'\omega\eta}
 \end{equation}
pour $\mu' \in \{\xi'\omega', \omega':\F_{q^2}^\times/\F_q^\times \to E^\times\}$. Notons que le terme de droite dans (\ref{SCind2}) est juste $q\overline{\beta}_\xi + (q-1)\sum_{\omega\neq 1} \overline{\beta}_{\xi\omega}$ et que les classes de conjugaison $p$-r\'eguli\`eres dans $U_{(m-1)/2}$ sont pr\'ecis\'ement celles des \'el\'ements $[c]$ pour $c \in \F_{q^2}^\times$. Le caract\`ere de Brauer du terme de droite dans (\ref{SCind2}) envoie un tel \'el\'ement vers $q(q-1)\xi([c])$ si $c\in \F_q^\times$ et vers $\xi([c])$ sinon. En utilisant le fait que, si $c \not\in\F_q^\times$ et $g \in U_{(m-1)/2} \setminus U_{(m+1)/2}$, alors $g[c]g^{-1} \not \in U_{(m+1)/2}$, on trouve que le caract\`ere de Brauer du terme de gauche dans (\ref{SCind2}) est le m\^eme. Cela termine la preuve de (\ref{SCind1}). En appliquant (\ref{SCind1}), l'hypoth\`ese de r\'ecurrence et le lemme \ref{app:sums} donnent alors :
\begin{eqnarray*}
\overline{\vartheta}_\xi &=&  \sum_{\omega\neq 1} \overline{\vartheta}_{\xi'\omega}\\
&=&  (-1)^{m-2} \sum_{\omega\neq 1} \Theta(\overline{\xi\omega})
          + q\cdot \frac{q^{m-2} - (-1)^{m-2}}{q+1}\ind_{\oL^\times\II}^{\K}\psi \\
 &=& (-1)^{m-1}\Theta(\overline{\xi}) + (-1)^{m-2}\ind_{\oL^\times\II}^{\K}\psi
              + \frac{q^{m-1} - (-1)^{m-2}q}{q+1}\ind_{\oL^\times\II}^{\K}\psi \\
&=& (-1)^{m-1}\Theta(\overline{\xi}) + \frac{q^{m-1} - (-1)^{m-1}}{q+1}\ind_{\oL^\times\II}^{\K}\psi.
\end{eqnarray*}
\end{proof}

\begin{prop0}\label{app:SCoGL2}
Supposons que $r$ est irr\'eductible de conducteur essentiel impair $n=2m+1$ (avec $m\ge 1$). Soit $\vartheta$ le $K$-type correspondant \`a $[r,N]=[r,0]$. On a~:
$$\overline{\vartheta} = q^{m-1}\ind_{\oL^\times\II}^{\K} \psi$$
dans $R_{k_E}(\K)$, o\`u $\psi$ est la r\'eduction du caract\`ere central de $\vartheta$.
\end{prop0}
\begin{proof} 
Il y a dans ce cas un unique type $\vartheta$ associ\'e \`a $[r,N]$ dont nous rappelons la d\'efinition suivant \cite[\S~5]{GK} (voir \cite[\S\ 3]{He} pour son unicit\'e). On identifie d'abord ${\mathrm M}_2(L)$ avec $\End_{L}(M')$ et ${\mathrm M}_2(\oL)$ avec $\End_{\oL}({\mathcal O}_{M'})$ pour une extension quadratique {\it ramifi\'ee} $M'$ de $L$ (qui d\'epend de $r$) en chosissant $\{1,\varpi_{M'}\}$ comme base de $M'$ sur $L$ (o\`u $\varpi_{M'}$ est une uniformisante de $M'$). Pour $t\ge 1$, on d\'efinit des sous-groupes ouverts compacts de $\K$ par :
$$V_t \= \{ g\in \K, g(\beta) - \beta \in \beta\varpi_{M'}^t{\mathcal O}_{M'} \mbox{\ si $\beta \in{\mathcal O}_{M'}$} \}$$
de sorte que $V_1 = \II$, $[\II:V_t] = q^{2(t-1)}$ et $ {\mathcal O}_{M'} \cap V_t = 1 + \varpi_{M'}^t {\mathcal O}_{M'} $, o\`u l'on identifie $\beta \in {\mathcal O}_{M'}$ avec l'endomorphisme de ${\mathcal O}_{M'}$ d\'efini par la multiplication par $\beta$. La seule chose que l'on a besoin de savoir sur $\vartheta$ est qu'il est de la forme $\ind_{{\mathcal O}_{M'}V_m}^{\K}\alpha$ pour un caract\`ere $\alpha: {\mathcal O}_{M'}V_m \to E^\times$. Comme $[\oL^\times\II: {\mathcal O}_{M'}V_m] = q^{m-1}$ et que les classes de conjugaison $p$-r\'eguli\`eres de ${{\mathcal O}_{M'}V_m}$ et $\oL^\times\II$ sont celles des \'el\'ements $[a]$ pour $a\in \F_q^\times$, on voit que le caract\`ere de Brauer de $\ind_{{\mathcal O}_{M'}V_m}^{\oL^\times\II}\overline{\alpha}$ est celui de $q^{m-1}\psi$. Ceci ach\`eve la preuve.
\end{proof}

Soit maintenant $D$ une alg\`ebre de quaternions sur $L$, $\oD$ un ordre maximal dans $D$, $K \= \oD^\times$, $Z\=\oL^\times$ et $I_1$ le pro-$p$ sous-groupe de Sylow de $K$. On note $\Pi_D$ une uniformisante de $D$. On a $I_1 = 1 + \Pi_D\oD$ et on pose $I_n \= 1 + \Pi_D^n\oD$ pour $n\ge 1$. On choisit un plongement $M \hookrightarrow D$ ce qui permet d'identifier les repr\'esentations irr\'eductibles de 
$K$ sur $k_E$ avec les caract\`eres de $K/I_1 \cong \F_{q^2}^\times$. Notons que l'analogue du lemme \ref{app:sums} dans ce contexte dit juste que $\ind_{ZI_1}^K\psi$ est la somme des $q+1$ caract\`eres $\xi$ de caract\`ere central $\psi$.

On fixe un type de Weil-Deligne $[r,N]$.

\begin{prop0}\label{app:special}
Supposons que $N\neq 0$, de sorte que $r|_{I_L} = \chi \oplus \chi$ pour un caract\`ere lisse $\chi: \oL^\times \to E^\times$. Soit $\vartheta$ le $K$-type correspondant \`a $[r,N]$. On a :
$$\overline{\vartheta} = \overline{\chi}\circ\det.$$
\end{prop0}
\begin{proof}
C'est clair puisque $\vartheta = \chi\circ \det$.
\end{proof}

\begin{prop0}\label{app:SCeD}
Supposons que $r$ est irr\'eductible de conducteur essentiel pair $n=2m$ (avec $m\ge 1$), de sorte que $r|_{I_L} = \xi\oplus\xi^\sigma$ pour un caract\`ere lisse $\xi:\oM^\times \to E^\times$ tel que $\xi/\xi^\sigma$ est de conducteur $m$. Soit $\vartheta$ un $K$-type correspondant \`a $[r,N]=[r,0]$. On a :
$$\overline{\vartheta} = (-1)^{m-1}\overline{\mu} + \frac{q^{m-1} - (-1)^{m-1}}{q+1}\ind_{ZI_1}^K\psi$$
dans $R_{k_E}(K)$, o\`u $\mu \in \{\xi,\xi^\sigma\}$ et $\psi = \overline{\xi}|_{\oL^\times}$ est la r\'eduction du caract\`ere central de $\vartheta$.
\end{prop0}
\begin{proof}
Dans ce cas la construction de \cite[\S\ 5]{Ger} montre qu'il y a deux types associ\'es \`a $r$, chacun d\'etermin\'e par un choix de $\mu\in\{\xi,\xi^\sigma\}$. La d\'efinition des types est analogue \`a celle pour $\K$, mais avec le r\^ole des parit\'es invers\'e. En particulier, si $m$ est impair, alors $\vartheta$ est induit d'un caract\`ere $\alpha_\mu:\oM^\times I_{m} \to E^\times$, et si $m$ est pair, alors $\vartheta$ est induit d'une repr\'esentation de dimension $q$ de $\oM^\times I_{m-1}$. On omet la preuve de la proposition, qui est tr\`es similaire \`a celle de la proposition \ref{app:SCeGL2}.
\end{proof}

\begin{prop0}\label{app:SCoD} 
Supposons que $r$ est irr\'eductible de conducteur essentiel impair $n=2m+1$ (avec $m\ge 1$). Soit $\vartheta$ le $K$-type correspondant \`a $[r,N]=[r,0]$. On a :
$$\overline{\vartheta} = q^{m-1}\ind_{ZI_1}^K \psi$$
dans $R_{k_E}(K)$, o\`u $\psi$ est la r\'eduction du caract\`ere central de $\vartheta$.
\end{prop0}
\begin{proof} 
Ici encore on omet la preuve car elle est tr\`es similaire \`a celle de la proposition \ref{app:SCoGL2} en utilisant \cite[\S~54]{BH} pour d\'eterminer le type, qui est unique dans ce cas et est induit d'un caract\`ere de ${\mathcal O}_{M'}^\times I_m$ pour une extension quadratique ramifi\'ee $M'$ de $L$ plong\'ee dans $D$.
\end{proof}

\begin{rem0}
{\rm Dans ce cas des alg\`ebres de quaternions, des r\'esultats si\-milaires (r\'edig\'es avec plus de d\'etails) ont \'et\'e aussi r\'ecemment obtenus par Tokimoto (\cite{To}).}
\end{rem0}

\end{document}